\documentclass[12pt]{amsart}
\usepackage{amsmath,bbm,amsfonts,dsfont,amssymb,amsxtra,amscd,enumerate,amsthm,yfonts,verbatim}
\usepackage{color,xcolor}
\usepackage{latexsym}
\usepackage{epsfig}
\usepackage{fullpage}
\usepackage{hyperref}
\usepackage{comment}
\usepackage{enumitem}
\usepackage{esint}

\newcommand{\e}{\ensuremath{\mathbf{e}}}

\newcommand\les{\lesssim}
\newcommand\ges{\gtrsim}

\newcommand\R{\mathbb{R}}
\newcommand\C{\mathbb{C}}
\newcommand\Z{\mathbb{Z}}
\newcommand\N{\mathbb{N}}
\renewcommand\S{\mathbb{S}}

\newcommand{\calQ}{\mathcal Q}

\newcommand\la{\langle}
\newcommand\ra{\rangle}

\newtheorem{theorem}{Theorem}[section]
\newtheorem{lemma}[theorem]{Lemma}
\newtheorem{proposition}[theorem]{Proposition}
\newtheorem{corollary}[theorem]{Corollary}
\newtheorem{definition}[theorem]{Definition}

\newtheorem{r1}{Remark}

\numberwithin{equation}{section}

\newcommand{\Ec}{\tilde{E}}

\usepackage{mathtools}
\mathtoolsset{showonlyrefs}

\begin{document}

\title[A resolution space ]{A new resolution space for nonlinear Schr\"odinger equations and applications}

\author{Ioan Bejenaru }
\address{Department of Mathematics, University of California, San Diego}
\email{ibejenaru@ucsd.edu}

\begin{abstract}  
Resolution spaces play a central role in constructing solutions for nonlinear partial differential equations. 
One of the main goals in the area of nonlinear dispersive PDEs has been to construct effective resolution spaces which capture the known bilinear restrictions estimates for free solutions. In this paper we propose a new structure for the Schr\"odinger equation which effectively replicates the classical bilinear $L^2_{t,x}$ estimate. In addition, the new structure 
has the property that its "dual" is an effective candidate for a space for the forcing in the linear inhomogeneous Schr\"odinger equation, a feature that has been elusive so far in the literature. As an application, we show how these structures can recover the known global well-posedness results for derivative NLS with null structure, with Schr\"odinger Maps being one such model.

\end{abstract}


\maketitle

\setcounter{tocdepth}{1}
\tableofcontents

\section{Introduction}

In this article we propose a new solution to a classical problem in the field of dispersive PDEs: construct effective resolution spaces which inherit  as much as possible of the behavior of free solutions. The most basic set of properties of free solutions (with data in $L^2$) is the family of Strichartz estimates, and it is generally the case that resolution spaces (used in the literature) provide them. Under some transversality assumptions, free solutions enjoy additional estimates, such as bilinear or higher order of multilinearity type estimates, which reflect interactions between waves and are more difficult (or impossible) to capture using linear estimates. Designing effective resolutions spaces to capture such estimates has been the subject of intense research. 

Let us recall the bilinear $L^2_{t,x}$ estimate for the Schr\"odinger equation, 
\begin{equation} \label{bilS}
\|e^{it\Delta} f \cdot e^{it\Delta} g  \|_{L^2_{t,x}} \les 2^{\frac{(n-1)k_1-k_2}2} \|f\|_{L^2(\R^n)} \|g\|_{L^2(\R^n)},
\end{equation}
where $f$ is localized at frequency $\approx 2^{k_1}$, $g$ is localized at frequency $\approx 2^{k_2}$ and $k_1 \leq k_2$. If $n=1$, \eqref{bilS} requires additional transversality assumptions when $|k_1-k_2| \les 1$, but it holds as stated if $n \geq 2$. In this paper we work in dimensions $n \geq 2$ and all the statements made from here on are meant for $n \geq 2$. 

A similar estimates hold true for the wave equation:
\begin{equation} \label{bilW}
\|e^{it|\nabla|} f \cdot e^{it|\nabla|} g  \|_{L^2_{t,x}(\R^{n+1})} \les 2^{\frac{(n-1)k_1}2} \theta^{\frac{n-3}2} \|f\|_{L^2(\R^n)} \|g\|_{L^2(\R^n)},
\end{equation}
where $f,g$ are localized at frequency $2^{k_1}$, respectively $2^{k_2}$ and the angle between their support is $\approx \theta$. 
Similar estimates to \eqref{bilS} and \eqref{bilW} hold for other linear dispersive flows. It is very important to replicate the estimates \eqref{bilS} or \eqref{bilW} (or the equivalent for other flows) beyond the setup of free waves and such that the new setup accommodates actual solutions to nonlinear Schr\"odinger equations, or other dispersive flows.  Schr\"odinger Maps and Wave Maps are two such examples.

We now formalize the concept of a resolution space and that of a transference theory. We do so in the context of a nonlinear Schr\"odinger equation, but this outline can be adapted to any other nonlinear dispersive PDE.

Consider a general nonlinear Schr\"odinger equation:
\[
(i \partial_t + \Delta)u = N(u, \nabla u), \quad u(0,x)=u_{0}(x),
\]
where $u: \R \times \R^n \rightarrow \C$. To keep the presentation compact, we ignore the regularity at which one carries the theory. One seeks to identify two spaces $X,Y$ (with $X$ being a Banach space) with the following two properties:

1) The solvability of the inhomogeneous equation
\begin{equation} \label{LXY}
(i \partial_t + \Delta)u = f, \quad u(0,x)=u_{0}(x)
\end{equation}
holds true as follows
\[
\|u\|_{X} \les \|f\|_{Y} + \|u_0\|_{L^2(\R^n)}.
\]

2) Control of the nonlinearity
\begin{equation} \label{Nuvb}
\|N(u,\nabla u)\|_Y \les \| u \|_X^p
\end{equation}
where $p$ is the "degree" of the nonlinearity of $N(u)$; variations may appear. 

If, in addition, $X^* \subset Y$ (in a sense that needs to be made precise), then, for the purpose of \eqref{Nuvb}, it suffices to establish instead
\begin{equation} \label{NXY}
|\la N(u, \nabla u), v \ra| = |\int N(u,\nabla u) \bar v dx dt| \les \| u \|_{X}^p \|v\|_{X},
\end{equation}
Indeed, this guarantees that $N(u,\nabla u) \in X^* \subset Y$ with the appropriate bounds and \eqref{Nuvb} follows. 

$X$ is referred to as the resolution space for the underlying nonlinear PDE. One can add regularity to $X$, directly or by defining spaces $X_k$ meant to measure solutions localized at frequency $\approx 2^k$ and then summing the $X_k$ information in a fashion similar to how Sobolev regularity is recouped from information on the frequency localized components. 

Next we formalize the transference theory. Assume that we have a multilinear bound that holds for free solutions, that is for some $1 \leq k \leq n$,
\begin{equation} \label{mre}
\| \Pi_{j=1}^k e^{it\Delta} f_j \|_{S} \les \Pi_{j=1}^k \|f_j\|_{L^2_x},
\end{equation} 
where $S$ is a Lebesgue space $L^p_{t,x}$ or mixed one $L^p_{t} L^q_x$ (or potentially other spaces), the functions $f_j$ may be required to have some support/transversality conditions, and the constant in $\les$ is allowed to depend on some parameters characterization the support properties of $f_j$'s (such as frequency localization).  The space $X$ is amenable to transference theory if
\begin{equation}
\| \Pi_{j=1}^k u_j \|_{S} \les \Pi_{j=1}^k \|u_j\|_{X},
\end{equation} 
under similar requirements on the support of $u_j$ and with the same constant used in $\les$; the universal constant may be different, but the constant depending on the support properties of the inputs should be the same. Examples of such estimates are: the standard linear Strichartz estimates, the bilinear $L^2_{t,x}$ estimates such as \eqref{bilS} and \eqref{bilW}, the bilinear $L^p_{t,x}$ with $p<2$ counterpart of \eqref{bilS} and \eqref{bilW} (this will be detailed later in this Introduction), and $k$-linear estimates with $k \geq 3$, such as those developed by the author in \cite{Be-mrec}. There are also "broad" versions of the \eqref{mre}, see for instance \cite{Gu-I} an \cite{Gu-II}, but in PDE applications one needs the full estimate in \eqref{mre}. For the rest of this paper the focus will be on transferring the bilinear estimate \eqref{bilS}. 

We can now explain why the property $X^* \subset Y$ is desirable. If such a resolution space satisfies $X^* \subset Y$ and it is amenable to a transference theory, then we reach an ideal setup where, in the analysis of the nonlinear estimate \eqref{NXY}, we can deploy any known multilinear restriction estimate known for free solutions. 

Achieving the two properties mentioned above runs into a rather delicate balance. For $X$ to be amenable to a transference theory, it is helpful to make it smaller; this makes $X^*$ larger to the point of failing the solvability component, that is we fail to achieve \eqref{LXY} when $f \in X^*$.  To our best knowledge, constructing a resolution space $X$ which transfers the bilinear estimate \eqref{bilS} and has the property $X^* \subset Y$ has not yet been yet achieved. 

Let us now turn to concrete examples of how the above has been implemented in various contexts.

The linear transference setup ($k=1$) is the easiest. Consider a power type nonlinearity $N(u)=u^{r}$ or $=|u|^{r-1} u$ for some $r \in \N, r \geq 2$. If one can establish an estimate of type
\[
|\la N(u), v \ra| \les \|u\|^p_{X} \|v\|_{X},
\]
where $X=\cap L^p_t L^q_x$ is a finite intersection of admissible Strichartz spaces, then by letting $Y= X^*=\sum (L^p_t L^q_x)^*$, one achieves the desired result.

Variations of the above scheme, which are essentially based on the linear Strichartz estimates have been very successful in tackling many problems in dispersive PDEs. For a comprehensive introduction to this topic the books of Cazenave \cite{Caz-book} and Tao \cite{Tao-book} are very good references. 

In the above we presented a caricature of this argument, where we stripped off the regularity. If one needs to account for regularity, $N(u)$ is replaced by $N(u_1,..,u_{p})$ where each $u_i$ is localized at frequency $2^{k_i}$ and assume that $v$ is localized at frequency $2^k$, which essentially means that we estimate $P_k N(u_1,..,u_{p})$. The spaces may be frequency adapted, that is we measure each $P_k f$ in the corresponding $X_k$ and the constants that appear in $\les$ will depend on the variables $k_1,..,k_p, k$, that is we obtain an estimate with $\les C(k_1,..,k_p,k)$. The precise formula giving $C(k_1,..,k_p,k)$ dictates the (minimal) regularity $s$ that is needed to close the iterative argument. 

Regardless of how much variation of the above caricature scheme one uses, the crucial ingredient is that
one relies on the linear Strichartz estimates, and the standard linear theory for \eqref{LXY} recovers the Strichartz estimates when the forcing $f$ belongs to a dual Strichartz space. In other words, the standard theory providing the Strichartz estimates is effectively providing the necessary transference theory in this setup. 

On the other hand, replicating the above scheme for problems where a bilinear estimate such as \eqref{bilS} is needed  is a more difficult task; for instance such estimates are necessary in nonlinear Schr\"odinger equations which contain derivatives. 
As a consequence, seeking resolution spaces which could emulate \eqref{bilS} (or its counterpart for other flows) has been an important problem in nonlinear dispersive PDEs. There are a few candidates which have been effective to various extents which we now explain. 

One of the main directions was to construct resolution spaces whose elements can be thought as "formal superpositions" of free solutions to the underlying PDE. The  "$X^{s,b}$" spaces which appeared in the works of Beals \cite{Bea} (for the wave equation), Bourgain \cite{Bo-1, Bo-2} (for the Schr\"oedinger and KdV equations) and Klainerman and Machedon \cite{KlMa} (for the wave equation) were the first structures known to be effective in replicating estimates known for bilinear interactions of free waves. $s$ corresponds to classical Sobolev regularity and we set it to be $s=0$ in this paragraph; doing so brings no changes to the main point made below.  For local in time theories, taking $b> \frac12$ ensures that $X^{0,b}$ spaces are effective in providing a transference theory, while for global in time theories the space $X^{0,\frac12,1}$ plays the same role. 
However, in both cases, the $X^*$ is not a good candidate for solvability of \eqref{LXY}. Indeed, $(X^{0,b})^* = X^{0,-b}$ and $(X^{0,\frac12,1})^* = X^{0,-\frac12, \infty}$ and the solution of \eqref{LXY} with forcing in such a space ends up in $X^{0,-b+1}$, correspondingly in $X^{0,\frac12,\infty}$ which are both strictly larger than the original $X^{0,b}$ or $X^{0,\frac12,1}$, and, most importantly, they fail to be effective in a transference theory.  However, in combination with other structures, particularly the end-point Strichartz estimates and energy type estimates, these spaces have been extremely effective in providing a good resolution space for constructing global solutions to difficult nonlinear PDEs. 

A second direction was undertaken by Tataru in his attempt to tackle the Wave Map equation in low dimensions. In an unpublished manuscript, Tataru introduced the  $U^2,V^2$ spaces (adapted to the wave equation) and attempted to prove the global well-posedness of the wave map equation in low dimensions using these spaces, with one of the key ingredients being that these structures transfer the bilinear $L^2_{t,x}$ estimate for free waves \eqref{bilW}. While an error was found by Nakanishi in that manuscript, later works by Tataru and Koch \cite{KT, KT1,KT2}, Hadac, Herr, Koch \cite{HHK}, Herr, Tataru, Tzvetkov \cite{HeTaTz}, to name just a few, have successfully used these spaces in various contexts. In particular, transference properties have been formalized by Hadac, Herr, Koch  \cite{HHK}, see Proposition 2.19; yet in all these works the bilinear $L^2_{t,x}$ for free solutions is recovered when $X=U^2_{flow}$ ($U^2$ adapted to the flow), but losses occur once $V^2_{flow}$ is involved; this is needed because if one starts with $X=U^2_{flow}$, then when solving \eqref{LXY} with forcing in $(U^2_{flow})^*$, the solution is returned in $V^2_{flow}$ - more details about this process are provided in section \ref{CRS}. Later in this introduction we discuss some more recent refinements in the works Candy-Herr \cite{CaHe-t, CaHe} and Candy \cite{Ca} in this context. 

Introduced by Tataru, in the context of the wave maps problem, the last and most effective direction was to circumvent entirely a transference theory and instead obtain the bilinear estimate from linear estimates via the caricature estimate $L^2 L^\infty \cdot L^\infty L^2 = L^2_{t,x}$. In high dimensions this works well given that the end-point Strichartz estimate $L^2_t L^\infty_x$ is available for the wave equation ($n \geq 4$); in combination with the $X^{0,\frac12,1}$ space, Tataru \cite{Ta-wm1} established the global well-posedness for the wave maps problems with small data in the Besov space when $n \geq 4$. However this runs into problems in low dimensions, that is when $n=2,3$. In \cite{Ta-2}, again in the context of the wave map equation, Tataru recovered such estimates in adapted frames: fairly rigid frames (and many of them) are needed to capture estimates of type $L^2_{t_\Theta} L^\infty_{x_\Theta}$, while energies of type $L^\infty_{t_\Theta} L^2_{x_\Theta}$ are more flexible and capture the underlying geometry of wave propagation. The frame method was efficiently used in the Schr\"odinger Map problem \cite{BIKT-1}, as well as in Klein-Gordon type problems, see \cite{BeHe}. We note however that these structures do not have the property $X^* \subset Y$ mentioned in \eqref{NXY}.

To summarize the current state of the art around the transference theory, the current structures, whether $X^{s,b}$ spaces or $U^2,V^2$ adapted spaces, are partially, but not completely, amenable to a bilinear transference theory: in all cases one encounters at least logarithmic (or even power) losses over what the bilinear restriction theory for free wave records. The spaces build on adapted frames do not have the property $X^* \subset Y$. 

In this article we propose a new resolution space for the Schr\"odinger equation which has the two key desired properties:
it is amenable to a transference theory and its dual acts as an effective structure for the inhomogeous equation; this is why we have two main theorems. The starting point is that for each $k \in \Z$ and $m \in \Z, m \leq 2k$ we can construct a space $X_{k,m}$ such that the following two theorem hold true.

\begin{theorem} \label{MTbil}
Assume $u \in X_{k_1,2k_1}, v \in X_{k_2,2k_1}$ are localized at frequency $\approx 2^{k_1}$, respectively $\approx 2^{k_2}$, where $k_1 \leq k_2$. Then the following holds true:
\begin{equation} \label{uvL2}
\| u \cdot v \|_{L^2_{t,x}} \les 2^{\frac{(n-1)k_1-k_2}2} \|u\|_{X_{k_1,2k_1}}  \|v\|_{X_{k_2,2k_1}}.
\end{equation}
\end{theorem}

\begin{theorem} \label{MTl}
i) Homogeneous equation: assume $g \in L^2_x$ is localized at frequency $2^k$ and $m \leq 2k$. Then $e^{it\Delta} g \in X_k$ with the bound
\[
\| e^{it\Delta} g \|_{X_{k,m}} \les \|g\|_{L^2_x}. 
\]
ii) Consider the inhomogeneous equation 
\[
(i \partial_t + \Delta) u =f, \quad u(0) = 0. 
\]
We assume that $f \in L^\infty_t L^2_x$ is localized at frequency $\approx 2^{k}$ and the following holds true
\begin{equation}
|\la f, \phi \ra | \leq C \|\phi\|_{X_{k,l}}
\end{equation}
for any $\phi \in X_{k,l}$, where $l \leq 2k$. Then the following holds true:
\begin{equation}
\sup_{m \leq 2k} \| u \|_{X_{k,m}} \les C.
\end{equation}
\end{theorem}

In a first reading of the above theorems it may be helpful to ignore the second parameter $m$ in the space $X_{k,m}$ so that one gets a sense of what the results are stating. Now we explain some features of the space $X_{k,m}$. First $k$ indicates that they meant to measure functions localized at frequency $\approx 2^k$. Second, the parameter $m$ is needed since in the analysis of the bilinear estimate \eqref{uvL2} we discover a modulation threshold which plays a key role - this depends on the low frequency. 
We could free the spaces from this second parameter and state the above results using $S_k = \cap_{m \leq 2k} X_{k,m}$ with the appropriate norm, but this would obscure some of the arguments and add some technicalities when using duality. Note however that 
the last estimate in the Theorem \ref{MTl} provides an estimate for $u$ precisely in $S_k$.  

As applications we show how these resolution spaces can be used to solve the division problem for derivative nonlinear Schr\"odinger equations, where the nonlinearity has a null structure. Such an example is the Schr\"odinger Map equation, where we provide an alternative proof for the classical result of global well-posedness for small and localized data in the Besov critical space. To be more precise our result in Theorem \ref{SMT} is alternative to the ones in \cite{Be2, IoKe-2} in dimensions $n \geq 3$, while the corresponding one (to the results in \cite{Be2, IoKe-2}) for $n=2$  has never been written but could be written based on the work in \cite{BIKT-1}. 

Our construction is not limited to the Schr\"odinger equation and we anticipate that similar constructions can be adapted to other flows, such as the wave equations. 

Let us put our work in the context of similar attempts to transfer bilinear estimates, particularly the work of Candy \cite{Ca}, and Candy and Herr \cite{CaHe}. Before doing so, it is helpful to review the bilinear restriction theory. For the Schr\"odinger equation this stated the following
\begin{equation} \label{bilSp}
\|e^{it\Delta} f \cdot e^{it\Delta} g  \|_{L^p_{t,x}} \les \|f\|_{L^2(\R^n)} \|g\|_{L^2(\R^n)},
\end{equation}
holds true if $p \geq \frac{n+3}{n+1}$, where $f,g$ are localized at frequency $\approx 1$ and their Fourier support is separated. 
Note that this improves the integrability of the product over the bilinear $L^2_{t,x}$ in \eqref{bilS} once $n \geq 2$.

Klainerman and Machedon conjectured that the equivalent estimate for the wave equation holds true in the same range of exponents; this was formalized in \cite{FoKl} by Foschi and Klainerman. After early progress by Bourgain \cite{Bou-CM}, and Tao and Vargas \cite{TV-CM1}, the full range for the wave equation $p >  \frac{n+3}{n+1}$  was established by Wolff in \cite{Wo}, while the end-point $p = \frac{n+3}{n+1}$ was established by Tao in \cite{Tao-BW} (both these results hold for all dimensions $n \geq 2$). Tao established the the full range for the Schr\"odinger equation $p >  \frac{n+3}{n+1}$ in \cite{Tao-BP}.
  
In \cite{CaHe-t} Candy and Herr initiated the study of transferring bilinear $L^p_{t,x}$ (such as \eqref{bilSp}) for a large class of phase functions and in the context of $U^2$ and $V^2$ adapted to the underlying flow as resolution spaces. They established, for the first time, that $U^2$ effectively transfers all the known bilinear $L^p_{t,x}$ estimates for free waves, while $V^2$ does so only if $p <2$; in earnest the first time where such a transferred estimate has been used was in \cite{StTa-1}. 
 
In \cite{Ca}, Candy provided a very general result for the transference of mixed Lebesgue type $L^p_t L^q_x$ estimates for large classes of phase, covering most flows that are needed in practice; the value of this work lies also in providing the general bilinear theory for free waves as well. In a nutshell, Candy establishes that $U^2_{flow}$ transfers all known bilinear estimates, while $V^2_{flow}$ transfers the estimates with $p <2$. Another interesting result that was established in \cite{Ca} was the following generalization of \eqref{bilW}
\begin{equation} \label{bilWU}
\|u \cdot v  \|_{L^2_{t,x}(\R^{n+1})} \les 2^{\frac{(n-1)k_1}2} \theta^{\frac{n-3}2} \|u\|_{U^2_{|\nabla|}} \|v\|_{V^2_{|\nabla|}},
\end{equation}
where $u,v$ are localized at frequency $2^{k_1}$, respectively $2^{k_2}$, and $k_1 \leq k_2$ and the angle between their support is $\approx \theta$. It was known, from the work of Candy-Herr \cite{CaHe-t} that the above was true if both terms were in $U^2_{|\nabla|}$; it was also desirable to establish the same estimate with $u,v \in V^2_{|\nabla|}$. Candy succeeded to obtain an "intermediate" estimate where the low frequency has to be kept in $U^2_{|\nabla|}$, but the high frequency can be relaxed to $V^2_{|\nabla|}$. Based on the crucial observation,  in \cite{CaHe} Candy and Herr solved the division problem for wave maps using the  adapted $U^2$ as a resolution space, just as Tataru intended in his original attempt; to be more precise they closed the global iteration argument for the wave map equation in the critical Besov space. 

However, \eqref{bilWU} seems to be limited to a wave type behavior and not known to hold true for, say, the Schr\"odinger equation. 
Our work does not fix this problem; instead we construct the spaces from ground-up and obtain the estimate \eqref{uvL2} which is 
 symmetric in the structures used. 

We also note that our result in Theorem  \ref{SMT} is the  Schr\"odinger Maps counterpart of the result for Wave Maps obtained by Candy and Herr in \cite{CaHe}. 

Returning to the statement of Theorem \ref{MTbil}, a natural question to ask is whether our structures transfer $L^p_{t,x}$ with $p<2$ (or mixed $L^p_{t} L^q_x$) estimates which are available for free solutions. The answer is 
positive as long as $p <2$ and it is covered by the theory developed by Candy in \cite{Ca}; this is simply a consequence of the fact that $V^2_\Delta$ is part of our resolution spaces and in \cite{Ca} it was shown that these estimates are transferred in the context of $V^2_\Delta$ if $p<2$. For the precise statements of available bilinear estimates of such type we refer the interested reader to \cite{Ca}.

\subsection{A summary of the key ideas} In this section we highlight the main ideas in our work.

In most prior literature, the transference theory has been attempted on given structures such as $X^{s,b}$ type spaces or $U^2, V^2$ adapted to the underlying flow. Our approach takes a different strategy: we aim to provide a transference theory, while seeking to identify the structures that are amenable to such a  theory. We highlight that, in some sense, our identification follows a first principle approach in that the identify what is needed and missing from the "base structure" and then augment it.

It is known that if one can close a transference theory at the level of the $V^2_\Delta$ structure, then the problem is solved since $V^2_\Delta$ would become a good resolution space by itself. But it is also known that is unlikely to have a transference theory solely based on the $V^2_\Delta$ structure. Our approach is to run as much as possible of the transference at the level of $V^2_\Delta$ structure and one of the novelties of this work is that one can do so if both factors have low modulations. Here we use the more advanced bilinear $L^p_{t,x}$ theory with $p<2$ in a modified way so as to essentially highlight an improvement of the actual bilinear $L^2_{t,x}$ for free waves outside cubes at a scale dictated by the frequency scales in our problem. This is then transferred to the $U^2_\Delta$ setup. 
On the other hand, within each such cube we obtain a very soft estimate that uses minimal information and an interpolation procedure gives the bilinear \eqref{bilS} for low modulation components in the $V^2_\Delta$ structure. 

Next, we consider the case when one of the terms has high modulation relatively to the basic scale. Here we identify the key information that is needed to close the arguments and that is reasonably expected to hold for solutions: local smoothing (for the high frequency) and a refined version of the standard mass estimate $L^\infty_t L^2_x$ (for the low frequency). Both these properties hold for free waves, but are not known to hold for functions in $V^2_\Delta$ and this is why we augment $V^2_\Delta$ with these structures. 

There is one last and important structure that we introduce, and which seems to be new in the literature. Once the modulation is bounded from below, we introduce an $l^2$ "decoupling" between the information on the physical space that are localized in blocks at the "dual scale". The "dual scale" is constructed  as follows: the dual time scale is simply the dual scale to the modulation and this is turn dictates the dual spatial scale based on the speed of propagation. This construction is specific to high modulation setup, as the estimates obtained do not hold true for free solutions. 

To summarize, we bring together ideas from the prior candidates, $X^{s,b}$ and $U^p,V^p$ (all adapted to the flow), the use of local smoothing and the improved bilinear $L^p_{t,x}, p < 2$ theory. Combining these with a new structure in high modulation gives rise to our resolution space.  

\subsection{Notation} \label{subsect:not}
We define $A\les B$, if there is a harmless constant $c>0$ such that $A\leq c B$, and $A\ges B$ iff $B\les A$. Further, we define $A\approx B$ iff both $A\les B$ and $B\les A$. Also, we define $A\ll B$ if the constant $c$ can be chosen such that $c<2^{-10}$. Also, $A\gg B$ iff $B\ll A$.

Similarly, we define $A\preceq B$ iff $2^A\les 2^B$, $A\succeq B$ iff $2^A\ges 2^B$, $A \sim B$ iff $2^A\approx 2^B$, $A\prec B$ iff $2^A\ll 2^B$, $A\succ B$ iff $2^A\gg 2^B$.

Let $\rho^0\in C^\infty_c(-\frac85,\frac85)$ be a fixed smooth, even, cutoff
satisfying $\rho^0(s)=1$ for $|s|\leq \frac54$ and $0\leq \rho\leq 1$. For
$k \in \Z$ we define $\rho_k: \R^n \to \R$,
$\rho_k(y):=\rho^0(2^{-k}|y|)-\rho^0(2^{-k+1}|y|)$. 
Let $\tilde{\rho}_k=\rho_{k-1}+\rho_k+\rho_{k+1}$ and notice that $\tilde{\rho}_k \cdot \rho_k=\rho_k$. For $k \in \Z$, let $P_k$ be the Fourier multiplier with symbol with respect to $\rho_k$ and  $\tilde P_k$ be the Fourier multiplier with symbol with respect to $\tilde{\rho}_k$; we record the identity $\tilde P_k P_k = P_k$. We say that $f$ is localized at frequency $2^k$ if $\hat f$ is supported in the annulus $\{\xi: 2^{k-1} \leq |\xi| \leq 2^{k+1}\}$. We note that $P_k f$ is localized at frequency $2^k$ and that $\tilde P_k f = f$ for any $f$ that  is localized at frequency $2^k$.  Next we define $P_{\leq k}=\sum_{0\leq k'\leq k }P_{k'}$, and $P_{> k}=I-P_{\leq k}$. 

For any  $\e\in\mathbb{S}^{n-1}$ and $k\in\Z$ we define the operators $P_{k,\e}, \tilde P_{k,\e}$ by the Fourier multipliers with symbol 
$\rho_k(\xi\cdot\e)$, respectively $\tilde \rho_k(\xi\cdot\e)$. 
We select a finite subset of vectors $V^n=\{\e_j: \e_j \in \mathbb{S}^{n-1} \}$ with the property that for any $f$ localized at frequency $2^k$ there exists a decomposition of $ f$ with the following property
\begin{equation} \label{decej}
 f = \sum_{j: \e_j \in V^n} f_j , \quad   \tilde P_{k,\e_j} f_j = f_j. 
\end{equation}
It is an easy exercise to construct  such a set which is also independent of the value of $k$. This is based on the observation that if $f$ is supported in the set $\{ \xi: |\xi \cdot \e| \geq (1-c) |\xi| \}$ for some universally small $c>0$ and some $\e\in\mathbb{S}^{n-1}$ then we have that $\tilde P_{k,\e} f =  f$; then we cover the sphere with finitely many such sets and construct a partition of unity subordinated to them. 

Next we introduce norms adapted to specific frames. For $\e \in \mathbb S^{n-1}$, we let $x_\e=x \cdot \e$ (the coordinate along the vector $\e$), and $\bar x_\e=x \cdot \e^\perp$ (the coordinates in the directions orthogonal to $\e$). For a function $h:\R^{n+1} \rightarrow \C$, we introduce the norms $\|h\|_{L^\infty_{x_\e} L^2_{\bar x_\e,t}}$ and $\|h \|_{L^1_{x_\e} L^2_{\bar x_\e,t}}$ in the standard fashion. Following  \cite{BIKT-1}, if one denotes by $H_\e$ the orthogonal complement of $\e$ in $\R^n$ with the induced measure and defines
the spaces $L^{\infty,2}_\e$ with norms
\begin{equation}\label{Lpqe}
  \|h\|_{L^{\infty,2}_{{\e}}}=
  \sup_{x_1} \Big[\int_{H_\mathbf{e}\times \mathbb{R}}
  |h(x_1\mathbf{e}+x',t)|^2\,dx'dt\Big]^{1/2},
\end{equation}
and similarly $L^{1,2}_{{\e}}$; it is easy to see that $ \|h\|_{L^{\infty,2}_{{\e}}} = \|h\|_{L^\infty_{x_\e} L^2_{\bar x_\e,t}}$ 
and $ \|h\|_{L^{1,2}_{{\e}}} = \|h\|_{L^1_{x_\e} L^2_{\bar x_\e,t}}$

For $k \in \Z$, we let $\mathcal{I}_k=\{ [2^{k}l,2^{k}(l+1)]: l \in \Z\}$ be the set of intervals of length $2^{k}$ covering $\R$; these are meant to be time interval throughout this paper. Also, $C_k$ is the set of cubes of side-length $2^{k}$ on the physical side covering $\R^n$ - this can be obtained by considering the standard $\Z^n$ lattice in $\R^n$ and rescale it by $2^{k}$. 

Given an interval $I \subset \R, |I| >0$, a bump function  $\varphi$ adapted to $I$ is defined in a standard fashion, that is it
 satisfies the following bounds
\[
| ((t-c(I)) |I|^{-1})^\alpha (|I| \partial_t)^\beta \varphi | \les_{\alpha,\beta} 1,
\]
for any $\alpha, \beta \in \N \cup \{0\}$. Here $c(I)$ is the center of $I$, but it can chosen to be any other point in $I$. We also denote $\chi_I$ to be the standard characteristic function of the interval $I$; we will also use the notation $\mathds{1}_I$. Since $\chi_I$ is not smooth, we let $\tilde \chi_I$ be a smooth approximation of $\chi_I$ in the sense that $\tilde \chi_I$ is a bump function adapted to $I$ with two additional properties: $\tilde \chi_I(t)=1, \forall t \in I$, thus $\tilde \chi_I \chi_I = \chi_I$, and it is supported in $2I$, the double of $I$. $2I$ is meant to be the double of $I$ from its center, that is if $I=[c(I)-\frac{|I|}2, c(I)+ \frac{|I|}2]$, the $I=[c(I)-|I|, c(I)+ |I|]$. Accordingly we will have the functions $\chi_{2I}$ and $\tilde \chi_{2I}$. 

Given a cube $q \in \R^n$ a bump function adapted to $q$ is defined similarly, $\chi_q$ is its characteristic function, and $\tilde \chi_q$ is the smooth approximation of $\chi_q$ (a bump function adapted to $q$, compactly supported in $2q$, the double of $q$ from its center, and with $\tilde \chi_q \chi_q =\chi_q$). 

The Fourier transform of a Schwartz function $f:\R^n \rightarrow \C$ is defined by
\[
\mathcal{F}(f)(\xi)=\hat f(\xi)= \int e^{-i x \cdot \xi} f(x) dx. 
\]
The inverse Fourier transform is defined by 
\[
\mathcal{F}^{-1}(g)(x)=\check g(x)= \frac1{(2\pi)^n} \int e^{i x \cdot \xi} g(\xi) d\xi.
\]
These definitions are then extended to distributions, in particular to $L^p(\R^n)$ spaces, in the usual manner. 

\bf Acknowledgments: \rm The author is grateful to his PhD students, Vitor Borges and Tiklung Chan, for their feedback on an earlier version of this manuscript. The author is also grateful to Timothy Candy for pointing out and explaining a result in \cite{Ca-note}, which we used as a reference for a compact transference argument at the end of the Appendix in this paper. 

\section{The construction of the resolution space} \label{CRS}

In this section we introduce our resolution spaces. We begin with recalling the basic theory of $U^p,V^p$ spaces. 
Then we identify some new properties in the context of some frequency localization. Then we recall their counterpart $U^p_\Delta,V^p_\Delta$ in the context of Schr\"odinger flow along with their standard properties. Finally we construct our key resolution space by augmenting the $V^2_\Delta$ space with three key structures: a refined version of the standard mass estimate,  local smoothing and refinements of these structures in the context of higher modulation.

\subsection{$U^p$ and $V^p$ spaces} 

In this section we provide a brief introduction of the $U^p,V^p$ spaces; the material here can be found in \cite{KT,HHK,KT1,KT2,CaHe} and the interested reader is advised to consult these papers for any additional details; we would suggest \cite{HHK}, \cite{KTV} and \cite{KT2} as the more comprehensive ones.

Below we consider intervals $(a,b)$ with $-\infty \leq a < b \leq +\infty$. A partition $\tau$ of $(a,b)$ is a finite monotone sequence $a < t_1 < ...< t_N =b$ 

We work with functions $u,v: (a,b) \rightarrow L^2(\R^n)$; as it is standard in this context, throughout this paper we work with function $u,v$ that are ruled in the sense that they have left and right limits everywhere. We also consider $1 \leq  p < \infty$ below. 

\begin{definition}  A $U^p(a,b)$ atom is a step function
\[
a(t)=\sum_{k=1}^K \mathds{1}_{[t_{k-1},t_{k})}(t) \phi_k , \quad \sum_{k=1}^K\|\phi_k\|_{L^2(\R^n)}^p=1,
\]
$U^p(a,b)$ is the atomic space based on these atoms, that is $u \in U^p(a,b)$ if
\[
u = \sum_{j \in \N} \lambda_j a_j,
\]
where $(\lambda_j) \in l^1$ and $a_j$ are atoms as above; $U^p(a,b)$ is equipped with the standard norm
\[
\|u\|_{U^p(a,b)} = \inf \{ \sum_{j \in \N}  |\lambda_j |:  u = \sum_{j \in \N} \lambda_j a_j \ \mbox{as above} \}. 
\]
\end{definition}

By definition, functions $u \in U^p$ are right continuous and $\lim_{t \rightarrow a+} u(t)=0$ in $L^2(\R^n)$. If we allow atoms to include the end-point $t_1=a$, we denote the new space by $\tilde U^p$; note that in this context we lose $\lim_{t \rightarrow a+} u(t)=0$. This is particularly useful when we involve the $U^p$ structure on subintervals of $(a,b)$. In this context we have the following result (see Lemma B.5 in \cite{KT2}):
\begin{equation} \label{U2int}
\| u_1 + u_2 \|^2_{\tilde U^2(a,b)}  \leq \| u_1 \|^2_{\tilde U^2(a,c)} + \| u_2 \|^2_{\tilde U^2(c,b)}, 
\end{equation}
where $c \in (a,b)$. 

\begin{definition} Given a function $v: (a,b) \rightarrow L^2(\R^n)$ we define
\[
|v|_{V^p} =\sup_{\tau=\{(t_j)_{j=1,N}\}} \left( \sum_j \| v(t_{j+1}) -v(t_j) \|_{L^2(\R^n)}^p \right)^\frac1p, 
\quad  \|v\|_{V^p} = |v|_{V^p} + \|v\|_{L^\infty_t L^2_x}. 
\]
We let $V^p=V^p(a,b)$ is the space of functions  $v: (a,b) \rightarrow L^2(\R^n)$ for which $ \|v\|_{V^p}$ is finite. 

\end{definition}

It is easy to show that if $v \in V^p$, then $v$ has lateral limits everywhere, including the end-points $a,b$ (this includes the case $a=-\infty$ and $b=+\infty$); for reference see \cite{HHK}. In order to avoid jumps in $V^2$ that are not related to the lateral limits, it is customary to work with $V^2_{rc}$, the subspace of $V^p$ of functions that are right-continuous, or  $V^2_{lc}$, the subspace of $V^p$ of functions that are left-continuous; in addition a normalization at one of the end-points is usually present: in $V^2_{rc}$ one uses $\lim_{t \rightarrow a+} v(t)=0$ (in $L^2(\R^n)$), while in $V^2_{lc}$ one uses $\lim_{t \rightarrow b-} v(t)=0$. Such a normalization allows to skip the use of $\|v\|_{L^\infty_t L^2_x}$ in the $\|\cdot \|_{V^p}$ and simply use only $|v|_{V^p}$. 

In practice, in the context of solution to nonlinear PDEs, we will have apriori knowledge of the continuity in time of solutions, while the normalization at $-\infty$ or $+\infty$ is unnecessary and incompatible with prescribing an initial data at a specific time $t_0$, which itself acts as a normalization. However one can reach any desired normalization by subtracting a constant in time function. 

Next we continue with basic properties of these spaces. We record the embeddings
\begin{equation}
U^p \subset V^p \subset U^q, \quad  1 \leq p < q < \infty;
\end{equation}
the first embedding is straightforward for atoms, hence for the whole $U^p$, while the second requires a bit more work, see for instance \cite{HHK} for a complete proof. We have found Remark 4.2 in \cite{CaHe} to be helpful in providing some more intuition into this spaces. 

It is important to understand how the $U^p$ and $V^p$ structures interact with decompositions which have good orthogonality properties; we use the following result from Proposition 4.3 in \cite{CaHe}.
\begin{proposition} \label{Pref}
Let $M_1,M_2 \in (0,+\infty)$. For $k \in \N$, let $T_k : L^2(\R^n) \rightarrow L^2(\R^n)$ be a linear and bounded operator (acting spatially) such that
\[
M_1^{-1} \| \sum_{k \in \N}  T_k f \|_{L^2} \leq \left( \sum_{k \in \N}  \|T_k f \|^2_{L^2} \right)^\frac12 \leq M_2 \|f\|_{L^2}, \quad \forall f \in L^2.
\]
If $1 \leq p \leq 2$, then for all $u \in U^p$ we have the bound
\[
\left( \sum_{k \in \N}  \|T_k u \|^2_{U^p} \right)^\frac12 \leq M_2 \|u\|_{U^p}.
\]
If $ p \geq 2$, then for all $v \in V^p$ we have the bound
\[
 \| \sum_{k \in \N}  T_k v \|_{V^p}  \leq M_1 \left( \sum_{k \in \N}  \|T_k v \|^2_{V^p} \right)^\frac12. 
\]
\end{proposition}

The (distributional) time derivative of functions in $U^p,V^p$ plays an important role in the theory. We define
\begin{equation} \label{DU}
DU^p =\{ \partial_t u: u \in U^p \}, \quad DV^p =\{ \partial_t v:  v \in V^p \},
\end{equation}
where $\partial_t$ is meant to be a distributional derivative. An important ingredient is the following formal duality (which is not to be read verbatim below)
\[
(DU^p)'=V^q, \quad (DV^p)'=U^q, \quad 1 < p,q< \infty, \frac1p+\frac1q=1.  
\]
By now this duality is completely understood, particularly in light of the work of Koch and Tataru in \cite{KT2}. We use the following formalization from \cite{KT2}, Proposition B.17; see also \cite{CaHe}, Theorem 4.4 and 4.6. 

\begin{theorem} \label{DualUV} If $1 < p,q< \infty, \frac1p+\frac1q=1$, the spaces $DU^p$ and $DV^p$ can be characterized as the space of distributions for which the following norms are finite:
\begin{equation} \label{TDU}
\|f\|_{DU^p}=\sup_{ \phi \in C_0^\infty: \| \phi\|_{V^q} \leq 1}  \left| \int \la  f , \phi \ra_{L^2} dt \right|,
\end{equation}
\begin{equation} \label{TDV}
\|f\|_{DV^p}=\sup_{ \phi \in C_0^\infty: \| \phi\|_{U^q} \leq 1}  \left| \int \la  f , \phi \ra_{L^2} dt \right|,
\end{equation}
\end{theorem}
 
\begin{r1}
In practice, the distributions $f$ that we test will have additional properties such as $f \in L^1_{t, loc} L^2_x$, and, as a consequence, the associated $u$ with the property $f = \partial_t u$ comes with a priori continuity properties $u \in C_t L^2_x$. In order to uniquely identify $u$, either in $U^p$ or $V^p$, one has to impose a normalization, such as the standard one in $U^p$, $\lim_{t \rightarrow a+} u(t)=0$, or the standard one in $V^p$, $\lim_{t \rightarrow b-} u(t)=0$. 
\end{r1}

For $s \in \R$ and $1 \leq p,q \leq +\infty$, we let 
\[
\|f\|_{\dot B^s_{p.q}}=\big\|\big(2^{sj}\|P_{j}(\partial_t)
f\|_{L^p(\R: L^2(\R^n))}\big)_{j \in \Z }\big\|_{\ell^q}. 
\]
We are interested in the particular case $p=2$ and record the following embeddings
\begin{equation} \label{BUV}
\dot B^{\frac12}_{2,1} \subset U^2 \subset V^2 \subset \dot B^{\frac12}_{2,\infty}.
\end{equation}
These embeddings are well known: see for instance $(2.5)$ in \cite{KT1} in the context of the Schr\"odinger flow (which is what we need in this paper) or \cite[Proposition 2.14]{HHK} for a more general version (more general $p$). 

We note that the (time) frequency cut-offs are bounded on the $U^p,V^p$ spaces, that is $P_{\geq k}(D_t), P_{< k}(D_t)$ are bounded operators on $U^p,V^p$; this is based on the fact that the kernel of $P_{< k}(D_t)$ is an $L^1_t$ function. For complete arguments see for instance \cite[Proposition 2.18]{HHK} or \cite[Lemma 4.17]{KTV}.

We conclude this section with a result highlighting a new structural property of high modulation components of functions in $U^2$; to our best knowledge this is new in the literature. 

\begin{lemma} \label{U2dec}
Assume that $u \in U^2$. Given $m \in \Z$ assume $\{I_l\}_{l \in \Z}$ is a set of intervals in $\R$ of length $\approx 2^{-m}$ with mutually disjoint interior. Then the following holds true
\begin{equation} \label{U2IX}
\sum_{l} \|\mathds{1}_{I_l}(t)  P_{\geq m}(D_t) u \|_{U^2}^2 \les \|u\|_{U^2}^2.  
\end{equation}
\end{lemma}
The most practical example of a set of intervals $I_l$ is $\mathcal{I}_l=\{ 2^{-m}l, 2^{-m}(l+1): l \in \Z \}$. 
As stated the result above is implicitly made in terms of $U^2(\R)$, but it can be easily adjusted to $\tilde U^2(I_l)$.

\begin{proof} Throughout this proof we will repeatedly use the fact that $U^2$ is stable under standard frequency projectors. We drop the use of $P_{\geq m}(D_t)$ and simply assume that $u$ satisfies $P_{\geq m-10}(D_t)u=u$. It suffices to establish the above property for atoms in $U^2$, thus consider 
\[
a(t)=\sum_{k=1}^K \mathds{1}_{[t_{k-1},t_{k})}(t) \phi_k , \quad \sum_{k=1}^K\|\phi_k\|_{L^2(\R^n)}^2=1.
\]
We organize the intervals  $I_k = [t_{k-1},t_{k})$ based on how their length compares with $2^{-m}$ and split $a(t)$ as follows
\[
a=a_1+a_2, \quad a_1(t)=\sum_{k: |I_k| \leq 2^{-m}} \mathds{1}_{I_k}(t) \phi_k, \quad a_2(t)=\sum_{k: |I_k| > 2^{-m}} \mathds{1}_{I_k}(t) \phi_k. 
\]
To regain the frequency localization we instead work with $P_{\geq m} a_1$ and $P_{\geq m} a_2$.

The claim for $a_1(t)=\sum_{k: |I_k| \leq 2^{-m}} \mathds{1}_{I_k}(t) \phi_k$ is the easier one. We organize
\[
a_{1,l}: = \sum_{k: |I_k| \leq 2^{-m}, I_k \cap I_l \ne \emptyset} \mathds{1}_{I_k}(t) \phi_k,
\]
note that $a_{1,l} \chi_{I_l}$ is an $U^2$ atom on the interval $I_l$, and use the fact that an interval $I_k$ can intersect at most two intervals $I_l$ to obtain the desired conclusion. This is stable when applying $P_{\geq m}$ since $P_{\geq m}=I - P_{<m}$ and the kernel $K_{<m}$ of $P_{<m}$ is an $L^1$ normalized  bump function adapted to the interval $[-2^{-m},2^{-m}]$. 

Thus we are left with establishing the property for $a_2(t)$ and here it suffices to analyze a single interval $I_k$ given that $|I_k| > 2^{-m}$ since we have $l^2$ summability with respect to these intervals from the original structure.  The difficulty here comes from the fact that an interval $I_k$ can intersect a large number of intervals $I_l$, and the estimate should be independent of this number. We further write
\[
\mathds{1}_{I_k}(t) = \sum_l \tilde \chi_{I_l}(t) \mathds{1}_{I_k}(t), 
\]
where it suffices to restrict the sum over those $l$'s with the property that the double of the interval $I_l$ intersects $I_k$. Next we note that there are only finitely many $l$'s for which the identity $\tilde \chi_{I_l}(t) \mathds{1}_{I_k} = \tilde \chi_{I_l}(t) $ fails, essentially those those intervals $I_l$ whose double contains $t_{k-1}$ or $t_k$. These intervals can be treated just as those in $a_1$. 

For the other intervals where $\tilde \chi_{I_l}(t) \mathds{1}_{I_k} = \tilde \chi_{I_l}(t) $ holds true we record the decay estimate $|\mathcal{F} \tilde \chi_{I_l} (\tau)| \les 2^{-m} \la 2^{-m} |\tau| \ra ^{-N})$. Next we write 
\[
P_{\geq m} \mathds{1}_{I_k}(t)= \mathds{1}_{I_k}(t) - P_{< m} \mathds{1}_{I_k}(t),
\]
and note that $P_{< m}= K_{<m} \ast $ performs an averaging (highly) concentrated at scale $2^{-m}$, and its kernel satisfies
\[
| (2^{-m} \partial_t)^\alpha K_{<m}(t)| \les_\alpha 2^m \la 2^m t \ra^{-N}, \forall \alpha \geq 0, \quad \int K_m(t) dt =1.   
\]
From these we obtain the following estimate
\[
|(2^{-m} \partial_t)^\alpha P_{\geq m} \mathds{1}_{I_k}(t)| \les_\alpha \la 2^m(d(t, I_k) + d(t,I_k^c)) \ra^{-N},
\]
which allows us to estimate
\[
\|  \tilde \chi_{I_l} P_{\geq m} \mathds{1}_{I_k} \|_{V^1} \les |I_l| \| \partial_t ( \tilde \chi_{I_l} P_{\geq m} \mathds{1}_{I_k}) \|_{L^\infty} \les  \la 2^{m} d(I_l,I_k^c) \ra^{-N}.
\]
Using the inclusion $V^1 \subset U^2$ we conclude with
\[
\| \tilde \chi_{I_l} P_{\geq m} \mathds{1}_{I_k}  \phi_k \|_{U^2} \les  \la 2^{m} d(I_l,I_k^c) \ra^{-N} \| \phi_k\|_{L^2}. 
\]
Summing with respect to the intervals $I_k$ gives the following
\begin{equation} \label{aux2}
\| \tilde \chi_{I_l} P_{\geq m} (\sum_k \mathds{1}_{I_k}  \phi_k) \|_{U^2} \les  \sum_k \la 2^{m} d(I_l,I_k^c) \ra^{-N} \| \phi_k\|_{L^2},
\end{equation}
from which we obtain 
\[
\begin{split}
\sum_l \| \tilde \chi_{I_l} P_{\geq m} (\sum_k \mathds{1}_{I_k}  \phi_k) \|^2_{U^2} & \les  \sum_l \left(\sum_k \la 2^{m} d(I_l,I_k^c) \ra^{-N} \| \phi_k\|_{L^2} \right)^2 \\
& \les \sum_l \sum_{k_1,k_2} \la 2^{m} d(I_l,I_{k_1}^c) \ra^{-N} \la 2^{m} d(I_l,I_{k_2}^c) \ra^{-N} \| \phi_{k_1}\|_{L^2}  \| \phi_{k_2}\|_{L^2}  \\
& \les \sum_{k_1,k_2} \la 2^{m} d(I_{k_1},I_{k_2}) \ra^{-N}  \| \phi_{k_1}\|_{L^2}  \| \phi_{k_2}\|_{L^2}  \\
& \les \sum_{k}  \| \phi_{k}\|^2_{L^2}  \les  \|a_2\|_{U^2}^2. 
\end{split}
\]
This concludes the proof.

\end{proof}

We remark the following consequence of \eqref{aux2}. If $I_0=[t_0,+\infty)$, the following holds true
\begin{equation} \label{aux3}
\sum_{I \in \mathcal{I}_{-m}}  \la 2^{m} d(I,t_0)\ra^N \| \mathds{1}_I  P_{\geq m}   \mathds{1}_{I_0} \phi \|_{U^2} \les_N \|\phi\|_{L^2}.
\end{equation}

\subsection{$U^p,V^p$ adapted to the Schr\"odinger flow.}

Consider the  linear Schr\"odinger equation 
\begin{equation}
(i \partial_t + \Delta) \psi = 0, \quad  \psi(0,x)= \psi_0(x),
\end{equation}
where $\psi: \R \times \R^n \rightarrow \C$. The solution $\psi(t)=e^{it\Delta} \psi_0$ satisfies $e^{-it \Delta} \psi(t)=\psi_0 = \mathds{1}_{\R}(t) \psi_0(x)$ is a constant in time function, just like an atom in $U^p$ (to be more precise we would have to consider the solution on some interval $[T,+\infty)$ for some $T\in \R$ to obtain an atom with the standard normalization). Such a nice representation is lost when we seek solutions for a nonlinear PDE, but instead a more realistic expectation is that $e^{-it \Delta} \psi(t)$ belongs to $U^2$ or $V^2$. This motivates the definition 
\[
U^p_{\Delta}= \{ u: e^{-it \Delta} u \in U^p \}, \quad V^p_{\Delta}= \{ u: e^{-it \Delta} u \in V^p \}. 
\]
Let us highlight some of the key properties of these spaces. We note the straightforward property $e^{it \Delta} (\chi_I \phi) = \chi_I  e^{it \Delta}  \phi$, thus the atoms in $U^2_\Delta$ are 
\[
a(t)=\sum_{k=1}^K \mathds{1}_{[t_{k-1},t_{k})}(t) e^{it \Delta}  \phi_k , \quad \sum_{k=1}^K\|\phi_k\|_{L^2(\R^n)}^p=1.
\]
We recall the standard Strichartz estimate for the linear equation
\[
\|e^{it\Delta} f \|_{L^q_t L^r_x} \les \|f\|_{L^2},
\]
for any pair $(q,r)$ satisfying $2 < q \leq \infty$ and $\frac{2}q + \frac{n}r=\frac{n}2$; we do not need the end-point in this paper. 
 The following result is standard
\begin{equation} \label{V2Str}
\| u \|_{L^q_t L^r_x} \les \|u\|_{V^2_\Delta}. 
\end{equation}
This is a consequence of the fact that $V^2_\Delta \subset U^q_\Delta$ (since we chose $q>2$) and the Strichartz estimate is straightforward for atoms in $U^q_\Delta$. 

Let us recall the definition of the standard $X^{0,b,q}$ spaces in this context. We begin with the modulation localization operator. For $j \in \Z$ we define
\[
\mathcal{F}[Q_{j}f](\tau,\xi)=\rho_j(\tau +  \xi^2)\mathcal{F}f(\tau,\xi),
\]
where we recall that $\rho_j$ is the function localizing $|\tau + \xi^2| \approx 2^j$, see Section \ref{subsect:not} for details.

For $1\leq p\leq \infty$, $b \in \R$, we define the following semi-norm
\[
\|f\|_{\dot{X}^{0,b,p}}=\big\|\big(2^{bj}\|Q_{j}
f\|_{L^2}\big)_{j \in \Z }\big\|_{\ell^p}, 
\]
with the standard modification when $p=\infty$. We note that $b$ quantifies the time regularity after "pulling back" the linear flow, that is the time regularity of the object $e^{-it \Delta} u$. Indeed 
\[
\|f\|_{\dot{X}^{0,b,p}} \approx \|  e^{-it \Delta} f \|_{(\dot B^{b}_{2,p})_t L^2_x},
\]
where $(B^{b}_{2,p})_t L^2_x$ stands for the space of functions $ g: \R \times \R^n$ with $\|g(t, \cdot)\|_{L^2_x(\R^n)} \in \dot B^{b}_{2,p}$ (as a function of $t$) and with the associated semi-norm. 

The following embeddings are well known (see for instance $(2.5)$ in \cite{KT1}) :
\begin{equation} \label{UVX}
X^{0,\frac12,1} \subset U^2_\Delta \subset V^2_\Delta \subset X^{0,\frac12,\infty};
\end{equation}
these are a direct consequence of \eqref{BUV} and the definitions of these spaces. 

In the previous section we noted that $P_{\geq k}(D_t), P_{< k}(D_t)$ are bounded operators on $U^p,V^p$; as a consequence $Q_{\leq}, Q_{\geq}$ are bounded operators on  $V_\Delta^2,U_\Delta^2$.

Moreover, we also have
\begin{equation} \label{V2L2}
\| Q_{\geq k}(D_t) f \|_{L^2(\R \times \R^{n})} \les 2^{-\frac{k}2} \| f \|_{V^2_\Delta},
\end{equation}
which follows from \eqref{UVX} (the inclusion into $\dot X^{\frac12}_{2,\infty}$ part). 

In what follows we rely on the following simple identity
\[
e^{-it\Delta} Q_m ( e^{it \Delta} f )=   P_m (D_t) f;
\]
a similar identity holds true if $Q_m$ is replaced by $Q_{ \leq m}$ or $Q_{\geq m}$. As a consequence of 
 Lemma \ref{U2dec} we have the following result. 

\begin{lemma} \label{U2hm}
i) Assume that $u \in U^2_\Delta$ and $m \in \Z$.  Then we have the following 
\begin{equation} \label{U2DIX}
\sum_{I  \in \mathcal{I}_{-m} } \| \mathds{1}_I  Q_{\geq m} u \|_{U^2_\Delta}^2 \les \|u\|^2_{U^2_\Delta}.  
\end{equation}

ii) In addition to the setup in i) assume that $u$ is localized at frequency $2^k$. Then
\begin{equation} \label{U2IXhm}
\sum_{I \in \mathcal{I}_{-m}}  \sum_{q \in C_{-m+k}} \| \chi_q e^{-it\Delta} \mathds{1}_I Q_{\geq m} u \|_{U^2}^2 \les \|u\|^2_{U^2_\Delta}.  
\end{equation}

iii) If $I_0=[t_0,+\infty)$ then the following holds true
\begin{equation} \label{U2fd}
\sum_{I \in \mathcal{I}_{-m}}  \la 2^{m} d(I,t_0)\ra^N \| \mathds{1}_I P_{\geq m}   \mathds{1}_{I_0} e^{it\Delta} \phi \|_{U^2_\Delta} \les_N \|\phi\|_{L^2}.
\end{equation}

\end{lemma}

Remark. The above is stated for the standard set of time intervals $\mathcal{I}_{-m}$, but it holds true for any set of intervals $\{I_l\}_{l \in \Z}$ is a set of intervals in $\R$ of length $\approx 2^{-m}$ with mutually disjoint interior. The set of spatial cubes $C_{-m+k}$ in part ii) can be similarly generalized. 

\begin{proof} i) is a corollary of Lemma \ref{U2dec}. ii) follows from i), the speed of propagation for free waves and Proposition \ref{Pref}. Finally iii) follows from \eqref{aux3}. 

\end{proof}

Similarly to \eqref{DU},  we define the space
\begin{equation} \label{DUD}
DU^p_\Delta = \{ (i \partial_t + \Delta) u: u \in U^p_\Delta \}, \quad DV^p_\Delta = \{ (i \partial_t + \Delta) u: u \in V^p_\Delta \}
\end{equation}
with the induced norm; here $\partial_t, \Delta $ are distributional derivatives. We note that since $u = e^{it\Delta} \tilde u, \tilde u \in U^2$, then $(i \partial_t + \Delta) u = e^{it\Delta} i \partial_t \tilde u$ and this justifies the use of notation. Next we adapt the statements in Theorem \ref{DualUV}  to the new context.  

\begin{theorem} \label{LDUD} If $1 < p,q< \infty, \frac1p+\frac1q=1$ the spaces $DU^p_\Delta$ and $DV^p_\Delta$ can be characterized as the space of distributions for which the following norms are finite:
\begin{equation} 
\|f\|_{DU^p_\Delta}=\sup_{ \phi \in C_0^\infty: \| \phi\|_{V^q_\Delta} \leq 1}  \left| \int \la  f , \phi \ra_{L^2} dt \right|,
\quad  
\|f\|_{DV^p}=\sup_{ \phi \in C_0^\infty: \| \phi\|_{U^q_\Delta} \leq 1}  \left| \int \la  f , \phi \ra_{L^2} dt \right|,
\end{equation}
\end{theorem}

Just as we had a remark after Theorem \ref{DualUV}, we have the corresponding remark in place in the new context. 

\begin{r1}
In practice, the distributions $f$ that we test will have additional properties such as $f \in L^1_{t, loc} L^2_x$, and, as a consequence, the associated $u$ with the property $f = (i\partial_t +\Delta) u$ comes with a priori continuity properties $u \in C_t L^2_x$. In order to uniquely identify $u$, one has to impose a normalization, such as the standard one in $U^p_\Delta$, $\lim_{t \rightarrow a+} u(t)=0$, or the standard one in $V^p_\Delta$, $\lim_{t \rightarrow b-} u(t)=0$, or the one imposed by the initial data $u(t_0)=u_0$.
\end{r1}

\subsection{Augmenting $V^2_\Delta$ with refined mass and local smoothing estimates.}

The space $V^2_\Delta$ provides a lot of information that is available for free solutions to the Schr\"odinger equation, but it misses two basic pieces of information: 

\begin{itemize}

\item a refined, local in time, version of the standard mass estimate $L^\infty_t L^2_x$ that reflects the speed of propagation. 

\item local smoothing. 

\end{itemize}

Both these structures are available in the smaller space $U^2_\Delta$; we will establish this in the next section. We do not imply that these are the only missing pieces of information from $V^2_\Delta$, but available for say $U^2_\Delta$ or $\dot X^{0,\frac12,1}$, but rather that these are the missing pieces that play a key role in our work.  There is an additional piece of information that we add to the above:

\begin{itemize}

\item a further refined version of the aforementioned structures in the context of high modulations. 

\end{itemize}

The later structure is particular to a resolution type structure, since global free waves have zero modulation. 

We begin with the refined version of the standard $L^\infty_t L^2_x$ estimate. For a solution that is localized the frequency $\approx 1$ the speed of propagation is $\approx 1$, and within a time interval of size $\approx 1$ there is no meaningful transfer of mass between cubes of side-length $\approx 1$; the evolution is mostly confined to such cubes. This is well quantified by the following norm:
\[
\| v \|_{E_0} = \sup_{I \in \mathcal{I}_0} \left( \sum_{q \in C_0} \| \chi_{I}(t) \chi_{q}(x) v \|^2_{L^\infty_t L^2_x} \right)^\frac12
\]
where  we recall that $\mathcal{I}_k=\{ [2^{k}l,2^{k}(l+1)]: l \in \Z\}$ is the set of intervals of length $2^{k}$ covering $\R$ and $C_k$ is the set of cubes of side-length $2^{k}$ on the physical side covering $\R^n$.  

The general structure (adapted to localization at frequency $\approx 2^k$) is obtained by rescaling:
\[
\| v \|_{E_k} = \sup_{I \in \mathcal{I}_{-2k}} \left( \sum_{q \in C_{-k}} \| \chi_{I}(t) \chi_{q}(x) v \|^2_{L^\infty_t L^2_x} \right)^\frac12.
\]
We refine this construction in the context of high modulation components as follows. Given $m \leq 2k$ we let
\[
\| v \|_{E_{k}^{\geq m}} = \left( \sum_{I \in \mathcal{I}_{-m}, q \in C_{-m+k}}  \| \chi_{I }(t) \chi_{q}(x) v \|^2_{E_k} \right)^\frac12.
\]
$E_{k}^{\geq m}$ records an improvement over $E_k$ by bringing  additional $l^2$ type summation across time intervals of size $\approx 2^{-m}$ and across space cubes of size $2^{-m+k}$.  The improved time structure is possible because we use them to measure functions at modulation $\geq 2^m$ and this is the "elliptic regime" for the equation. The improved spatial structure is due to the finite speed of propagation and the time scales being used. We note that something along these lines was employed by Koch-Tataru in \cite{KT1} in the particular case $m=2k$. 

Next we introduce similar type of information at the level of local smoothing.  Given a vector $\e \in \S^{n-1}$ we let  
\[
\|v\|_{E_{k,\e}} = 2^{\frac{k}2} \|v \|_{L^\infty_{x_\e} L^2_{t, \bar x_\e}}.
\]
These spaces are used to measure the components of functions $f$ with the property that their Fourier transform is localized in the region $|\xi| \approx 2^k, |\xi \cdot \e| \approx 2^k$. 

For high modulation we define 
\[
\|u\|_{E_{k,\e}^{\geq m}}= \left( \sum_{I \in \mathcal{I}_{-m}, q \in C_{-m+k}} \| u \|^2_{E_{k,\e}(I \times q)} \right)^\frac12. 
\]
Next, we collect of the above structures. For fixed $k \in \Z$ we define 
\[
X_k = E_k  \cap \cap_{\e \in V^n}  E_{k, \e}
\]
with the norm
\[
\| v \|_{X_{k}} = \|v\|_{E_k} + \max_{\e \in V^n} \| P_{k,\e} v \|_{E_{k,\e}}  .
\]
We note that the only reason we use finitely many frames ($\e \in V^{n}$) versus the full family ($\e \in \S^{n-1}$) is to have simpler formulations when it comes to duality arguments. 

For $m \leq 2k$, we introduce 
\[
X_{k}^{\geq m}= E_{k}^{\geq m} \cap \cap_{\e \in V^n}  E_{k, \e}^{\geq m},
\]
whose the norm is given by:
\[
\| v \|_{X_{k}^{\geq m}} =  \| v \|_{E_k^{\geq m}} 
+ \max_{\e \in V^n} \|  P_{k,\e} v \|_{E_{k}^{\geq m}}.
\]
The full structure $X_{k,m}= V^2_\Delta \cap X_k \cap X_k^{\geq m}$ has the norm
\[
\| v \|_{X_{k,m}} = \|v\|_{V^2_\Delta} + \|v\|_{X_k} + \| Q_{\geq m} v\|_{X_k^{\geq m}}. 
\]
For later purposes, it is useful to let $X_{k,-\infty}= V^2_\Delta \cap X_k$. 

We now introduce the formal "dual" counterparts of the structures defined above. The precise formulation of duality will be made in the next section in Lemma \ref{Ldual}. The dual of $E_k$ is the normed space $F_k$, whose norm is defined by
\[
\| v \|_{F_k} = \sum_{I \in \mathcal{I}_{-2k}} \left( \sum_{q \in C_{-k}} \| \chi_{I}(t) \chi_{q}(x) v \|^2_{L^1_t L^2_x} \right)^\frac12, 
\]
and the dual of $E_k^{\geq m}$ is $F_{k}^{\geq m}$, whose norm is defined by
\[
\| v \|_{F_{k}^{\geq m}} = \left( \sum_{I \in \mathcal{I}_{-m}}  \sum_{q \in C_{-m+k}} \| \chi_{I}(t) \chi_{q}(x) v \|^2_{F_k}  \right)^\frac12.
\]
Similarly, for $\e \in V^n$ we let
\[
\|v\|_{F_{k,\e}} = 2^{-\frac{k}2} \| v \|_{L^1_{x_\e} L^2_{t,x_\e}}.
\]
For high modulation components we introduce the corresponding structure as follows:
\[
\| v \|_{F_{k,\e}^{\geq m}} = \left( \sum_{I \in \mathcal{I}_{-m}}  \sum_{q \in C_{-m+k}}  \| \chi_{I}(t) \chi_q(x) v \|^2_{F_{k,\e}} \right)^\frac12. 
\]
The formal dual of $X_k$ is $Y_k=F_k + \sum_{\e \in V^n} F_{k,\e}$ whose norm is defined by 
\[
\|f\|_{X_{k}}= \inf_{f=f_1+f_2}   \|f_1\|_{F_{k}}   +  \|f_2\|_{\sum F_{k,\e}}, 
\quad \| f \|_{\sum F_{k,\e}} = \inf_{f = \sum_{\e \in V^n} \tilde P_{k,\e} f_{\e}}  \| \tilde P_{k,\e} f_\e\|_{F_{k,\e}},
\]
We also let $Y_k^{\geq m} = F_{k}^{\geq m} + \sum_{\e \in V^n} F_{k,\e}^{\geq m}$, whose norm is defined in a smilar way.

Finally we combine all the above structures in one 
\[
Y_{k,m}=DU^2_\Delta + Y_k + Y_k^{\geq m},
\]
whose norm is defined in the standard fashion with an additional constraint:
\[
\|f\|_{Y_{k,m}}= \inf_{f=f_1+f_2+ f_3}  \|f_1\|_{DU^2_\Delta} + \|f_2\|_{Y_{k}}   + \| f_3\|_{Y_{k}^{\geq m}},
\]
where we assume that $f_3$ is localized in modulation as follows $Q_{\geq m-10} f_3 = f_3$. We highlight that the use of localization operators on some of the elements in the decomposition is related to the use of localization operators in the original structures whose formal dual we describe here.

\section{Local smoothing estimates and a structural property of the  $U^2_\Delta$ space}

In this section we recall the basic local smoothing estimates from \cite{BIKT-1}. Then we analyze how these structures are related to the $U^2_\Delta$ and $V^2_\Delta$ spaces. The final result of this section establishes the basic fact that functions  $U^2_\Delta$ belong to the structures introduced in the previous section.

Recall from the notation section \ref{subsect:not} that we have introduced the operators $P_{k,\e}$ for $\e \in \S^{n-1}$ which localize, on the frequency side, in the region $|\xi \cdot \e| \approx |\xi|$ and the norms $L^\infty_{x_\e} L^2_{\bar x_\e,t}, L^1_{x_\e} L^2_{\bar x_\e,t}$. From \cite[Lemma 3.2]{BIKT-1} we have the following result
\begin{lemma}
 If $f\in L^2(\R^n)$, $k\in\Z$, and $\e\in\mathbb{S}^{n-1}$ then
   \begin{equation}
     \|e^{it \Delta} P_{k,\e} f\|_{L^\infty_{x_\e} L^{2}_{t,\bar x_\e}} \lesssim  
     2^{ -k/2} \|f\|_{L^2}.
     \label{limaxa}\end{equation}
\end{lemma}
An immediate consequence of this is that functions in  $U^2_\Delta$ enjoy a similar estimate
 \begin{equation}
     \| P_{k,\e} u\|_{L^\infty_{x_\e} L^{2}_{t,\bar x_\e}} \lesssim  
     2^{ -k/2} \|u\|_{U^2_\Delta},
     \label{U2LE}\end{equation}
which can be easily checked at the level of atoms in $U^2_\Delta$. However, it seems unlikely that functions in $V^2_\Delta$ enjoy such estimates and this is one of the reasons we needed to augment the structure of the $V^2_\Delta$ with local smoothing estimates.

Next we introduce the corresponding estimates for the inhomogeneous equation
\begin{equation} \label{inhLE}
(i \partial_t + \Delta) \psi = f, \quad  \psi(0,x)= \psi_0.
\end{equation}
We recall the following result from \cite{BIKT-1}.

\begin{lemma} \label{LEinh} i) Assume that $f \in L^1_t L^2_x$. Then for any $\e_1 \in S^{n-1}$ the solution to \eqref{inhLE} satisfies
\begin{equation}
2^{\frac{k}2} \| P_{k, \e_1} \psi \|_{L^\infty_{x_{\e_1}} L^2_{t, \bar x_{\e_1}}} 
\les \|\psi_0 \|_{L^2}+  \|f\|_{L^1_t L^2_x}. 
\end{equation}

ii) Assume that, for some $\e \in \S^{n-1}$, $f \in L^1_{x_\e} L^2_{t, \bar x_\e}$ and $\hat f$ is localized in the region $|\xi| \approx 2^k$ and $|\xi \cdot \e| \ges 2^k$. Then for any $\e_1 \in S^{n-1}$ the solution to \eqref{inhLE} satisfies
\begin{equation}
2^{\frac{k}2} \| P_{k, \e_1} \psi \|_{L^\infty_{x_{\e_1}} L^2_{t, \bar x_{\e_1}}} 
\les \|\psi_0 \|_{L^2}+ 2^{-\frac{k}2} \|f\|_{L^1_{x_\e} L^2_{t, \bar x_\e}}. 
\end{equation}
In addition, the solution satisfies the standard mass estimate
\begin{equation}
 \| \psi \|_{L^\infty_t L^2_x} 
\les \|\psi_0 \|_{L^2}+ 2^{-\frac{k}2} \|f\|_{L^1_{x_\e} L^2_{t, \bar x_\e}}. 
\end{equation}
\end{lemma}
The proof of this result can be found in \cite[Lemma 7.4]{BIKT-1}. A direct estimate of the above results is the classical local smoothing estimate.
\begin{lemma} \label{LLE} Assume $q$ is a (spatial) cube of side-length $L$.  If $f$ is localized at frequency $2^k$, the following estimate holds true
\begin{equation} \label{LEloc}
\| e^{it \Delta}  f \|_{L^2_{t,x}(\R \times q)} \les L^\frac12 2^{-\frac{k}2} \| f \|_{L^2}.
\end{equation} 
If $u$ is localized at frequency $2^k$, the following estimate holds true
\begin{equation} \label{LEX}
\| u \|_{L^2_{t,x}(\R \times q)} \les L^\frac12 2^{-\frac{k}2} \|u\|_{X_k}.
\end{equation}

\end{lemma}

\begin{proof} We use the decomposition \eqref{decej}
\[
 f = \sum_{j: \e_j \in V^n} f_j , \quad   \tilde P_{k,\e_j} f_j = f_j. 
\]
and note that, by using \eqref{limaxa} for each $j$ we obtain
\[
\| e^{it \Delta}  f_j  \|_{L^2(q \times \R)} \les  L^\frac12 \| f_j \|_{L^\infty_{x_{\e_j}} L^2_{t, \bar x_{\e_j}}} \les  L^\frac12 2^{-\frac{k}2} \| f \|_{L^2},
\]
from which \eqref{LEloc} follows by summing with respect to $j$. \eqref{LEX} follows in a similar manner since $X_k$ provides control for the local smoothing norms. 
 
\end{proof}

The next result is a straightforward upgrade to \eqref{U2LE}. 

\begin{lemma} \label{LLE2} Assume that $u$ is localized at frequency $2^k$. Then the following holds true
\begin{equation} \label{LElocU2}
\| u \|_{X_k}  \les \| u \|_{U_\Delta^2}.
\end{equation} 
\end{lemma}

\begin{proof} The estimate for $\| u \|_{X_k} $ contains two parts. The local smoothing estimates have been established in \eqref{U2LE}. The refined mass estimate in $E_k$ is obvious for free solutions. Indeed, by rescaling we can assume $k=0$ and due to the translation invariance it suffices to consider $I=[0,1]$. Here the kernel $e^{it \Delta} P_0$ decays fast away from the region $|x| \les 1$, thus leading to the desired estimate. Next it is easy to lift the estimate from free solutions to atoms in $U^2_\Delta$. This concludes the argument.

\end{proof}

In Lemma \ref{LEinh} part ii) we recalled basic properties of the solution to the inhomogeneous equation \eqref{inhLE} with specific structure on the forcing  $f$. In the next result we upgrade the properties of solutions by establishing that they belong to $V^2_\Delta$. 
 
\begin{lemma} \label{U2lat} Assume that, for some $\e \in \S^{n-1}$, $f$ has the properties listed in Lemma \ref{LEinh}, part ii). Then the following holds true
\begin{equation}
\| \psi \|_{V^2_\Delta} \les  \|\psi_0 \|_{L^2}+ 2^{-\frac{k}2} \|f\|_{L^1_{x_\e} L^2_{t, \bar x_\e}}. 
\end{equation}

\end{lemma}

\begin{proof} 

The proof is based on the duality result from Theorem \ref{DualUV}, part b). Note that, from the previous Lemma, we already know that $\psi \in L^\infty_t L^2_x \cap C_t L^2_x$ which is needed in order to use the duality characterization of $V^2_\Delta$ which requires the ruled function setup. The continuity in time is obvious when the $f \in L^1_t L^2_x$; for the case when $f \in L^1_{x_\e} L^2_{t, \bar x_\e}$ one relies on the representation formula in Lemma 7.5 from \cite{BIKT-1}. 

For $u \in U^2_\Delta \cap \C_0^\infty$ (as in the setup of Theorem \ref{LDUD}), we have the following
\[
|\la f, u \ra |= |\la \tilde P_{k,\e}  f, u \ra| =  |\la \tilde P_{k,\e}  f,  P_{k,\e}  u \ra| \les \| f \|_{L^1_{x_\e} L^2_{t, \bar x_\e}} 
\| P_{k,\e}  u \|_{L^\infty_{x_\e} L^2_{t, \bar x_\e}} \les \| f \|_{L^1_{x_\e} L^2_{t, \bar x_\e}} 
2^{-\frac{k}2} \|  u \|_{U^2_\Delta},
\]
thus invoking Theorem \ref{LDUD} gives $f \in DV^2_\Delta$ and, as a consequence $\psi \in V^2_\Delta$. To be more precise, what we obtain in fact is a solution $\tilde \psi \in V^2_\Delta$ to the inhomegenous equation $(i\partial_t + \Delta ) \psi=f$ whose initial data at $+\infty$ is $0$ and satisfying $\|\tilde \psi \|_{V^2_\Delta} \les 2^{-\frac{k}2} \|f\|_{L^1_{x_\e} L^2_{t, \bar x_\e}}$. The difference $\psi-\tilde \psi$ solves the homogenous equation with initial data at $t=0$ being $\psi_0-\tilde \psi(0)$ and the statement of the Lemma is now obvious. 

\end{proof}

So far we have established that functions in $U^2_\Delta$ and localized at frequency $2^k$ have all properties in $X_{k,m}$ with one exception: being in $E_k^{\geq m}$; the next results validates this property too. 

\begin{lemma} \label{U2hm2}
 Assume that $u \in U^2_\Delta$ and is localized at frequency $2^k$, and $m \in \Z$.  Then the following holds true 
\begin{equation} \label{U2tE}
 \| Q_{\geq m} u \|_{X_k^{\geq m}} \les \|u\|_{U^2_\Delta}.  
\end{equation}
\end{lemma}

\begin{proof} We begin with the estimate in $E_k^{\geq m}$. We rely on \eqref{U2IXhm} which we recall here for convenience:
\[
 \sum_{I \in \mathcal{I}_{-m}} \sum_{q \in C_{-m+k}} \| \chi_q e^{-it\Delta} \chi_{I}(t) Q_{\geq m} u \|_{U^2}^2 \les \|u\|^2_{U^2_\Delta}.  
\]
This estimate is very close to what we need to establish, with two main obstacles: the lack of commutativity between $e^{-it\Delta}$ and $\chi_q$ and the fact that $e^{\pm it\Delta}$ does not preserve the compact (in space) support property of $\chi_q$. The argument below essentially deals with these issues so as to obtain the desired inequality. 

Given the $l^2$ structure with respect to the intervals in $I \in \mathcal{I}_{-m}$, it suffices to consider a single interval $I  \in \mathcal{I}_{-m}$ and aim to prove
\begin{equation} \label{U2tE2}
 \sum_{q \in C_{-m+k}} \| u \|^2_{E_k(q \times I)} \les \sum_{q \in C_{-m+k}}  \| \chi_q e^{-it\Delta} \chi_{I}(t) u \|_{U^2}^2;
\end{equation}
we tacitly replaced $Q_{\geq m} u$ with $u$, since the modulation information has been already used. Since the estimate is time translation invariant, it suffices to consider the case $I=[0,2^{-m}]$. We let $u_q = \chi_q e^{-it\Delta} \chi_{I}(t) u$ and this gives a decomposition 
\[
u = \sum_{q \in C_{-m+k}} \tilde P_k e^{it\Delta} u_{q},
\]
where $u_{q}$ is supported in $q \times I$ and, invoking Proposition \ref{Pref}, 
\[
\sum_{q \in C_{-m+k}} \| u_{q} \|_{U^2}^2 \les \|u\|^2_{U^2_\Delta}.  
\]
Since $e^{it\Delta} u_{q} \in U^2_\Delta$, we use \eqref{LElocU2} to obtain
\[
\| e^{it\Delta}  \tilde P_k u_{q} \|_{E_k} \les  \| u_{q} \|_{U^2}. 
\]
On the other hand, within the time interval $I$, the kernel of $e^{it\Delta}  \tilde P_k$ is concentrated in the region $|x| \les 2^{-m+k}$, and this is why, for any $q,q' \in C_{-m+k}$, we claim the stronger estimate
\begin{equation} \label{uqconc}
\| e^{it\Delta}  \tilde P_k u_{q} \|_{E_k(q')} \les  \la 2^{m-k} d(q,q') \ra^{-N} \| u_{q} \|_{U^2}. 
\end{equation}
It is now obvious that \eqref{U2tE2} follows from \eqref{uqconc}, and this in turn proves our main claim  \eqref{U2tE}.  

It remains to prove \eqref{uqconc}. Let us first establish this for when $u_q = \mathds{1}_{J}(t) \phi$ (a piece of an atom) which in turn creates a single free solution $e^{it\Delta}  \tilde P_k \phi$, where $\phi$ is localized in $q$ and the free solution is truncated in the time interval $J \subset I$. Since the problem is space translation invariant, it suffices to assume that $q$ is the cube with one of the corners being $0=(0,..,0)$.  
We begin with the following commutator identity 
\[
(\frac{x+2it\nabla}{2^{-m+k}})^{N} e^{it\Delta}  \tilde P_k \phi = e^{it\Delta}  (\frac{x}{2^{-m+k}})^N  \tilde P_k \phi.
\]
Then we note that, in $I$ and in the support of $\tilde P_k$, $t \nabla \tilde P_k$ has a kernel whose norm in $L^1_x$ is bounded by $\les 2^{-m+k}$. Thus we obtain
\[
\| (\frac{x}{2^{-m+k}})^N  e^{it\Delta}  \tilde P_k \phi  \|_{\Ec_k} \les_N  \| \tilde P_k \phi \|_{L^2} + \| (\frac{x}{2^{-m+k}})^N  \tilde P_k \phi \|_{L^2}.
\]
Last we note that
\[
\| (\frac{x}{2^{-m+k}})^N  \tilde P_k \phi \|_{L^2} \les_N \| \phi \|_{L^2},
\]
which is a direct consequence of the fact that the kernel of $\tilde P_k$ (which is in $L^1$) is highly concentrated in the region $|x| \les 2^{-k} \les 2^{-m+k}$ (since $m \leq 2k$) and $\phi$ is supported in $q$. As a consequence we obtain
\[
\| (\frac{x}{2^{-m+k}})^N  e^{it\Delta}  \tilde P_k \phi  \|_{E_k} \les_N  \| \tilde P_k \phi \|_{L^2},
\]
from which \eqref{uqconc} follows in this particular case. Next we consider an atom in $U^2$ 
\[
a(t)=\sum_{k=1}^K \mathds{1}_{[t_{k-1},t_{k})}(t) \phi_k , \quad \sum_{k=1}^K\|\phi_k\|_{L^2(\R^n)}^2=1,
\]
where $t_0,t_K \in I$ and $\phi_k$ are supported in $q$.  It is an easy exercise to see that \eqref{uqconc} follows from the previous case (of free solutions). Given the atomic structure of $U^2$, this concludes the proof of \eqref{uqconc} for the $X_{k}^{\geq m}$ component. Lastly, the same localization argument above using commutators can be used to establish
\[
\| Q_{\geq m} u \|_{E_{k,\e}^{\geq m}} \les \|u\|_{U^2_\Delta},  
\]
for any $\e \in \S^{n-1}$. This concludes the proof of the Lemma.

\end{proof}

We can now conclude this section by stating that functions in $U^2_\Delta$ and localized at frequency $2^k$ belong to the $X_{k,m}$ structure. As a consequence of Lemma \ref{LLE2} and Lemma \ref{U2hm2}, we have

\begin{corollary} \label{U2Xc}
 Assume that $u \in U^2_\Delta$ is localized at frequency $2^k$, and $m \in \Z, m \leq 2k$.  Then the following holds true 
\begin{equation} \label{U2Xcest}
 \| u \|_{X_{k,m}} \les \|u\|_{U^2_\Delta}.  
\end{equation}
\end{corollary}

\section{Proof of Theorem \ref{MTl}}

In this section we establish the proof of our main linear estimates from Theorem \ref{MTl}. The information provided about $f$ in this Theorem comes in a duality fashion along with some qualitative apriori characterization.  In Lemma \ref{Ldual} we  use the duality to characterize $f$ in a more quantitative way, in that it belongs to a sum of concrete spaces. In Lemma \ref{Pinh} we develop the linear theory where we establish that if $f$ belongs to any of these spaces, then we obtain the corresponding linear  estimates claimed in  \ref{MTl}.

Thus the proof of Theorem \ref{MTl} is broken down into two parts. The first lemma extracts effective information on $f$. 

\begin{lemma} \label{Ldual}
 Assume that  $f \in L^\infty_t L^2_x$ is localized at frequency $\approx 2^{k}$ and that we have the following bound
\begin{equation} \label{fphib}
|\la f, \phi \ra | \leq C \|\phi\|_{X_{k,m}}
\end{equation}
for any $\phi \in X_{k,m}$, where $m \leq 2k$. Then $f \in Y_{k,m}$ and  $\| f \|_{Y_{k,m}} \les C$.
\end{lemma}

The second result establishes the actual linear estimate. 

\begin{lemma}   \label{Pinh}
Consider the inhomogeneous equation 
\begin{equation} \label{LSE}
(i \partial_t + \Delta) u =f, \quad u(0) = 0, 
\end{equation}
where $f \in Y_{k,l}$ is localized at frequency $\approx 2^{k}$ and $l \leq 2k$. Then the following holds true:
\begin{equation}
\sup_{m \leq 2k} \| u \|_{X_{k,m}} \les \|f\|_{Y_{k,l}},
\end{equation}
where we include the case $m=-\infty$ (similarly for $l$) and recall the convention $X_{k,-\infty}= X_k \cap V^2_\Delta$. 
\end{lemma}

It is clear that Theorem \ref{MTl} follows from the above results, modulo the statement in part i); but this has already been established in the previous section in Corollary \ref{U2Xc}, since free solutions belong belong to $U^2_\Delta$. The remainder of this section is dedicated to the proofs of the these two Lemmas. 

\subsection{Proof of Lemma \ref{Ldual}}
In this short section we clarify the use of duality in this paper. Our main goal is to establish the following. With $f$ as in Theorem  \ref{MTl}, we establish that
\begin{equation}
\| f \|_{Y_{k,m}} \les \sup_{\|v\|_{X_{k,m}}=1} |\la v,f\ra | 
\end{equation}
The key step is to clarify the duality between individual components that appear in the definitions of $X_{k,m}$ and $Y_{k,m}$, that is the duality between $E_k$ and $F_k$, the duality between $E_k^{\geq m}$
and $F_{k}^{\geq m}$ etc. Once that is established, we use the standard result that $(E_1 \cap E_2 \cap ... \cap E_m)^*= E_1^*+E_2^*+...+E_m^*$; this can be found in  \cite{Si}, see Theorem 14.9 there. Note that here $E_j$ are meant to be general separated semi-normed spaces, and should not be mistaken for our particular spaces $E_k$.

We begin with explaining the role of apriori information in the use of duality in the simplest setup, that of sequences $(x_n)_{n \in \N}$ and the standard spaces $l^\infty$ and $l^1$. It is well-known fact that the inclusion $l^1 \subset (l^\infty)^*$ is strict, that is $l^1 \ne (l^\infty)^*$; Banach limits are standard examples of elements in $(l^\infty)^*$ which are not in $l^1$. On the other hand, if we have apriori knowledge of the fact that the bounded functional on $l^\infty$ is in fact (representable as) a sequence $x=(x_n), x_n \in \C, \forall n \in \N$, then we can conclude that $x \in l^1$ and
\begin{equation} \label{l1linf}
\|x\|_{l^1} = \sup_{y \in l^\infty: \|y\|_{l^\infty}=1} \la x,y \ra=\sum_{n=1}^\infty x_i y_i. 
\end{equation}
To recap, the key point above is that by testing $x$ with elements in $l^\infty$ suffices to establish that $x \in l^1$ only if we assume some minimal a priori information on $x$. Below, the same principle will be used in the context of Lebesgue spaces.

Similarly the dual of $L^\infty_t L^2_x$ is not $L^1_t L^2_x$; however if we have the apriori information that 
$f \in L^\infty_t L^2_x$ then $f \in L^1_{t,loc} L^2_x$, and $f \in (L^\infty_t L^2_x)^*$ implies that $f \in L^1_t L^2_x$ with
\[
\|f \|_{ L^1_t L^2_x} =  \sup_{g \in L^\infty_t L^2_x: \|g\|_{L^\infty_t L^2_x}=1} \la f,g \ra.  
\]
Now we explain the use of duality between individual components in the structure of $X_{k}^{\geq m}$ and $Y_{k}^{\geq m}$. We begin with the duality between $E_k$ and $F_k$ and recall that
\[
\begin{split}
\| v \|_{E_k} & = \sup_{I \in \mathcal{I}_{-2k}} \left( \sum_{q \in C_{-k}} \| \chi_{I}(t) \chi_{q}(x) v \|^2_{L^\infty_t L^2_x} \right)^\frac12 \\
\| u \|_{F_k} & = \sum_{I \in \mathcal{I}_{-2k}} \left( \sum_{q \in C_{-k}} \| \chi_{I}(t) \chi_{q}(x) u \|^2_{L^1_t L^2_x} \right)^\frac12. 
\end{split}
\]
We claim that, given the additional information $f \in L^\infty_t L^2_x$, the following holds true
\begin{equation} \label{FkEk}
\|f\|_{F_k} \les \sup_{\|v\|_{E_k}=1} |\la f, \phi \ra|. 
\end{equation}
Indeed, for fixed $I \in  \mathcal{I}_{-2k}$ we have the apriori qualitative information $f \in L^1_t L^2_x$, which in particular controls the norm  $\left( \sum_{q \in C_{-k}} \| \chi_{I}(t) \chi_{q}(x) f \|^2_{L^1_t L^2_x} \right)^\frac12$. Moreover, we obtain the bound
\[
\|\chi_I f\|_{F_k} \les \sup_{\|\phi\|_{E_k}=1} |\la \chi_I \phi, \chi_I f \ra|,
\]
which is a consequence of $l^2$ being self-dual. With the apriori knowledge that $\chi_I f \in F_k$, we can now invoke \eqref{l1linf} and obtain \eqref{FkEk}. 

The duality between $E_k^{\geq m}$ and $F_k^{\geq m}$ is justified in a similar manner.

The duality between $E_{k,\e}$ and $F_{k,\e}$ is justified in a similar manner, by noticing that $L^2_{t,x}$ controls $L^1_{x_\e, loc} L^2_{t, \bar x_\e}$ for any $\e$. To justify the insertion of the operators $\tilde P_{k,\e}$
in the definition of $Y_k$ we notice that $\la (1- \tilde P_{k,\e} ) f, P_{k,\e} \phi \ra=0$, thus the $E_{k,\e}$ structure is used only when estimating $\tilde P_{k,\e} f$; a more complete formalization is provided at the end of this argument for the use of modulation localization operators.

The duality between $LE_k^{\geq m}$ and $LF_k^{\geq m}$ is justified in a similar manner.

The last duality is the one between $V^2_\Delta$ and $DU^2_\Delta$ stating that
\[
\| f \|_{DU^2_\Delta} \les \sup_{\|\phi\|_{V^2_\Delta}=1} |\la f, \phi\ra|.
\]
Here we begin with the following observation. Since $f \in L^\infty_t L^2_x \subset L^1_{t,loc} L^2_x $, the solution to the equation $(i \partial_t + \Delta) u =f$ is a $C_t L^2_x$ function (hence ruled function). Then we can apply Lemma \ref{LDUD} to conclude $f \in DU^2_\Delta$ and claim the above estimate.

There is just one more aspect that needs clarification: the use of modulation localization $Q_{\geq m-10} f_3=f_3$ for the component that appears in decomposition of $f$ in the definition of $Y_{k}^{\geq m}$. Consider the low modulation part of $f$, that is $Q_{\leq m-10} f$. We have
\[
\la Q_{\leq m-10}  f, \phi \ra= \la Q_{< m-10} f  , Q_{< m-8} \phi + Q_{\geq m-8} \phi \ra = \la Q_{< m-10} f, Q_{< m-8} \phi \ra. 
\]
As a consequence \eqref{fphib} reads
\[
|\la Q_{< m-10} f, \phi \ra | \les C (\|\phi\|_{V^2_\Delta} + \|\phi\|_{X_k}),
\]
thus no high modulation structure from $X_{k}^{\geq m}$ is being used in controlling $Q_{< m-10} f$. This justifies why, in the sum structure from $Y_{k}^{\geq m}$ we use the high modulation structures to measure precisely high modulation components of $f$. A similar argument justifies the use of the projectors $\tilde P_{k,e}$ for the components measured in $F_{k,\e}$. 

This concludes the proof of Lemma \ref{Ldual}. 

\subsection{Proof of Proposition \ref{Pinh}} 

Given the the complexity of the space $Y_{k,l}$, to which $f$ belongs to, we split the argument in several parts, each corresponding to replacing $f$ with one of the elements in the definition of $Y_{k,l}$. 

But first we make a general observation. As a consequence of Corollary \eqref{U2Xc} and the structure of the space $Y_{k,l}$ we obtain that
\[
|\la f, \phi \ra| \les \|f\|_{Y_{k,l}} \| \phi \|_{U^2_\Delta}. 
\]
Then we invoke Lemma \ref{LDUD} to obtain the estimate in $V^2_\Delta$ for $u$. This is why regardless of which component of $Y_{k,l}$ we consider below, we obtain the $V^2_\Delta$ estimate for $u$ and we will not repeat the argument for the estimate in $V^2_\Delta$. 

\subsection{Forcing in $DU^2_\Delta$.} As already stated in Lemma \ref{LDUD}, the solution $u \in U^2_\Delta$
with the appropriate bound. Then Corollary \eqref{U2Xc} supplies the desired bound. 

\subsection{Forcing in $F_{k}$.} Here $f$ is assumed to be localized at frequency $2^k$ and we seek to prove that the solution of \eqref{LSE} satisfies the estimate
\begin{equation}
\| u \|_{X_{k,m}} \les \|f\|_{F_k},
\end{equation}
for any $m \leq 2k$. By rescaling it suffices to consider the case $k=0$. Given the $l^1$ structure across time intervals of $F_{0}$, it suffices to consider the case when $f$ is supported in the time interval $[0,1]$; this is equivalent to replacing $f$ with $\chi_{[0,1]}(t) f$. This also implies that on the interval $(-\infty,0)$ the solution $u$ is identical to $0$, while on the interval $[1,\infty)$ the solution is free in the sense 
$u(t)=e^{i(t-1)\Delta} u(1), t \geq 1$. From the arguments below we obtain that $\|u(1)\|_{L^2_x} \les \|\chi_{[0,1]} f \|_{F_{0}}$ (see \eqref{V2b} below) and thus $\chi_{[1,+\infty]} u \in U^1_\Delta$. Using the fact that  $U^1_\Delta \subset U^2_\Delta \subset V^2_\Delta$, and invoking \eqref{U2Xcest}, gives the bound
\[
\| \chi_{[1,+\infty)} e^{i(t-1)\Delta} u(1) \|_{X_{k,m}} \les \|f\|_{F_k}, \quad \forall m \leq 2k.
\]
Thus it remain to focus  on the interval $[0,1]$. Here we have
\[
u= \sum_{q \in C_0} u_q, \quad  u_q= \int_0^t e^{i(t-s)\Delta} \tilde P_0(\chi_q f(s)) ds, 
\]
where $\tilde P_0$ can be easily inserted in the Duhamel formula based on the localization properties of the solution; we do so because $\chi_q f$ loses the localization at frequencies $\approx 1$.

\vspace{.1in}

\bf The bound in $E_0$. \rm Each $u_q$ solves the corresponding Schr\"odinger equation with forcing $f_q= \chi_q f \in L^1_t L^2_x$, hence the corresponding solution $u_q \in V^1_\Delta$. More precisely we have
\[
e^{-it\Delta} u_q = \int_0^t e^{-is\Delta} \tilde P_0 (\chi_q f(s)) ds,  
\]
and one can easily see that, for any $q,q' \in C_0$,
\begin{equation} \label{uqdec}
\| \chi_{q'} e^{-it\Delta} u_q \|_{V^1} \les d(q',q) \ra^{-N} \| f_q \|_{L^1L^2};
\end{equation}
indeed for $|s| \les 1$, the kernel of the operator $P_0 e^{-is\Delta}$ records fast decay away from the region $|x| \les 1$. 
From this we obtain 
\begin{equation} \label{U1b}
\left( \sum_{q' \in C_0} \| \chi_{q'} e^{-it\Delta} u \|_{V^1}^2 \right)^\frac12 \les \left( \sum_{q \in C_0} \| f_q \|_{L^2}^2 \right)^\frac12. 
\end{equation}
This implies the corresponding bound with $V^1$ replaced $V^2$. Finally, on behalf of Proposition \ref{Pref} we obtain
\begin{equation} \label{V2b}
\|e^{-it\Delta} u\|_{V^2} \les \|\chi_{[0,1]} f \|_{F_{0}}. 
\end{equation}
From \eqref{U1b} we obtain
\[
\left( \sum_{q' \in C_0} \| \chi_{q'} e^{-it\Delta} u \|_{L^\infty_t L^2_x}^2 \right)^\frac12 \les \|f\|_{F_{0}},
\]
which, when combined with the properties of the kernel of $P_0$ (fast decay away from the region $|x| \les 1$) leads to the bound in $E_{0}$: $\| \chi_{[0,1]} u \|_{E_0} \les \|f\|_{F_{0}}$.

\bf The bound in $E_{0,\e}$. \rm Here the analysis is carried for any $\e \in S^{n-1}$, not just those in $V^n$. We start by computing 
\[
(i \partial_t + \Delta) (\tilde \chi_q u)= \tilde \chi_q f + 2 \nabla \tilde \chi_q \cdot \nabla u + \Delta \tilde \chi_q \cdot u;
\]
recall $\tilde \chi_q$ is a bump function adapted to $q$, and with that $\tilde \chi_q =1$ in $q$ and $=0$ in the $(2q)^c$. 
Using \eqref{uqdec}, we obtain the following estimate
\begin{equation} \label{aux51}
\|  \tilde \chi_q f + 2 \nabla \tilde \chi_q \cdot \nabla u + \Delta \tilde \chi_q \cdot u \|_{L^1_t L^2_x} \les \sum_{q' \in C_0} \la d(q,q') \ra^{-N} \| f_{q'} \|_{L^1_t L^2_x}, 
\end{equation}
from which, by using Lemma \ref{LEinh}, we obtain that for any $\e \in \S^{n-1}$ the following holds true
 \begin{equation} \label{aux1}
 \|P_{0,\e} (\tilde \chi_q u) \|_{L^\infty_{x_\e} L^2_{t,\bar x_\e}} \les \sum_{q' \in C_0} \la d(q,q') \ra^{-N} \| f_{q'} \|_{L^1_t L^2_x}. 
 \end{equation}
 We can modify the functions $\tilde \chi_q$ so that they form a partition of unity in $\R^n$ which gives
 \[
\sum_{q \in C_0} P_{0,\e} (\tilde \chi_q u) = P_{0,\e} u.
 \]
 The final piece is the following estimate
 \begin{equation} \label{leqconc}
 \|P_{0,\e} (\tilde \chi_q u) \|_{L^\infty_{x_\e} L^2_{t,\bar x_\e}(q'' \times I)} \les \la d(q,q'') \ra^{-N} \sum_{q' \in C_0} \la d(q,q') \ra^{-N} \| f_{q'} \|_{L^1_t L^2_x},
\end{equation}
Assuming \eqref{leqconc}, we can now close the argument as follows
\[
\begin{split}
 \|P_{0,\e} u \|_{L^\infty_{x_\e} L^2_{t,\bar x_\e}}^2 & \les  \sum_{q \in C_0} \|P_{0,\e} u \|_{L^\infty_{x_\e} L^2_{t,\bar x_\e}(q \times I)}^2 
  \les  \sum_{q \in C_0} \left( \sum_{q'' \in C_0} \|P_{0,\e} (\chi_{q''} u) \|_{L^\infty_{x_\e} L^2_{t,\bar x_\e}(q \times I)} \right)^2 \\
 & \les  \sum_{q \in C_0} \left( \sum_{q'' \in C_0}  \la d(q,q'') \ra^{-N} \sum_{q' \in C_0} \la d(q,q') \ra^{-N} \| f_{q'} \|_{L^1_t L^2_x} \right)^2 \\ 
 & \les \sum_{q' \in C_0} \| f_{q'} \|_{L^1_t L^2_x}^2 \les \|f\|_{F_0}^2. 
\end{split}
\]
To establish \eqref{leqconc}  we use multipliers, just as in the proof of Lemma \ref{U2hm2}. Using the space translation invariance properties of the flow, it suffices to consider the case when $q \in C_0$ is one of the unit cubes touching the origin.  We apply the vector field $(x+2it\nabla)^N$ to the equation to obtain
\[
(i \partial_t + \Delta) (x+2it\nabla)^N (\tilde \chi_q u)= (x+2it\nabla)^N \left( \tilde \chi_q f + 2 \nabla \tilde \chi_q \cdot \nabla u + \Delta \tilde \chi_q \cdot u \right) := \tilde f_{q,N}. 
\]
Given all localizations in place, including $|t| \leq 1$,  and using \eqref{aux51} we obtain
\begin{equation} \label{fqN}
\|  \tilde f_{q,N} \|_{L^1_t L^2_x} \les_N \sum_{q' \in C_0} \la d(q,q') \ra^{-N} \| f_{q'} \|_{L^1_t L^2_x}.
\end{equation}
We apply the projector $P_{0,\e}$ to the equation to obtain
\[
\begin{split}
 & P_{0,\e} (i \partial_t + \Delta) (x+2it\nabla)^N (\tilde \chi_q u)  \\
 = & (i \partial_t + \Delta)  (x+2it\nabla)^N P_{0,\e} (\tilde \chi_q u)  +  (i \partial_t + \Delta)  [P_{0,\e}, (x+2it\nabla)^N ] (\tilde \chi_q u),
\end{split}
\]
from which it follows 
\[
 (i \partial_t + \Delta)  (x+2it\nabla)^N P_{0,\e} (\tilde \chi_q u)  = P_{0,\e}  \tilde f_{q,N} - (i \partial_t + \Delta)  [P_{0,\e}, (x+2it\nabla)^N ] (\tilde \chi_q u).
\]
Let $N=1$. We have the exact commutator identity
\[
[P_{0,\e}, (x+2it\nabla) ]= [P_{0,\e}, x ]= i (\nabla_\xi p_{0,\e})(D),
\]
where $p_{0,\e}$ is the symbol of $P_{0,\e}$. Thus 
\[
(i \partial_t + \Delta)  (x+2it\nabla) P_{0,\e} (\tilde \chi_q u) = P_{0,\e} \tilde f_{q,1} -  i (\nabla_\xi p_{0,\e})(D) \tilde f_{q,0}. 
\]
The kernel of $(\nabla_\xi p_{0,\e})(D)$ is similar to the kernel of $P_{0,\e}$ (bump function adapted to the cube $|x| \leq 1$) and, combined with \eqref{fqN}, this allows us to claim
 \[
\|   P_{0,\e} \tilde f_{q,1} \|_{L^1_t L^2_x}  + \|   i (\nabla_\xi p_{0,\e})(D) \tilde f_{q,0}  \|_{L^1_t L^2_x} \les_N \sum_{q' \in C_0} \la d(q,q') \ra^{-N} \| f_{q'} \|_{L^1_t L^2_x}.
\]

By using Lemma \ref{LEinh}, we obtain that for any $\e \in \S^{n-1}$ the following holds true
 \[
 \| (x+2it\nabla) P_{0,\e} ( \tilde \chi_q u) \|_{L^\infty_{x_\e} L^2_{t,\bar x_\e}} \les \sum_{q' \in C_0} \la d(q,q') \ra^{-N} \| f_{q'} \|_{L^1_t L^2_x}. 
\]
Given the localizations in place we already have
\[
 \| 2it\nabla P_{0,\e} ( \tilde \chi_q u) \|_{L^\infty_{x_\e} L^2_{t,\bar x_\e}} \les \| \nabla P_{0,\e} ( \tilde \chi_q u) \|_{L^\infty_{x_\e} L^2_{t,\bar x_\e}}  \les \sum_{q' \in C_0} \la d(q,q') \ra^{-N} \| f_{q'} \|_{L^1_t L^2_x}. 
 \]
where we use the identity $\nabla P_{0,\e} = \nabla \tilde P_{0,\e}  P_{0,\e}$, the properties of the kernel of $\nabla \tilde P_{0,\e} $ and \eqref{aux1}. From here we conclude with 
 \[
 \| x P_{0,\e} ( \tilde \chi_q u) \|_{L^\infty_{x_\e} L^2_{t,\bar x_\e}} \les \sum_{q' \in C_0} \la d(q,q') \ra^{-N} \| f_{q'} \|_{L^1_t L^2_x}. 
\]
which gives \eqref{leqconc} for $N=1$; here we also use  \eqref{leqconc} for $N=0$. We can then reiterate the process and obtain it for all $N \in \N$. The reason we do not do this in one step is due to the fact that the $[P_{0,\e}, (x+2it\nabla)^N]$ is more complicated and involves lower order terms (relatively to the main one $x^N$) which are estimated inductively.

\vspace{.1in}

\bf The bound in $E_0^{\geq m}$. \rm We note that from the above estimates we already have
\begin{equation} \label{aux6}
\| \chi_{[0,1]} u \|_{E^{\geq m}_{0}} \les \|f\|_{F_0}. 
\end{equation}
Indeed, the solution is supported in $[0,1] \subset [0,2^{-m}]$, thus the $l^2$ summability with respect to the intervals $I \in \mathcal{I}_{-m}$ is trivial. We already have $l^2$ summability with respect to $q \in C_0$ which provides the $l^2$ summability with respect to $q \in C_{-m}$. Now we have to show that this is stable under the modulation cut-off. 

We first dispose of the particular case $ - 100 \leq m \leq 0$. Here $Q_{\geq m}= I - Q_{< m}$ and the kernels of $Q_{< m} P_{0}$  and $Q_{< m} P_{0,\e}$ are easily shown to belong to $L^1_{t,x}$, which, using \eqref{aux6}, establishes
\begin{equation} \label{dhm}
\| Q_{\geq m} \chi_{[0,1]} u \|_{E^{\geq m}_{0}} + \sup_{\e \in \S^{n-1}} \| Q_{\geq m} \chi_{[0,1]} u \|_{E^{\geq m}_{0,\e}}  \les \|f\|_{F_0}.
\end{equation}

For general $m \leq -100$, our goal is to provide an estimate for $Q_{< m} \chi_{[0,1]}(t) u$ in $E^{\geq m}_0$. We have already established $\| \chi_{[0,1]}(t) u \|_{E_0^{\geq m}} \les  \|f\|_{F_0}$, thus it suffices to establish one the following more general equivalent bounds:
\begin{equation} \label{QbEm}
\| Q_{< m} u \|_{E_0^{\geq m}} \les \| u \|_{E_0^{\geq m}}, \quad \| Q_{\geq m} u \|_{E_0^{\geq m}} \les \| u \|_{E_0^{\geq m}},
\end{equation}
for any $u \in E_0^{\geq m}$ which is localized at frequency $\approx 1$. The two are obviously equivalent since $Q_{\geq m} = I - Q_{< m}$. \eqref{QbEm} is a consequence of the following claim
\begin{equation} \label{aux10b}
\|  \chi_I Q_{< m}  \chi_{I'} u \|_{E_0^{\geq m}}  \les_N \la 2^m d(I,I') \ra^{-N}  \| \chi_{I'} u \|_{E_0^{\geq m}},
\end{equation}
for any $I,I' \in \mathcal{I}_{-m}$. The rest of the argument is dedicated to establishing \eqref{aux10b}. We establish the above for the particular case $I'=I_0 \in \mathcal{I}_{-m}$ (recall that $0 \in I_0$) and begin with the identity  
\begin{equation} \label{aux5}
 Q_{< m} \chi_{I_0} u = e^{it\Delta} P_{< m} (D_t) e^{-it\Delta} \chi_{I_0} u =  e^{it\Delta} P_{< m} (D_t) e^{-it\Delta} \chi_{I_0} u. 
\end{equation}
Note that if, in \eqref{aux10b}, we worked with a different $I'$ (than $I_0$), then we would pick $t_{I'} \in I'$ and conjugate the above with $e^{\pm i(t-t_{I'})\Delta}$. The operator $e^{\pm it\Delta}$  is bounded on $L^\infty_t L^2_x$; however we need to preserve the more refined structures in $E^{\geq m}_0$.  Using commutator estimates, 
\begin{equation} \label{aux66}
\|  e^{\pm it\Delta} P_0(D_x) \chi_q \chi_{I_0} u \|_{L^\infty_t L^2_x(q')} \les \la 2^m d(q,q') \ra^{-N} \| \chi_q \chi_{I_0} u \|_{L^\infty_t L^2_x}, 
\quad \forall q,q' \in C_{-m},
\end{equation}
which allows us to conclude 
\[
\| e^{-it\Delta}  \chi_{I_0} u \|_{E_0^{\geq m}} \les \| \chi_{I_0} u \|_{E_0^{\geq m}}.
\]
The kernel $K_{< m}$ of $P_{< m} (D_t)$ satisfies the bound $|K_{< m}(t)| \les 2^{m} \la 2^m t \ra^{-N}$, thus 
\begin{equation} \label{aux8}
\| \chi_I P_{< m} (D_t) e^{-it\Delta}  \chi_{I_0} u \|_{E_0^{\geq m}} \les \la 2^m d(I,I_0) \ra^{-N}  \| \chi_{I_0} u \|_{E_0^{\geq m}}.
\end{equation}
Now let $I=I_0$ in \eqref{aux10b}. Just as earlier, using commutator allows us to claim
 \begin{equation} \label{aux6666}
\| \chi_{I_0} e^{-it\Delta} P_0(D_x) \chi_q v \|_{L^\infty_t L^2_x(q')} \les \la 2^m d(q,q') \ra^{-N} \| \chi_q v \|_{L^\infty_t L^2_x}, 
\quad \forall q,q' \in C_{-m},
\end{equation}
which combined with \eqref{aux8} gives \eqref{aux10b} for $I=I_0$. 

Now we consider general $I$ (not necessarily $=I_0$), where we have a weaker version of \eqref{aux6666}. The basic idea is that within $I$ the flow $e^{it\Delta}$ mixes the mass between various cubes $q \in C_{-m}$, thus booking a polynomial loss (depending on the dimension only) in $\la 2^m d(I,I_0) \ra$; but this can be compensated using the very fast decay from \eqref{aux8}.

We let $d \in \N$ such that $ \la 2^m d(I,I_0) \ra \approx 2^d$. Using commutators just as above, we obtain
\[
\| \chi_{I} e^{-it\Delta} \tilde P_0(D_x) \chi_q v \|_{L^\infty_t L^2_x(q')} \les \la 2^{m-d} d(q,q') \ra^{-N} \| \chi_q v \|_{L^\infty_t L^2_x(q')}, 
\quad \forall q,q' \in C_{-m+d}.
\]
This provides the weaker estimate 
\[
\| \chi_I e^{-it\Delta} \tilde P_0(D_x) v \|_{E_0^{\geq m-d}} \les \|v\|_{E_0^{\geq m-d}} \les \|v\|_{E_0^{\geq m}}. 
\]
To recover the desired estimate we proceed as follows
\[
\| \chi_I e^{-it\Delta} \tilde P_0(D_x) v \|_{E_0^{\geq m}} \les 2^{nd} \| \chi_I e^{-it\Delta} \tilde P_0(D_x) v \|_{E_0^{\geq m-d}}  \les 2^{nd} \|v\|_{E_0^{\geq m}}. 
\]
We let $v=\chi_I P_{< m} (D_t) e^{-it\Delta}  \chi_{I_0} u$, invoke \eqref{aux8} and \eqref{aux5} to obtain \eqref{aux10b}; note that here we need to take $N$ large in \eqref{aux8} so that we compensate for the loss $2^{nd}$.

\vspace{.1in}

\bf The bound in $E_{0,\e}^{\geq m}$. \rm  To access the local smoothing estimates, we rewrite the modulation information in the lateral coordinates and proceed as above. We fix $\e \in \S^{n-1}$ and work with $P_{0,\e} \chi_{[0,1]} u$ below. In the support of $P_{0,\e}$ we have that $ 2^{-1} \leq |\xi \cdot \e| \leq 2$, and the exact sign of $\xi \cdot \e$ plays a role in the argument below. We can further assume that we localize in the region  $ 2^{-1} \leq \xi \cdot \e \leq 2$; the argument in the region $2^{-1} \leq - \xi \cdot \e \leq 2$ is similar. 

We have already disposed of the very high modulations, that is when $m \geq -100$ in \eqref{dhm}.  We now consider the case $m < -100$; moreover since the structure in $E_{0,\e}^{\geq -100}$ is stronger than the one in $E_{0,\e}^{\geq m}$ and we have \eqref{dhm}, it suffices to establish the estimate for $Q_{\geq m} \chi_{[0,1]} u - Q_{\geq -100} \chi_{[0,1]} u$.

In the Fourier variables we let $\xi_\e=\xi \cdot \e \in \R$ and $\xi \cdot \e^\perp=\bar \xi_e \in \R^{n-1}$. We write the symbol
\[
\tau-\xi^2= \tau - \xi^2_e -\bar \xi_\e^2.
\]
Since $2^m \les |\tau-\xi^2| \leq 2^{-10}$, this takes the form $2^m \les |\xi_\e^2 - (\tau - \bar \xi_e^2)| \leq 2^{-10}$ and it has two implications. First since $ 2^{-2} \leq \xi_\e^2 \leq 2^2$, we obtain that $ 2^{-3} \leq \tau - \bar \xi_e^2 \leq 2^3$ (in particular it is positive). Second,  $2^{-2} \leq  \xi_\e + \sqrt{\tau - \bar \xi_\e^2} \leq 2^4$ and we obtain that  $|\xi_\e - \sqrt{\tau - \bar \xi_\e^2}| \geq 2^{m-6}$; this corresponds to the fact that the modulation in the "$\e$ frame" is bounded from below.  To make this precise we define the following modulation operator
\[
Q_j^{\e,-} = \rho_j(\xi_\e - \sqrt{\tau - \bar \xi_\e^2}). 
\]
which is adapted to the flow 
\begin{equation} \label{SEL}
(i \partial_{x_\e} + a(D_t, D_{\bar x_e})) u =0, \quad a(\tau,\bar \xi)= \sqrt{\tau - \bar \xi_\e^2},
\end{equation}
where the above is to be read as follows:
\[
(a(D_t, D_{\bar x_\e}) f)(t,x_\e) = \frac1{(2\pi)^n} \int e^{i(t,x_\e) \cdot (\tau,\bar \xi_\e)}  a(\tau,\bar \xi) \hat f(\tau,\bar \xi_\e) d \tau d \bar \xi_\e.
\]
Next we write
\[
Q_{\geq m} - Q_{\geq -100}  = Q^{\e,-}_{\geq  m}  - Q^{\e,-}_{\geq -100} + Q_{\geq m} -  Q^{\e,-}_{\geq m} + Q^{\e,-}_{\geq -100} - Q_{\geq -100}.
\]
We can add the operator $Q_{\geq m-10} - Q_{\geq -90}$, so below we implicitly assume that we work  in the range $2^{m-12} \leq |\tau-\xi^2| \leq 2^{-80}$. 

The term $Q_{\geq m} -  Q^{\e,-}_{\geq m}$ localizes in the region $|\tau-\xi^2| \approx 2^m$ and thus we have
\[
\| (Q^{\Delta}_{\geq m} -  Q^{\e,-}_{\geq m}) \chi_{[0,1]} u \|_{U^2_\Delta} \les \| (Q^{\Delta}_{\geq m} -  Q^{\e,-}_{\geq m}) \chi_{[0,1]} u \|_{V^2_\Delta} 
\]
on behalf of the inclusion $V^2_\Delta \subset \dot X^{0,\frac12,\infty}$, the fact that we deal with a bounded range of modulations and the boundedness of the operator $Q^{\Delta}_{\geq m} -  Q^{\e,-}_{\geq m}$ on $L^2$. Using \eqref{U2Xcest} provides the desired estimate for this term. 

The term $Q^{\e,-}_{\geq -100} - Q_{\geq -100}$ is the easiest. We dispose of $Q^{\e,-}_{\geq -100} P_{0,\e} \chi_{[0,1]} u$ in a similar way to how we did for $Q_{\geq -100} \chi_{[0,1]} u$ in \eqref{dhm}.

The term $Q^{\e,-}_{\geq  m}  - Q^{\e,-}_{\geq -100}$ is the nontrivial one. Here we emulate the same argument as before, just that we work in the new frame. First we adapt the definitions of $\mathcal{I}_{k}$ and $C_{k}$ to the new coordinates: $\mathcal{I}_{k}^{\e}$ is the set of intervals of length $2^k$ in the direction of $\e$, while $C_{k}^{\e}$ is the set of cubes of size $2^k$ in the hyperplane spanned by $\e^\perp$ and the time direction.

We also define the norm
\[
\| v \|_{E_{0,\e}^{\geq m}} =\left( \sum_{I \in \mathcal{I}_{-m}^{\e}, q \in C_{-m}^{\e}} \| v  \|^2_{L^\infty_{x_\e} L^2_{t, \bar x_\e}(I \times q)} \right)^\frac12. 
\]
With this notation, our control is of the form
\[
\left( \sum_{I \in \mathcal{I}_{0}^{\e}, q \in C_{0}^{\e}} \| \chi_{[0,1]}(t) P_{0,\e} u \|^2_{L^\infty_{x_\e} L^2_{t, \bar x_\e}(I \times q)} \right)^\frac12  \les \|f\|_{F_0},
\]
which is a consequence of \eqref{leqconc}.   The above estimate provides control at scale $2^{-m}$, that is  
\begin{equation} \label{aux65}
\| \chi_{[0,1]}(t) P_{0,\e} u  \|_{E_{0,\e}^{\geq m}}   \les \|f\|_{F_0}.
\end{equation}
Our task becomes then to establish the equivalent of \eqref{QbEm}:
\begin{equation} \label{QbEml}
\| Q^{\e,-}_{<  m} u \|_{E_{0,\e}^{\geq m}} \les \| u \|_{E_{0,\e}^{\geq m}}, 
\end{equation}
for any $u \in E_{0,\e}^{\geq m}$ which is localized at frequency $\approx 1$, in the sense that $|(\tau,\xi'_\e)| \approx 1$ in the support of its Fourier transform; it is clear that this provides the desired estimate in the $\e$ frame. 

To prove \eqref{QbEml} we start with the equivalent of \eqref{aux5}
\[
Q_{\geq m}^{\e,-} \chi_{I} u = e^{i (x_\e-c(I)) a(D_t,D_{\bar x_{\e}})} P_{\geq m} (D_{x_\e}) e^{-i (x_\e-c(I)) a(D_t,D_{\bar x_{\e}})}   \chi_I u,
\]
where $I \in \mathcal{I}_{-m}^\e$ and $c(I)$ is any point in $I$. Associated to the flows $ e^{\pm i (x_\e-c(I)) a(D_t,D_{\bar x_{\e}})} $, we also have the corresponding commutators:
\[
[(i \partial_{x_\e} + a(D_t, D_{\bar x_\e})), (t,\bar x_{\e}) + x_\e (\nabla_{\tau,\bar \xi_{\e}} a) (D_t, D_{\bar x_\e}))]=0. 
\]
Using the multiplier $(t, \bar x_{\e}) \pm x_\e (\nabla_{\tau, \bar \xi_{\e}} a) (D_t, D_{\bar x_\e}))$ instead of $x \pm 2it \nabla$, a similar argument to how we obtained \eqref{aux10b} gives the following estimate
\[
\|  \chi_{I}  Q_{\geq m}^{\e,-}  \chi_{I'} u \|_{E_{0,\e}^{\geq m}}  \les \la 2^m d(I,I') \ra^{-N}  \| \chi_{I'} u \|_{E_{0,\e}^{\geq m}},
\]
for any $I,I' \in \mathcal{I}_{-m}^{\e}$. From this one obtains \eqref{QbEml} and this concludes the bound in  
$E_{0,\e}^{\geq m}$.

\subsection{Forcing in $F_k^{\geq l}$} 

After rescaling, we can assume that $k=0$.  Recall that we seek to recover the $X_{k,m}$ structure for all $m \leq 2k$. We first set $m=l$, recover this structure and then we explain how the other structures (for $m \ne l$) are recovered. 

We set $m=l$. Our forcing $f \in F_0^{\geq l}$ has the additional property that $Q_{\geq m-10} f =f$. We also note that $Q_{\geq m-10} u$ solves the same equation
\[
(i \partial_t + \Delta) Q_{\geq m-10} u =f,
\]
but with, potentially, a different initial data. This means that $u - Q_{\geq m-10} u$ solves the free Schr\"odinger equation, thus $u - Q_{\geq m-10} u= e^{it\Delta} f, f(x) = (u - Q_{\geq m-10} u)(0,x)$.  Our theory below provides the estimate
\begin{equation} \label{aux72}
\| (Q_{\geq m-10} u) (0,x) \|_{L^2_x} \les \| f \|_{F_0^{\geq m}},
\end{equation}
thus it suffices to establish the desired estimates for $Q_{\geq m-10} u$. 

We  continue with the analysis of the inhomogeneous equation by decomposing $f$ as follows:
\[
f = \sum_{I \in \mathcal{I}_{-m}} \chi_I f, \quad \sum_{I \in \mathcal{I}_{-m}} \sum_{q \in C_{-m}}  \|\chi_I \chi_q f \|^2_{L^1_t L^2_x} \les \|f\|_{F_0^{\geq m}}^2. 
\]
For fixed $I \in \mathcal{I}_{-m}$ we let $f_I=  \chi_I f$ and $u_I$ be the forward in time solution to the equation
\begin{equation} \label{SEFhm}
(i \partial_t + \Delta) u_I = f_I, \quad  u(t^{l}_{I},x)= 0.
\end{equation} 
For fixed $I \in  \mathcal{I}_{-m}$, we record the following estimate
\begin{equation} \label{aux71}
 \la 2^m d(I,I') \ra^{N} \| \chi_{I'} Q_{\geq m-10} (\chi_I u_I)  \|_{E_0^{\geq m}} \les \left( \sum_{q \in C_{-m}}  \|\chi_I \chi_q f \|^2_{L^1_t L^2_x} \right)^\frac12,
\end{equation}
which, in the absence of the modulation projector $Q_{\geq m-10}$,  is a consequence of the speed of propagation being $\approx 1$ and the size restriction of the evolution time interval being $2^{-m}$. Technically it is established in a similar manner as \eqref{aux10b} by using commutators. 

As a consequence of \eqref{aux71} we have the following estimate
\begin{equation} \label{aux70}
\| u_I (t_I^r)\|_{L^2_x} \les \left( \sum_{q \in C_{-m}}  \|\chi_I \chi_q f \|^2_{L^1_t L^2_x} \right)^\frac12,
\end{equation}
Then we produce a particular solution to the original inhomogeneous equation:
\[
u = \sum_{I \in \mathcal{I}_{-m}} \chi_I u_I +  \chi_{(t^r_I,+\infty)}  e^{i(t-t_I^r)\Delta} u_I(t_I^r). 
\]
This is a formal object which cannot be meaningfully defined unless the sum $\sum_{I \in \mathcal{I}_{-m}}$ has finitely many terms. What we can make sense of is instead
\[
Q_{\geq m-10} u = \sum_{I \in \mathcal{I}_{-m}} Q_{\geq m-10}( \chi_I u_I) +  Q_{\geq m-10}(\chi_{(t^r_I,+\infty)}  e^{i(t-t_I^r)\Delta} u_I(t_I^r)).
\]
From \eqref{U2fd} it follows that for each $I \in \mathcal{I}_{-m}$ we have
\[
\sum_{I' \in \mathcal{I}_{-m}}  \la 2^{m} d(I',I)\ra^N \| \chi_{I'}   Q_{\geq m-10}(\chi_{(t^r_I,+\infty)}  e^{i(t-t_I^r)\Delta} u_I(t_I^r)) \|_{U^2_\Delta} \les_N \| u_I(t_I^r)) \|_{L^2} \les \| f_I \|_{F_0^{\geq m}}. 
\]
Based on \eqref{U2int} and \eqref{aux70} we conclude with
\[
\begin{split}
& \| Q_{\geq m-10} \sum_{I \in \mathcal{I}_{-m}} \chi_{(t^r_I,+\infty)}  e^{i(t-t_I^r)\Delta} u_I(t_I^r) \|_{U^2_\Delta}^2 \les 
\sum_{I \in \mathcal{I}_{-m}} \| u_I(t_I^r) \|^2_{L^2_x} \\
\les & \sum_{I \in \mathcal{I}_{-m}} \sum_{q \in C_{-m}}  \|\chi_I \chi_q f \|^2_{L^1_t L^2_x} \les \|f\|_{F_0^{\geq m}}^2.
\end{split}
\]
Then we invoke \eqref{U2Xcest} to conclude with 
\[
\| Q_{\geq m-10} \sum_{I \in \mathcal{I}_0} \chi_{(t^r_I,+\infty)}  e^{i(t-t_I^r)\Delta} u_I(t_I^r) \|_{X_{0,m}} \les \|f\|_{F_0^{\geq m}}. 
\]
Based on \eqref{aux71} we obtain
\[
\left(  \sum_{I' \in \mathcal{I}_{-m}} \sum_{q \in C_{-m}} \| \sum_{I \in \mathcal{I}_{-m}} Q_{\geq m-10}( \chi_I u_I) \|^2_{L^\infty_t L^2_x(q \times I')} \right)^\frac12 \les \| f \|_{F_0^{\geq m}}. 
\]
Therefore so far we have provided all the necessary estimates for $Q_{\geq m-10} u$ in $E_0^{\geq m}$ in the case $m=l$; in particular we also obtained \eqref{aux72}. 

Below we consider the cases $m \ne l$.

If $m < l-12$, then $Q_{\geq m} u = Q_{\geq l-10} u$ and since the norm in $E_0^{\geq l}$ controls the one in $E_0^{\geq m}$, we obtain the desired control on $\| Q_{\geq m} Q_{\geq l-10} u \|_{E_0^{\geq m}}$. 

Finally, if $m \geq l-12$, then we notice that $\| Q_{\geq m} f \|_{F_0^{\geq m}} \les \|f \|_{F_0^{\geq l}}$; this is based on two things: the structure $F_0^{\geq l}$ is stronger than $F_0^{\geq m}$ and the operator $Q_{\geq m}$ is bounded on either structures. The boundedness of the operator $Q_{\geq m}$ is done as follows. We write $Q_{\geq m} = I -  Q_{< m}$ and conjugate $Q_{< m}$ as usual $Q_{< m}  = e^{it\Delta} P_{< m} (D_t) e^{-it\Delta}$ to obtain
\[
Q_{< m}  f = \sum_{I \in \mathcal{I}_{-m}} e^{it\Delta} P_{< m} (D_t) e^{-it\Delta} \chi_I f. 
\]
Just as we did before, using commutator estimates, and the fact that kernel of $P_{< m}$ is a bump function adapted to scale $2^{-m}$, provides the boundedness $Q_{\geq m}$ on $F_0^{\geq m}$ which suffices. Once we have control on $\| Q_{\geq m} f\|_{F_0^{\geq m}}$, 
we repeat the previous argument (for $m=l$) and obtain the bound for $Q_{\geq m}$ in $E_0^{\geq m}$. 

The analysis in lateral coordinates, that is $(x_\e,t, \bar x_\e)$ is entirely similar, just that we rewrite the equation in these coordinates, see \eqref{SEL}.

\subsection{Forcing in $F_{k,\e}$.} By rescaling we can assume $k=0$.  We start from the equation
\begin{equation} \label{SEFLF}
(i \partial_t + \Delta) u =  f, \quad  u(0,x)= 0,
\end{equation}
where we recall that $f \in F_{0,\e}$ satisfies $\tilde P_{0,\e} f =f$, which in turn implies that the solution $u$ satisfies $\tilde P_{0,\e} u =u$. 

As mentioned earlier, it is convenient to dispose of high modulations when passing between the standard frame and the $\e$ frames. This can be easily done by noticing that, given the localization of $f$ in the direction of $\e$ we obtain $\| f \|_{L^2_{t,x}} \les \| f \|_{F_{0,\e}}$, from which we obtain
\[
\| Q_{\geq -100} u  \|_{X^{0,\frac12,1}} \les \| Q_{\geq -100} u  \|_{L^2_{t,x}} \les  \| f \|_{L^2_{t,x}} \les \| f \|_{F_{0,\e}}. 
\]
Thus in what follows we can assume that both $u$ and $f$ are localized at modulation $\leq -100$; note that the kernel of the operator $Q_{\leq -100}$ belongs to $L^1_{t,x}$, hence it is bounded on $F_{0,\e}$. 

Next, establishing the estimates for the solution $u$ of \eqref{SEFLF} in $X_{0,m}$ is similar to the analysis carried for the case $f \in F_0$, just that we work in the coordinates adapted to the $\e$ direction (versus the standard $t,x$), while having the correct localization and having disposed of the high modulation.

\subsection{Forcing in $F_{k,\e}^{\geq l}$.} By rescaling we can assume $k=0$; just as explained above, the analysis is similar to the one when $f \in F_{0}^{\geq l}$, just that we work in the adapted frame.

\section{The bilinear estimate: Proof of Theorem \ref{MTbil}}

The main goal in this section is to prove  Theorem \ref{MTbil}. 
Using rescaling arguments it suffices to consider the case $k_1 \leq 0$ and $k_2=0$, in which case we need to prove the following
\begin{equation} \label{pbil}
\|u \cdot v\|_{L^2_{t,x}} \les 2^{\frac{(n-1)k}2} \|u\|_{X_k} \| v \|_{X_0},
\end{equation}
where $u$ is localized at frequency $2^k \les 1$ and $v$ is localized at frequency $1$. 

We split the analysis based on a specific modulation threshold of the inputs, that is
\[
u \cdot v = Q_{\geq 2k} u \cdot v + Q_{< 2k} u \cdot Q_{\geq 2k} v + Q_{< 2k} u \cdot Q_{< 2k} v. 
\] 
Then \eqref{pbil} is a consequence of the following two results:
\begin{proposition} \label{Hprop}
Assume that $u \in X_{k,2k}$ is localized at frequency $\approx 2^k, k \leq 0$ and $v \in X_{0,2k}$ is localized at frequency $\approx 1$. Then the following holds true
\begin{equation} \label{L2hm}
 \| Q_{\geq 2k} u \cdot v \|_{L^2_{t,x}}  +  \|  u \cdot Q_{\geq 2k} v \|_{L^2_{t,x}}  \les 2^{\frac{(n-1)k}2} \| u \|_{X_{k,2k}} \|v \|_{X_{0,2k}}.
\end{equation}
\end{proposition}

\begin{proposition} \label{Mprop}
Assume that $u \in V^2_\Delta$ is localized at frequency $\approx 2^k, k \leq 0$ and $v \in V^2_\Delta$ is localized at frequency $\approx 1$. Then following holds true
\begin{equation} \label{L2V2}
 \| Q_{\leq 2k} u \cdot  Q_{\leq 2k} v \|_{L^2_{t,x}}  \les 2^{\frac{(n-1)k}2} \| u \|_{V^2_\Delta} \|v \|_{V^2_\Delta}.
\end{equation}
\end{proposition}

It is instructive to keep in mind that all the structures that populated the space $X_{k,m}$, in addition to the  $V^2_\Delta$, were developed precisely in order to accommodate the result in Proposition \ref{Hprop}. The proof of Proposition \ref{Hprop} is fairly straightforward, while the proof of Proposition \ref{Mprop} is far more involving.

\begin{proof}[Proof of Proposition \ref{Hprop}]

The term $Q_{\geq 2k} u \cdot v$ is estimated as follows: 
\[
\begin{split}
\| Q_{\geq 2k} u \cdot v \|_{L^2_{t,x}} & \les \left( \sum_{I \in \mathcal{I}_{-2k}} \sum_{q \in C_{-k}} \| Q_{\geq 2k} u \|^2_{L^\infty_{t,x}(q \cap I)} \right)^\frac12 \cdot  \sup_{I \in \mathcal{I}_{-2k}} \sup_{q \in C_{-k}} \|  v\|_{L^2_{t,x}(q \cap I)} \\
 & \les 2^{\frac{nk}2} \left( \sum_{I \in \mathcal{I}_{-2k}} \sum_{q \in C_{-k}} \| Q_{\geq 2k} u \|^2_{L^\infty_{t}L^2_x(q \cap I)} \right)^\frac12 \cdot  2^{-\frac{k}2} \|  v\|_{X_0} \\
&\les 2^{\frac{(n-1)k}2} \| Q_{\geq 2k} u \|_{E_{k}} \cdot \|  v\|_{X_0}  =
 2^{\frac{(n-1)k}2}  \| u\|_{X_{k,2k}}\| v \|_{X_0}. 
\end{split}
\]
In passing to the second line above we used \eqref{LEX} for $v$ and the following refined version of the Bernstein inequality
\[
\sum_{q \in C_{-k}}   \| Q_{\geq 2k} u \|^2_{L^\infty_{t,x}(q \cap I)}   
\les 2^{nk} \sum_{q \in C_{-k}} \| Q_{\geq 2k} u \|^2_{L^\infty_{t}L^2_x(q \cap I)};
\]
this is justified using the fast decay of the kernel of the projector $P_k$ away from the region $|x| \les 2^{-k}$. 
 
Next we estimate the term $Q_{< 2k} u \cdot Q_{\geq 2k} v$ as follows:
\[
\begin{split}
\| Q_{< 2k} u \cdot Q_{\geq 2k} v \|_{L^2_{t,x}} & \les \sup_{I \in \mathcal{I}_{-2k}} \left( \sum_{q \in C_{-k}} \| Q_{< 2k} u \|^2_{L^\infty_{t,x}(q \cap I)} \right)^\frac12 \cdot \left( \sum_{I \in \mathcal{I}_{-2k}} \sup_{q \in C_{-k}} \| Q_{\geq 2k}  v\|_{L^2_{t,x}(q \cap I)}^2  \right)^\frac12 \\
& \les 2^{nk} \sup_{I \in \mathcal{I}_{-2k}} \left( \sum_{q \in C_{-k}} \| Q_{< 2k} u \|^2_{L^\infty_{t}L^2_x(q \cap I)} \right)^\frac12 \cdot  2^{-\frac{k}2}  \| Q_{\geq 2k}  v\|_{X_{0}^{\geq 2k}} \\
& \les  2^{\frac{(n-1)k}2}  \| u\|_{X_k}\| v \|_{X_{0,2k}}. 
\end{split}
\]
This concludes the proof of the Proposition.
\end{proof}

\subsection{Proof of Proposition \ref{Mprop}} We let $\mathcal{C}_{j}$ be the usual $C_{j}(\R^{n+1})$, the set of (space-time) cubes of size $2^j$ covering $\R^{n+1}$. We begin by restating the classical result for free waves. 

\begin{lemma} \label{bilR}
Let  $p>\frac{n+3}{n+1}$. There exists $C(p)$ such that the following holds true. Assume that $f, g \in L^2_x$, with $f$ being localized at frequency $\approx 2^k, k \leq 0$ and $g$ localized at frequency $\approx 1$. The following holds true
\begin{equation} \label{Lbilfw}
\left( \sum_{q \in \mathcal{C}_{-2k}} \| e^{it\Delta} f \cdot  e^{it\Delta} g \|^p_{L^2(q)} \right)^\frac1p \leq C(p)  2^{\frac{(n-1)k}2} \|  f\|_{L^2} \|  g \|_{L^2}.
\end{equation}
A similar result holds true if we replace the free waves with elements $u,v \in U^2_\Delta$, with with $u$ localized at frequency $\approx 2^k \ll 1$ and $g$ localized at frequency $\approx 1$:
\begin{equation} \label{LbilfwU2}
\left( \sum_{q \in \mathcal{C}_{-2k}} \| u \cdot  v \|^p_{L^2(q)} \right)^\frac1p \leq C(p)  2^{\frac{(n-1)k}2} \|  u\|_{U^2_\Delta} \|  g \|_{U^2_\Delta}.
\end{equation}

\end{lemma}

The proof of this Lemma is provided in the Appendix.  We remark that the standard theory provides this estimate with $L^2(q)$ replaced by $L^p(q)$, in which case the estimate reads an estimate for $ \| e^{it\Delta} f \cdot  e^{it\Delta} g \|_{L^p}$ and with modified powers. 

Our second result is a soft estimate which has the advantage of requiring minimal information on the inputs. 
\begin{lemma} \label{Lbilv2}
Assume that $u, v \in L^\infty_t L^2_x$, with $u$ being localized at frequency $\approx 2^k \ll 1$ and $v$ localized at frequency $\approx 1$.  Then for any cube $q  \in \mathcal{C}_{-2k}$ the following holds true
\begin{equation} \label{linfq}
 \| Q_{\leq 2k} u \cdot   Q_{\leq 2k} v \|_{L^2(q)}  \les 2^{\frac{(n-1)k}2} \| u  \|_{L^\infty_t L^2_x} \| v \|_{L^\infty_t L^2_x}.
\end{equation}
\end{lemma}

\begin{proof} Below we skip the use $Q_{\leq 2k}$ and simply work with $u$ with the property that $Q_{\leq 2k+10} u =u$, and similarly for $v$.  Let $I$ be a time interval of length $2^{-2k} $ with the property that $q \subset I \times \R^n$. Since $\tilde \chi_I \equiv 1$ in $ q$ we can replace $u$ with $\tilde \chi_{I} u$ and $v$ with $\tilde \chi_{I} v$ in the above estimate. Since $u \in L^\infty_t L^2_x$, it follows that $\tilde \chi_I u \in L^2_{t,x}$, hence
\[
\tilde \chi_I u =  \sum_{j \in \Z} Q_j ( \tilde \chi_I u). 
\]
In Lemma \ref{lmX} we establish the following
\begin{equation} \label{aux21a}
\| \tilde \chi_I  u \|_{X^{0,\frac12,1}} \les  \| u \|_{L^\infty_t L^2_x}. 
\end{equation}
Then our result in \eqref{linfq} follows from \eqref{aux21a} and the following claim
\begin{equation} \label{aux21}
 \| \tilde \chi_I u \cdot  \tilde \chi_I v \|_{L^2(q)}  \les  2^{\frac{(n-1)k}2} \| \tilde \chi_I u  \|_{X^{0,\frac12,1}} \| \tilde \chi_I v \|_{X^{0,\frac12,1}}.
\end{equation}

This claim is standard in the literature: it says that $X^{0,\frac12,1}$ transfers known estimates for free solutions. 
\eqref{aux21} follows from its dyadic counterpart
\[
 \| Q_{j_1} u \cdot   Q_{j_2} v \|_{L^2(q)}  \les  2^{\frac{(n-1)k}2} 2^{\frac{j_1}2} \| Q_{j_1} u  \|_{L^2_{t,x}} 2^{\frac{j_2}2} \| Q_{j_2} v \|_{L^2_{t,x}},
\]
which we claim to be true for any $j_1,j_2 \in \Z$. But this is a consequence of the fact that $Q_{j_1} u, Q_{j_2} v$ are continuous superpositions of free waves and this allows us to invoke \eqref{Lbilfw}. Indeed letting $\tilde u = e^{-it\Delta} u$, we have $Q_{j_1} u = e^{it\Delta} P_{j_1}(D_t) \tilde u$ and 
\[
P_{j_1} \tilde u(t,x) = \frac1{(2\pi)^{n+1}} \int \rho_{j_1}(\tau) e^{i(t,x) \cdot (\tau,\xi)} \mathcal{F}_{t,x} \tilde u(\tau,\xi) d\tau d\xi = \frac1{2\pi} \int \rho_{j_1}(\tau) e^{it\tau} \mathcal{F}_t \tilde u(\tau,x) d \tau,
\]
so letting $\tilde u_\tau(x) := \frac1{2\pi} e^{it\tau} \mathcal{F}_t \tilde u(\tau,x)$ gives the representation formula
\[
P_{j_1}(D_t) \tilde u  = \int  \rho_{j_1}(\tau)  \tilde u_\tau d \tau, \quad  \int \rho_{j_1}(\tau)  \|\tilde u_\tau\|_{L^2_x} \les 2^{\frac{j_1}2}\| \tilde u \|_{L^2_{t,x}} \approx 2^{\frac{j_1}2}\| Q_j u \|_{L^2_{\tau,\xi}}.
\]
 note that $\tilde u_\tau$ does not depend on $t$. From this we obtain 
\[
Q_{j_1} u = \int \rho_{j_1}(\tau)  e^{it\Delta} \tilde u_\tau d \tau,
\]
which is a representation of $Q_{j_1} u$ as a continuous superposition of free waves. A similar representation is obtained for $Q_{j_2} v$. 

\end{proof}

The result in Lemma \ref{Lbilv2} can be restated as follows
\[
\sup_{ q \in \mathcal{C}_{-2k}} \| Q_{\leq 2k}  u \cdot Q_{\leq 2k}  v \|_{L^2_{t,x}(q)} \les  2^{\frac{(n-1)k}2} \| u  \|_{L^\infty_t L^2_x} \| v \|_{L^\infty_t L^2_x}.
\]
By interpolating between \eqref{LbilfwU2} and this statement we obtain
\begin{equation} \label{bilUa}
\left( \sum_{q \in \mathcal{C}_{-2k}} \| Q_{\leq 2k} u \cdot  Q_{\leq 2k} v \|^2_{L^2(q)} \right)^\frac12 \les 
2^{\frac{(n-1)k}2} \| u  \|_{U^a_\Delta} \| v \|_{U^a_\Delta},
\end{equation}
where $a=\frac4{p}$. Any viable choice of $\frac{n+3}{n+1} <p<2$ works since $V^2_\Delta \subset U^a_\Delta$ given that $a=\frac4{p} > 2$. This completes the proof of Proposition \ref{Mprop}, and in turn the proof of Theorem \ref{MTbil}.

The formalization of the above interpolation can be found in \cite{Ca}, see Proof of Theorem 1.7 there. In the language used there, for an atom
$ a(t)=\sum_{I} \mathds{1}_{I}(t) e^{it\Delta}  \phi_I$ one defines
\[
\|a\|_{l^a L^2}= \left( \sum_I \| \phi_I \|^a_{L^2_x} \right)^\frac1a. 
\]
With this notation, the two bounds between which interpolate read
\[
\left( \sum_{q \in \mathcal{C}_{-2k}} \| Q_{\leq 2k}  u \cdot  Q_{\leq 2k}  v \|^p_{L^2(q)} \right)^\frac1p \leq C(p)  2^{\frac{(n-1)k}2} \|  u\|_{l^2 L^2} \|  g \|_{l^2 L^2}.
\]
and
\[
 \sup_{q \in \mathcal{C}_{-2k}} \| Q_{\leq 2k}  u \cdot  Q_{\leq 2k}  v \|_{L^2(q)}  \les   2^{\frac{(n-1)k}2} \|  u\|_{l^\infty L^2} \|  g \|_{l^\infty L^2}.
\]
This makes the interpolation more visual, and again, more details can be found in \cite{Ca}. 

We conclude this section with a proof of \eqref{aux21a}; in fact we prove a slight generalization of it. 

\begin{lemma} \label{lmX}

Assume that $u\in L^\infty_t L^2_x$, is localized at modulation $\les 2^m$, that is $u=Q_{\leq m+10} u$ and $I$ is an interval of size $\les 2^{-m}$. Then the following holds true
\begin{equation} \label{V2X121}
\| \tilde \chi_I  u \|_{X^{0,\frac12,1}} \les  \| u \|_{L^\infty_t L^2_x}. 
\end{equation}

\end{lemma}

\begin{proof} The starting point is the control the low modulation part:
\[
\|Q_{\leq m} \tilde \chi_I  u \|_{X^{0,\frac12,1}}  \les (2^{m})^\frac12 \|Q_{\leq m} \tilde \chi_I  u \|_{L^2_{t,x}}  
 \les 2^{\frac{m}2} |I|^\frac12 \| u \|_{L^\infty_t L^2_{x}} \les \| u \|_{L^\infty_t L^2_{x}}.
\]
The next observation is that, while  $\tilde \chi_I  u$ loses the exact modulation localization that $u$ has, the contributions at modulations $\geq 2^m$ decay fast.  Fix $j \geq m$. We start with
\[
\begin{split}
\|Q_{j} \tilde \chi_I  u \|_{L^2_{t,x}} &= \| e^{-it\Delta} Q_{j} \tilde \chi_I  u \|_{L^2_{t,x}} =
\|   P_j(D_t) \tilde \chi_I e^{-it\Delta}  u \|_{L^2_{t,x}} \\
& = \|   P_j(D_t) \tilde \chi_I  P_{\leq 2k+10}(D_t) e^{-it\Delta} u \|_{L^2_{t,x}}. 
\end{split}
\]
We let $v = e^{-it\Delta} u$ and record the simple inequality $\|v\|_{L^\infty_t L^2_x} \approx \|u\|_{L^\infty_t L^2_x} $.
Then we compute
\[
\begin{split}
\|  \partial_t  P_j(D_t) \tilde \chi_I  P_{\leq 2k+10}(D_t)  v \|_{L^2_{t,x}} & = \|   P_j(D_t) \partial_t  (\tilde \chi_I  P_{\leq 2k+10}(D_t)  v) \|_{L^2_{t,x}} \\
& \les  \|   \partial_t  (\tilde \chi_I  P_{\leq 2k+10}(D_t)  v) \|_{L^2_{t,x}} \\
& \les  \|   \partial_t  \tilde \chi_I  \cdot P_{\leq 2k+10}(D_t)  v \|_{L^2_{t,x}} +  \|    \tilde \chi_I  \cdot \partial_t  P_{\leq 2k+10}(D_t)  v \|_{L^2_{t,x}} \\
& \les  2^{m} |I|^\frac12 \|  v \|_{L^\infty_t L^2_{x}} \approx 2^{\frac{m}2} \|u\|_{L^\infty_t L^2_x}. 
\end{split}
\]
On the other hand $\|  \partial_t  P_j(D_t) \tilde \chi_I  P_{\leq 2k+10}(D_t)  v \|_{L^2_{t,x}} \approx 2^j \| P_j(D_t) \tilde \chi_I  P_{\leq 2k+10}(D_t)  v \|_{L^2_{t,x}} $, thus from the above estimate we obtain
\[
\| P_j(D_t) \tilde \chi_I  P_{\leq 2k+10}(D_t)  v \|_{L^2_{t,x}} \les 2^{-j }  2^{\frac{m}2} \|u\|_{L^\infty_t L^2_x}. 
\]
From this it follows that
\[
\| Q_{\geq m} \tilde \chi_I  u \|_{X^{0,\frac12,1}}  \les \sum_{j \geq m} 2^{\frac{j}2} \| Q_{j} \tilde \chi_I  u \|_{L^2_{t,x}}  \les \sum_{j \geq m} 2^{\frac{m-j}2}   \|u\|_{L^\infty_t L^2_x} \les \|u\|_{L^\infty_t L^2_x}. 
\]

\end{proof}

\section{Applications: Derivative NLS and the Schr\"odinger Maps}

As an application of the theory developed in this paper we establish the global well-posedness (GWP) for small data in a critical space for a derivative NLS which exhibits null structure. This serves as a simplified model for the Schr\"odinger map equation, for which we provide an alternative proof of the GWP for the problem with small and localized data. In both cases, the standard approach for the GWP theory involves the use of the maximal function estimate $L^2_{x_\e} L^\infty_{t, \bar x_\e}$ combined with local smoothing type norms $L^\infty_{x_\e} L^2_{t, \bar x_\e}$. This works well in dimensions $n \geq 3$, see \cite{Be2, IoKe-2}, but it becomes more complicated in dimension $n=2$ because the maximal type estimate is known to fail "logarithmically". A rather subtle replacement of the maximal function estimate $L^2_{x_\e} L^\infty_{t, \bar x_\e}$ was introduced in \cite{BIKT-1} to effectively deal with this problem. We highlight that this replacement is time dependent, that is the construction depends on the time interval $[0,T]$ where it is made, and then it is shown that the estimates are uniform time so that one can derive the global theory. 

In this section we establish the GWP result for the problems mentioned above in an iteration space that is based on the spaces introduced in this paper. Moreover, when $n=2$, our iteration space has the added bonus that it is global in time, unlike its counterpart mentioned above. 

Our first equation is the following derivative NLS equation:
\begin{equation} \label{Mz}
iz_t + \Delta z= (\nabla z)^2, \quad z(0,x)=z_0(x);
\end{equation}
here $z: \R \times \R^n \rightarrow \C$. The nonlinearity is special in that it exhibits a null structure, which can be algebraically formalized as follows:
\begin{equation} \label{null}
2 \nabla z_1 \cdot \nabla z_2= -(i \partial_t + \Delta) z_1 \cdot z_2 -  z_1 \cdot (i \partial_t + \Delta) z_2 + (i \partial_t + \Delta)(z_1 \cdot z_2).
\end{equation}
The wording "null condition" is related to the fact that this nonlinearity has the property that it vanishes a specific type of interaction: two free waves whose product is a free wave; these types or interactions are known in the literature as resonances. 
The scaling of \eqref{Mz} is $s_c=\frac{n}2$. We develop a theory in the scale invariant $\dot B^{\frac{n}2}_{2,1}$ which is the standard Besov space of functions $f$ obeying the Calderon representation formula $f= \sum_{k \in \Z} P_k f$, and endowed with the norm 
\[
\|f\|_{\dot B^{\frac{n}2}_{2,1}} = \sum_k 2^{\frac{nk}2} \|P_k f\|_{L^2}. 
\]
A more classical way of defining this space is by taking the closure of the Schwartz space with respect to the norm above. This ensures that the functions in $\dot B^{\frac{n}2}_{2,1}$ have zero limit at $\infty$; this is important since functions in $\dot B^{\frac{n}2}_{2,1}$ are defined only up to constants and the zero limit at $\infty$ removes this uncertainty.

Our first result in this section is the following
\begin{theorem} \label{MzT}
i) There exists $\epsilon >0$ such that for any initial data $z_0 \in \dot B^{\frac{n}2}_{2,1}$ with 
$\|z_0\|_{\dot B^{\frac{n}2}_{2,1}} \leq \epsilon$, the equation \eqref{Mz} has a solution $z(t) \in C_t \dot B^{\frac{n}2}_{2,1}$ which is unique in the space $S^\frac{n}2_{str}$; the solution is Lipschitz with respect to the initial data. Moreover $z(t)$ scatters to free waves as $t \rightarrow \pm \infty$. 

ii) If the initial data has additional regularity $z_0 \in H^s$ with $s > \frac{n}2$, then the solution preserves that regularity uniformly in time.

\end{theorem}

The space $S^{\frac{n}2}_{str}$ will be introduced below, right after Theorem \ref{SMT}.

One of the motivation for studying the model \eqref{Mz} comes from the more relevant Schr\"odinger map equation which we describe below in \eqref{SMz}. The two equations share the same scaling and main difficulty - the derivative nonlinearity with a null structure. The reason to begin the analysis with the study of \eqref{Mz}, is to highlight the core and the simplicity of the argument, without the additional technicalities that are involved in the analysis of \eqref{SMz}.

We now introduce the Schr\"odinger map equation. We will be somehow brief in the presentation of the subject of Schr\"odinger Maps; for a more in-depth introduction we refer the reader to \cite{GrSt, IoKe-1, IoKe-2, Be2} (particularly for the stereographic projection model), or to \cite{KTV} (for a more general introduction in the context of geometric equations). 

Recall that the Schr\"odinger map equation $u: \R^n \times \R \rightarrow \S^2$ is given by 
\begin{equation}
  u_t = u \times \Delta u, \qquad u(0) = u_0.
\label{SM}\end{equation}
Assuming that the evolution of the map is localized near a point, one can use the stereographic projection based on the antipodal point to represent the equation in local coordinates. Assuming that the map is localized near the north pole $N=(0,0,1)$, we can use the stereographic projection through the north pole $N=(0,0,-1)$ 
\[
u \in \S^2 \rightarrow z= \frac{u_1+iu_2}{1+u_3}, \quad z \in \C \rightarrow (\frac{2 \Re z}{1+|z|^2}, \frac{2 \Im z}{1+|z|^2}, \frac{1-|z|^2}{1+|z|^2}) \in \S^2, 
\]
to represent the equation in local coordinates as follows
\begin{equation} \label{SMz}
i z_t + \Delta z=\frac{2 \bar z}{1+|z|^2} (\nabla z)^2. 
\end{equation}
To keep such a localization in place it suffices to assume $\|u(0)-N\|_{L^\infty} \ll 1$ which can be controlled by $\|u(0)-N\|_{\dot B^{\frac{n}2}_{2,1}} \ll 1$; the latter has the advantage that it is propagated by the flow. This translates into $\|z(0)\|_{\dot B^{\frac{n}2}_{2,1}} \ll 1$. Thus our focus is on the nonlinear equation \eqref{SMz} with small initial data $z_0 \in \dot B^{\frac{n}2}_{2,1}$. 

\begin{theorem} \label{SMT}
The same result as in Theorem \ref{MzT} holds true for the equation \eqref{SMz}. 
\end{theorem}

We end this section with the definition of the resolution space $S^\frac{n}2$. The resolution space is based on the spaces introduced in the first part of this paper and the multilinear estimates used to establish the necessary contraction type bounds rely heavily on the bilinear estimates in \eqref{MTbil}. 

We begin with defining the structure $S_k$ as follows
\[
\|f \|_{S_k} =  \sup_{m \leq 2k} \| Q_{\geq m} f\|_{X_{k,m}}. 
\]
where we recall the convention $Q_{\leq - \infty} = I$ and $X_{k,-\infty}=V^2_\Delta \cap X_k$. 

The core global space is $S^\frac{n}2= \{ f \in C_t \dot B^{\frac{n}2}_{2,1}: P_k f \in S_k \}$ with the norm
\[
\| f \|_{S^\frac{n}2} =  \sum_{k \in \Z} 2^{\frac{nk}2}  \|P_k  f\|^2_{S_{k}}.
\]
This is a Banach space; we briefly sketch an argument. The starting point is the fact that $V^2_\Delta$ is complete and its subspace $C_t \dot B^{\frac{n}2}_{2,1} \cap V^2_\Delta$ is closed. This suffices to identify the limit $f$ of the Cauchy sequence $(f_n)_{n \in \N}$ in $C_t \dot B^{\frac{n}2}_{2,1}$ with the property that $P_k f_n \rightarrow P_k f$ in $C_t \dot B^{\frac{n}2}_{2,1} \cap V^2_\Delta$. Then we upgrade the convergence $P_k f_n \rightarrow P_k f$  in all structures used in $X_{k,m}$ using the observation that all structures involved are mass type combined with $l^2$ (across cubes), hence complete. 

The space $S^\frac{n}2$ contains almost all the information that we need to close the iteration argument for the problem, except for one particular interaction: when a low frequency term contains a very high modulation. To deal with this case we strengthen a bit the resolution space. We refine $S_k$ in the high modulation as follows
\[
\| f \|_{S_{k,str}} = \|f\|_{S_k} + 2^{-k} \| |\tau| Q_{\geq 2k+100} f \|_{L^2_{t,x}}. 
\]
This structure is meant to measure functions localized at frequency $\approx 2^k$, and we note that $S_{k,str}$ augments the structure $S_k$ in the high modulation regime $|\tau-\xi^2| \geq 2^{2k+100}$, where in fact we have $|\tau| \approx |\tau-\xi^2|$. In the original $S_k$ we had $ \sup_{m \geq 2k+100}  \rho_{m}(\tau) |\tau|^\frac12 \hat f \in L^2$ (due to the embedding $V^2_\Delta \subset X^{0,\frac12,\infty}$), while in the new structure we have  $ \rho_{\geq 2k+100}(\tau) 2^{-k} |\tau| \hat f \in L^2$. 
At the boundary threshold $m =2k+100$ the two structures are the same since $|\tau| \approx 2^k |\tau|^\frac12$, but then as $m$ increases, the structure $S_{k,str}$ improves the decay in $\tau$. The reason we use the subscript $str$ is to indicate that $S_{k,str}$ is the stronger version of $S_k$. 

Based on the above, we introduce $S_{str}^\frac{n}2$ as the subspace of function in  $S^\frac{n}2$ with the property
\[
\| f \|_{S^\frac{n}2_{str}} =  \sum_{k \in \Z} 2^{\frac{nk}2}  \|P_k  f\|^2_{S_{k,str}} < \infty. 
\]

\subsection{Derivative NLS} 

In this section we provide the proof of Theorem \ref{MzT}. The main estimates that we seek to establish are the following
\begin{equation} \label{triSk} 
| \int (\nabla z_1 \cdot \nabla z_2) \cdot \bar z_3 dx dt| \les 2^{(k_3-k_{max})} 2^{\frac{n(k_1+k_2-k_3)}2} \| z_1 \|_{S_{k_1,str}} \| z_2 \|_{S_{k_2,str}}  \|z_3\|_{S_{k_3}},
\end{equation}
and
\begin{equation} \label{biSk} 
 \| P_{k_3} Q_{\geq 2k_3 + 100} (\nabla z_1 \cdot \nabla z_2) \|_{L^2_{t,x}} \les 2^{\frac{(k_3-k_{max})}4} 2^{k_3} 2^{\frac{n(k_1+k_2-k_3)}2}\| z_{1} \|_{S_{k_1}} \| z_2 \|_{S_{k_2}}. 
\end{equation}
In the above we assume that $z_i$ are localized at frequency $2^{k_i}, i=1,2,3$ and $k_{max}=\max(k_1,k_2,k_3)$. 

Assuming the two estimates above are true, we establish the Theorem \ref{MzT}. 

We define the operator $T: S_{str}^\frac{n}2 \rightarrow S_{str}^\frac{n}2$ as the solution of the equation
\[
(i \partial_t + \Delta) Tz= (\nabla z)^2, \quad z(0)=z_0.
\]
Below we establish that by invoking  \eqref{triSk}, \eqref{biSk} and Theorem \ref{MTl} it follows that $T: S_{str}^\frac{n}2 \rightarrow S_{str}^\frac{n}2$ is bounded and moreover it is a contraction. 

The first step is a qualitative observation. A direct estimate shows that
\[
\| P_{k_3} (\nabla P_{k_1} z \cdot \nabla P_{k_2} z) \|_{L^\infty_t L^2_x} \les 2^{k_1+k_2} \|P_{k_1} z\|_{L^\infty_t L^\infty_x} 
\| P_{k_2} z\|_{L^\infty_t L^2_x} \les 2^{\frac{(n+2)k_1+2k_2}2} \|P_{k_1}\|_{S_{k_1}} \|P_{k_2}\|_{S_{k_2}}.
\]
This guarantees that $P_{k_3} (\nabla P_{k_1} z \cdot \nabla P_{k_2} z) \in L^1_{t,loc} L^2_x$, hence  $P_{k_3} T ( P_{k_1} z, P_{k_2} z) \in C_t L^2_x$, in particular it is a ruled function; this also provides the qualitative estimate $P_{k_3} T ( P_{k_1} z, P_{k_2} z) \in C_t \dot B^{\frac{n}2}_{2,1}$, which then can be recovered globally using the uniform estimates described below. 

From \eqref{triSk} and Theorem \ref{MTl}, we obtain the following bound
\[
2^{\frac{nk_3}2} \| P_{k_3} T ( P_{k_1} z, P_{k_2} z) \|_{S_{k_3}} \les 2^{(k_3-k_{max})} 2^{\frac{nk_1}2} \|P_{k_1} z \|_{S_{k_1,str}} 2^{\frac{nk_2}2} \|P_{k_2} z \|_{S_{k_2,str}}
\]
A direct summation argument provides the bound $\| Tz \|_{S^\frac{n}2} \les \| z \|^2_{S^\frac{n}2_{str}}$. 
It is easy to see that the same argument gives the contraction type bound
\[
\| Tz_1 - Tz_2 \|_{S^\frac{n}2} \les (\| z_1 \|_{S^\frac{n}2_{str}} + \| z_2 \|_{S^\frac{n}2_{str}}) \| z_1-z_2 \|_{S^\frac{n}2_{str}}. 
\]
Next, using \eqref{biSk} we can improve the two bounds above by placing the left-hand side in $S^\frac{n}2_{str}$. This concludes the proof of the fact that $T: S_{str}^\frac{n}2 \rightarrow S_{str}^\frac{n}2$ is a contraction on $ S_{str}^\frac{n}2$ if we assume that the initial data is small in $\dot B^{\frac{n}2}_{2,1}$ and establishes the claim in Theorem \ref{MzT} part i); the scattering part is standard since $z(t) \in V^2_\Delta$, hence $e^{-it\Delta} z$ has well defined limits as $ t \rightarrow \pm \infty$. 

Part ii) is a standard and follows from the same type of arguments. 

In the remaining of this section we establish the key estimates \eqref{triSk} and \eqref{biSk}. 

\begin{proof}[Proof of \eqref{triSk}] The main goal is the analysis of the trilinear form
\[
I= \int \nabla z_{1} \cdot \nabla z_{2} \cdot \bar z_3 dx dt,
\]
where we recall that $z_i$ is localized at frequency $2^{k_i}, i=1,2,3$. In order to keep the notation at minimum, we abuse it in several places which we now explain. We involve modulation localizations on $z_{i}$, of type $Q_{\geq m} z_i$ and $Q_{< m} z_i$. This should trigger a split of $I$ in additional (but similar) terms to account for that localization; we abuse notation and use the same letter $I$ for the corresponding integrals. In some places we also abuse notation and use $\nabla z$ as if it were a scalar instead of using 
$(\partial_{x_i} z)_{i=1,..,n}$. 

The estimate is symmetric in $k_1$ and $k_2$, hence we can freely assume that $k_1 \leq k_2$.
We split the argument in two cases, based on whether $2^{k_3}$ is a low or high frequency (relatively to the other two).

Case 1. $|k_2-k_3| \leq 10$. Recall also that  $k_1 \leq k_2$. We refer to $z_2,z_3$ as the high frequency factors and $z_1$ as the low frequency factor. The target bound (from \eqref{triSk}) is the following 
\begin{equation} \label{Ib1}
|I| \les 2^{\frac{n k_1}2} \|z_{1}\|_{S_{k_1,str}}   \|z_{2}\|_{S_{k_2}} \|z_{3}\|_{S_{k_3}}.
\end{equation}
Assume that one high frequency factor has modulation $\geq 2^{k_1+k_2}$;
if that is $z_{2}$, we use \eqref{uvL2} and \eqref{UVX} to estimate as follow 
\[
\begin{split}
|I| & \les \| \nabla z_{2} \|_{L^2_{t,x}} \|  \nabla z_{1} \cdot \bar z_{3}\|_{L^2_{t,x}} \les 
2^{-\frac{k_1+k_2}2} 2^{k_2} \|z_{2}\|_{V^2_\Delta} 2^{\frac{(n-1)k_1-k_3}2} 2^{k_1} \|z_{1}\|_{S_{k_1}} 
\|z_{3}\|_{S_{k_3}}  \\
& \les  2^{\frac{n k_1}2} \|z_{1}\|_{S_{k_1}}   \|z_{2}\|_{S_{k_2}} \|z_{3}\|_{S_{k_3}}.
\end{split}
\]
A similar estimate is available if $z_{3}$ has modulation $\geq 2^{k_1+k_2}$. 

Next we consider the case when both high frequencies have modulations $\leq 2^{k_1+k_2}$. In that case, a straighforward computation shows that $z_{1}$ is also localized at modulation $\les 2^{k_1+k_2}$; the higher modulation component has zero contributions to $I$. Thus below we implicitly assume that all terms involved have modulations $\les 2^{k_1+k_2}$. Recalling \eqref{null}, we write  
\begin{equation} \label{nullI}
\begin{split}
I=\frac12(-I_1-I_2+I_3), \quad I_1 & = \int  (i \partial_t + \Delta ) z_{1} \cdot z_{2} \cdot \bar z_{3} dx dt, \\
I_2= \int   z_{1} \cdot (i \partial_t + \Delta )   z_{2} \cdot \bar z_{3} dx dt, \quad 
I_3 & =  \int z_{1} \cdot  z_{2} \cdot \overline{(i \partial_t + \Delta )  z_{3}} dx dt.
\end{split}
\end{equation}
Using \eqref{uvL2} and \eqref{UVX} we obtain
\[
\begin{split}
|I_2| & \les \| (i \partial_t + \Delta )   z_{2} \|_{L^2_{t,x}} \| z_{1} \cdot \bar z_{3} \|_{L^2_{t,x}} \les 
2^{\frac{k_1+k_2}2} \| z_{2}\|_{V^2_\Delta} 2^{\frac{(n-1)k_1-k_3}2} \|z_{1}\|_{S_{k_1}} \|z_{3}\|_{S_{k_3}} \\
& \les 2^{\frac{nk_1}2} \|z_{1}\|_{S_{k_1}} \|z_{2}\|_{S_{k_2}}  \|z_{3}\|_{S_{k_3}} .
\end{split}
\]
A similar estimate is available for $I_3$. For the remaining term $I_1$ we note that if we use only the $S_k$ type information, then by invoking \eqref{uvL2}, \eqref{UVX}, the Bernstein inequality and the Strichartz estimate \eqref{V2Str} for functions in $V^2_\Delta$, we obtain
\[
\begin{split}
|I_1| & \les \| (i \partial_t + \Delta )   z_{1} \|_{L^2_{t,x}} \| \tilde P_{k_1} ( z_{2} \cdot \bar z_{3}) \|_{L^2_{t,x}} \les 
2^{\frac{k_1+k_2}2} \| z_{1}\|_{V^2_\Delta}  2^{\frac{(n-2)k_1}2} \| z_{2} \cdot \bar z_3 \|_{L^2_{t} L_x^{\frac{n}{n-1}}} \\
& \les 2^{\frac{k_1+k_2}2} \| z_{1}\|_{V^2_\Delta}  2^{\frac{(n-2)k_1}2} \| z_{2} \|_{L^4_{t} L_x^{\frac{2n}{n-1}}} \| z_{3} \|_{L^4_{t} L_x^{\frac{2n}{n-1}}} \\
& \les 2^{\frac{(n-1)k_1+k_2}2} \| z_{1}\|_{V^2_\Delta} \|z_{2}\|_{V^2_\Delta} \|z_{3}\|_{V^2_\Delta} 
 \les  2^{\frac{(n-1)k_1+k_2}2} \|z_{1}\|_{S_{k_1}} \|z_{2}\|_{S_{k_2}}  \|z_{3}\|_{S_{k_3}};
\end{split}
\]
we note that above we could freely insert the operator $\tilde P_{k_1}$ on the product $z_{2} \cdot \bar z_{3}$ since $z_1$ is localized at frequency $2^{k_1}$, thus $\tilde P_{k_1} z_1 = z_1$. 

The above estimate is not as strong as required in \eqref{Ib1}. On the other hand, if the low frequency were to be localized at modulation $\leq 2^{2k_1+100}$,
that is we reset $I_1$ to be 
\[
I_1 = \int  (i \partial_t + \Delta ) Q_{\leq 2k_1 + 100} z_{1} \cdot z_{2} \cdot \bar z_{3} dx dt
\]
then estimating just as above gives us
\[
|I_1| \les 2^{\frac{nk_1}2} \|z_{1}\|_{S_{k_1}} \|z_{2}\|_{S_{k_2}}  \|z_{3}\|_{S_{k_3}},
\]
which suffices. 

Thus we are left with covering the range of modulation between $2^{2k_1+100}$ and $2^{k_1+k_2}$ for the low frequency. This is precisely where we need to bring in the stronger structure $S_{k_3,str}$; and in fact it is the only place in this argument where it is being used. Assuming that $z_1$ is localized at modulation $2^{2k_1+99} \leq |\tau-\xi^2| \les 2^{k_1+k_2}$, we estimate as follows
\[
\begin{split}
|I_1|  \les \| (i \partial_t + \Delta )   z_{1} \|_{L^2_{t,x}} \| P_{k_1}(z_{2} \cdot \bar z_{3}) \|_{L^2_{t,x}} \les 
2^{k_1} \| z_{1}\|_{S_{k_1,str}} 2^{\frac{(n-2)k_1}2} \|z_{2}\|_{S_{k_2}} \|z_{3}\|_{S_{k_3}},
\end{split}
\]
which suffices. This concludes the proof of the bound \eqref{Ib1}. 

Case 2. $|k_2-k_3| >10$. Since $k_1 \leq k_2$, the only possibility is that $k_3 \leq k_2 -11$ in which case $|k_1-k_2| \leq 10$;  we refer to $z_1,z_2$ as the high frequency factors and $z_3$ as the low frequency factor.  The target bound (from \eqref{triSk}) is the following 
\begin{equation} \label{Ib2}
| I | \les 2^{(n-1)(k_3-k_1)} 2^{\frac{n(k_1+k_2-k_3)}2} \| z_1 \|_{S_{k_1}} \| z_2 \|_{S_{k_2}}  \|z_3\|_{S_{k_3}};
\end{equation}
we already hint that we do not use of the strong structure in this estimate. If one high frequency has modulation $\geq 2^{2k_1-100}$, say $z_{1}$, then we estimate as follow 
\[
\begin{split}
|I| & \les \|\nabla z_1\|_{L^2_{t,x}} \| \nabla z_{2} \cdot \bar z_{3}\|_{L^2_{t,x}} \les 
\|z_{1}\|_{V^2_\Delta} 2^{\frac{(n-1)k_3-k_2}2} 2^{k_2} \|z_{2}\|_{S_{k_2}} \|z_{3}\|_{S_{k_3}}  \\
& \les  2^{\frac{(2n-1)(k_3-k_1)}2} 2^{\frac{n(k_1+k_2-k_3)}2} \|z_{1}\|_{S_{k_1}}  \|z_{2}\|_{S_{k_2}} \|z_{k_3}\|_{S_{k_3}}.
\end{split}
\]
If both high frequencies have modulation $\leq 2^{2k_1-100}$ then we use the decomposition \eqref{nullI} for $I$.
We estimate $I_1$ as follows
\[
\begin{split}
|I_1| & \les \| (i \partial_t + \Delta ) z_{1} \|_{L^2_{t,x}} \| z_{2} \cdot \bar z_{3}  \|_{L^2_{t,x}}
\les 2^{k_1} \|z_{1}\|_{V^2_\Delta} 2^{\frac{(n-1)k_3-k_2}2} \|z_{2}\|_{S_{k_2}} \| z_{3}\|_{S_{k_3}} \\
& \les 2^{\frac{(2n-1)(k_3-k_1)}2} 2^{\frac{n(k_1+k_2-k_3)}2} \|z_{1}\|_{S_{k_1}}  \|z_{2}\|_{S_{k_2}} \| z_{3}\|_{S_{k_3}}. 
\end{split}
\]
A similar estimate is derived for $I_2$. For $I_3$ we note that since the two high frequencies have modulation below $2^{2k_1-100}$ then $z_{k_3}$ is localized at modulation $\approx 2^{2k_1}$ (here we also use the fact that $k_3 \leq k_2-11$); the other components of $z_{k_3}$ bring zero contributions to $I_3$. Using this observation we estimate
\[
\begin{split}
|I_3| & \les \| (i \partial_t + \Delta ) z_{3} \|_{L^2_{t,x}} \| P_{k_3} (z_{1} \cdot z_2) \|_{L^2_{t,x}} 
\les 2^{k_1} \|z_{3}\|_{V^2_\Delta}  2^{\frac{(n-2)k_3}2} \| z_{1} \cdot  z_2 \|_{L^2_{t} L_x^{\frac{n}{n-1}}}  \\
& \les 2^{k_1} 2^{\frac{(n-2)k_3}2} \|z_{3}\|_{S_{k_3}} \| z_{1} \|_{L^4_{t} L_x^{\frac{2n}{n-1}}} \| z_{2} \|_{L^4_{t} L_x^{\frac{2n}{n-1}}} 
 \les 2^{k_1} 2^{\frac{(n-2)k_3}2}  \|z_{3}\|_{S_{k_3}} \|z_{1} \|_{S_{k_1}}  \|z_{2} \|_{S_{k_2}} \\
& \les  2^{(n-1)(k_3-k_1)} 2^{\frac{n(k_1+k_2-k_3)}2} \|z_{1}\|_{S_{k_1}}  \|z_{2} \|_{S_{k_2}}  \|z_{3} \|_{S_{k_3}}. 
\end{split}
\]
This completes the argument for \eqref{Ib2}.

\end{proof}

\begin{proof}[Proof of \eqref{biSk}] Just as before we split the argument into two cases, based on whether the output is high or low frequency. Due to the symmetry in $k_1,k_2$ we can freely assume $k_1 \leq k_2$. 

Case 1. $|k_3 - k_2| \leq 10$. It is easy to see that either $z_{1}$ or $z_{2}$ need to have modulation $\ges 2^{2k_3}$ or else the output vanishes. If $z_{2}$ has the high modulation, then
\[
\begin{split}
\|\nabla z_{1} \cdot \nabla z_{2} \|_{L^2_{t,x}} & \les 2^{k_1} \|z_{1}\|_{L^\infty_{t,x}} 2^{k_2} \|z_{2}\|_{L^2_{t,x}}
\les 2^{\frac{(n+2)k_1}2} \|z_{1}\|_{S_{k_1}} \|  z_{2} \|_{V^2_\Delta} \\
& \les 2^{k_1-k_2} 2^{\frac{nk_1}2}  2^{k_3} \|z_{1}\|_{S_{k_1}} \|z_{2}\|_{S_{k_2}}. 
\end{split}
\]
If $z_{1}$ has the high modulation, then
\[
\begin{split}
\|\nabla z_{1} \cdot \nabla z_{2} \|_{L^2_{t,x}} & \les 2^{k_1} \|z_{1}\|_{L^2_t L^\infty_{x}} 2^{k_2} \|z_{2}\|_{L^\infty_t L^2_{x}}
\les 2^{\frac{(n+2)k_1}2} \|z_{1}\|_{L^2_{t,x}} 2^{k_2} \|  z_{2} \|_{S_{k_2}} \\
& \les 2^{\frac{nk_1}2} 2^{k_3}    \|z_{1}\|_{S_{k_1}} \|z_{2}\|_{S_{k_2}}. 
\end{split}
\]
Case 2. $|k_3- k_2| \geq 11$. Since $k_1 \leq k_2$, the only possibility is that $k_3 \leq k_2 -11$ in which case $|k_1-k_2| \leq 10$. We can estimate directly
\[
\begin{split}
\|P_{k_3}(\nabla z_{1} \cdot \nabla z_{2}) \|_{L^2_{t,x}} & \les 2^{\frac{(n-2)k_3}2} \| \nabla z_{1} \cdot \nabla z_{2}  \|_{L^2_{t} L_x^{\frac{n}{n-1}}} \les 2^{\frac{(n-2)k_3}2} \| \nabla z_{1} \|_{L^4_{t} L_x^{\frac{2n}{n-1}}} \| \nabla z_{2}  \|_{L^4_{t} L_x^{\frac{2n}{n-1}}}  \\
& \les 2^{\frac{(n-2)k_3}2} 2^{k_1} \| z_{1} \|_{S_{k_1}} 2^{k_2} \| z_{2} \|_{S_{k_2}}. 
\end{split}
\]
This estimate suffices if $n \geq 3$, but it does not provide the off diagonal decay $2^{\frac{k_3-k_{max}}4}$ (or any off-diagonal decay of this type) in dimension $n=2$.  Below we  refine the estimate when $n=2$ so that we gain the additional factor. The starting point is the following estimate for free solutions
\begin{equation} \label{impr}
\|P_{k_3}(e^{it\Delta} f \cdot e^{it\Delta} g) \|_{L^2_{t,x}} \les 2^{\frac{k_3-k_1}2} \|f\|_{L^2_x} \|g\|_{L^2_x},
\end{equation}
where $f,g$ are supported at frequency $2^{k_1}$, respectively $2^{k_2}$. We postpone the proof of \eqref{impr} and continue with our argument. From \eqref{impr} we obtain
\[
\|P_{k_3}(\nabla z_{1} \cdot \nabla z_{2}) \|_{L^2_{t,x}} \les 2^{2k_1} 2^{\frac{k_3-k_1}2} \|z_{1}\|_{U^2_\Delta} \|z_{2}\|_{U^2_\Delta}; 
\]
indeed it is easy to see that \eqref{impr} easily generalizes to atoms in $U^2_\Delta$. On the other hand we have
\[
\|\nabla z_{1} \cdot \nabla z_{2} \|_{L^2_{t,x}} \les 2^{k_1} \|z_{1}\|_{L^4_{t,x}} 2^{k_2} \|z_{2}\|_{L^4_{t,x}} \les 2^{k_1} \|z_{1}\|_{U^4_\Delta}  2^{k_2} \|z_{2}\|_{U^4_\Delta}. 
\]
By interpolating between the two estimates above we obtain
\[
\|P_{k_3}(\nabla z_{1} \cdot \nabla z_{2}) \|_{L^2_{t,x}} \les  2^{\frac{k_3-k_1}4} 2^{k_1} \|z_{1}\|_{U^\frac83_\Delta}  2^{k_2} \|z_{2}\|_{U^\frac83_\Delta},
\]
and since $V^2_\Delta \subset U^\frac83_\Delta$, we conclude with 
\[
\|P_{k_3}(\nabla z_{1} \cdot \nabla z_{2}) \|_{L^2_{t,x}} \les  2^{\frac{k_3-k_1}4} 2^{k_1} \|z_{1}\|_{V^2_\Delta}  2^{k_2} \|z_{2}\|_{V^2_\Delta},
\]
This concludes the proof of the estimate \eqref{biSk}.

For proving \eqref{impr} we rescale so that $k_1=0$ and relabel $k_3=k$. Since we estimate $P_k(e^{it\Delta} f \cdot e^{it\Delta} g)$, it suffices to consider the case when $\hat f$ and $\hat g$ are localized in cubes in $\R^3$ of size $2^k$; the two cubes have to be symmetric with respect to the origin so that the support of $\hat f \ast \hat g$ is localized at frequency $2^k$.  The global estimate is then recovered using standard orthogonality arguments.

With the above localization in place, we use Plancherel so that we estimate $\hat f(\xi) \delta_{\tau=\xi^2} \ast \hat g(\eta) \delta_{\tau=\eta^2}$ in $L^2_{\tau,\xi}$. It is easy to see that the intersection of any translate of the support of $\hat f(\xi) \delta_{\tau=\xi^2}$ with the support of $g(\eta) \delta_{\tau=\eta^2}$ has $\mathcal{H}^{1}$ measure $\les 2^k$, and as a consequence we have
\[
\| \hat f(\xi) \delta_{\tau=\xi^2} \ast \hat g(\eta) \delta_{\tau=\eta^2} \|_{L^2} \les 2^{\frac{k}2} \| f\|_{L^2} \|g\|_{L^2}. 
\]
The formalization of this estimate can be found in many places, see for instance \cite{KTV}, Theorem 5.1, or \cite{Ca}, Theorem 
5.2. This concludes the argument for \eqref{impr}. 

\end{proof}

Now that we are done with the proofs of  \eqref{triSk} and \eqref{biSk} we highlight that the only place that we needed the $S_{k,str}$ structure was in the proof of \eqref{triSk}, Case 1, when the low frequency had the high modulation. Moreover the $S^{\frac{n}2}_{str}$ is recovered simply by using information on inputs at the level of the regular structure $S^{\frac{n}2}$, see \eqref{biSk}. The point we want to make is that our original $S^\frac{n}2$ provides almost all the estimates needed, and even the stronger structure is essentially obtained in the nonlinear equation from having inputs only in 
$S^{\frac{n}2}$. 

\subsection{The Schr\"odinger Map equation} In this section we establish the result in Theorem \ref{SMT}.  In comparison to the nonlinearity in \eqref{Mz}, here we have the additional the term $\frac{\bar z}{1+|z|^2}$. The first observation is that $\bar z$ requires the use of the space $\bar S^\frac{n}2$; this is simply defined as
\[
\bar S^\frac{n}2 = \{f: \bar f \in S^\frac{n}2 \},
\]
with the corresponding norm; in a similar fashion we have the spaces $\bar S_k$. We  note that, in our arguments, the only difference between elements in $S^{\frac{n}2}$ and $\bar S^{\frac{n}2}$, respectively $S_k$ and $\bar S_k$, is when it comes to modulation analysis; otherwise linear and bilinear estimates are the same because they involve absolute values where there is  no difference between $z$ and $\bar z$. 

The second aspect we need to address is the fraction $\frac1{1+|z|^2}$. Here we use that $S^\frac{n}2 + \bar S^\frac{n}2$ (the sum space) is an algebra, a fact established in Lemma \ref{Salg}. Using the smallness of the initial data, and hence of the solution, we obtain $\frac{\bar z}{1+|z|^2} \in S^\frac{n}2 + \bar S^\frac{n}2$, thus we treat $\frac{\bar z}{1+|z|^2}$ as an element in $S^\frac{n}2 + \bar S^\frac{n}2$. Therefore it suffices to consider the quadralinear expression 
\[
I=\int z_1 \nabla z_2 \cdot \nabla z_3 \cdot \bar z_4 dx dt,
\]
where $z_1 \in S^\frac{n}2 + \bar S^\frac{n}2$, and $z_i \in S^\frac{n}2_{str}, i=2,..,4$. We claim the following two estimates
\begin{equation} \label{triSka} 
| I | \les 2^{k_4-k_{max}} 2^{\frac{n(k_1+k_2+k_3-k_4)}2} \|z_1 \|_{S_{k_1} + \bar S_{k_1}} \| z_2 \|_{S_{k_2,str}} 
\| z_3 \|_{S_{k_3,str}}  \|z_4\|_{S_{k_4}},
\end{equation}
and
\begin{equation} \label{biSka} 
\begin{split}
 & \| P_{k_4} Q_{\geq 2k_4 + 100} (z_1 \nabla z_2 \cdot \nabla z_3) \|_{L^2_{t,x}}  \\
 \les & \ 2^{\frac{k_4-k_{max}}4} 2^{k_4} 2^{\frac{n(k_1+k_2+k_3-k_4)}2}\| z_{1} \|_{S_{k_1}+\bar S_{k_1}} \| z_2 \|_{S_{k_2}} \| z_3 \|_{S_{k_3}}. 
\end{split}
\end{equation}
In the above we assume that each $z_i, i=1,..,4$ is localized at frequency $2^{k_i}$ and $k_{max}=\max(k_1,k_2,k_3,k_4)$. 
We also use the notation $k_{min}=\min(k_1,k_2,k_3,k_4)$

Assuming \eqref{triSka} and \eqref{biSka}, along with Lemma \ref{Salg}, the proof of Theorem \ref{SMT} is similar to the proof of \ref{MzT} which was detailed in the previous section. The rest of this section  is dedicated to the proofs of \eqref{triSka} and \eqref{biSka}.

\begin{proof}[Proof of \eqref{triSka}]

Due to the symmetry in $k_2,k_3$ we can assume $k_2 \leq k_3$. The argument is organized based on which factor has  the lowest frequency. 

Case 1. $k_2=k_{min}$. We let note that two of the factors $k_2,k_3,k_4$ have to be comparable with $k_{max}$ in the sense $|k_i-k_{max}| \leq 10$ (obviously one of them is $k_{max}$). We let $k_{med}$ be the fourth value which has the property $k_{min} \leq k_{med} \leq k_{max}$.
Using \eqref{uvL2} we obtain the estimate 
\[
\begin{split}
|I| & \les 2^{\frac{(n-1)k_2 - k_{max}}2} 2^{\frac{(n-1)k_{med} - k_{max}}2} 2^{k_2+k_3} \|z_1\|_{S_{k_1} + \bar S_{k_1}} 
\|z_2\|_{S_{k_2}} \|z_3\|_{S_{k_3}} \|z_4\|_{S_{k_4}} \\
& \les 2^{\frac{k_{min}-k_{med}}2} 2^{k_4 - k_{max}} 2^{\frac{n(k_1+k_2+k_3-k_4)}2} \|z_1\|_{S_{k_1} + \bar S_{k_1}}  \|z_2\|_{S_{k_2}} \|z_3\|_{S_{k_3}} \|z_4\|_{S_{k_4}} \\
\end{split}
\] 

Case 2. $k_1=k_{min}$ and $k_1 \leq k_4-10$ (the subcase $k_1 > k_4-10$ will be treated in Case 4 below). Here we simply note that 
\[
\|z_{1} \cdot \bar z_{4}\|_{\bar S_{k_4}} \les 2^{\frac{nk_1}2} \|z_1\|_{S_{k_1}+\bar S_{k_1}} \|z_4\|_{S_{k_4}},
\]
and the output $z_{1} \cdot \bar z_{4}$ is localized at frequency $2^{k_4}$. This is in fact what is being established in the following section, in the proof of the claim that $S^{\frac{n}2}+ \bar S^{\frac{n}2}$ is an algebra; essentially it says that the high frequency input dictates the nature of the output (where in $S_{k}$ or $\bar S_k$). 

With the above estimate at hand, we can replace $z_{1} \cdot \bar z_{4}$ with an element $\bar w_4 $ where $w_4 \in S_{k_4}$ is localized at frequency $\approx 2^{k_4}$. Now we are in the setup of  \eqref{triSk}, hence we obtain
\[
|I| \les 2^{k_4-k_{max}} 2^{\frac{n(k_1+k_2+k_3-k_4)}2} \|z_1\|_{S_{k_1}+\bar S_{k_1}} \|z_2\|_{S_{k_2}} \|z_3\|_{S_{k_3}} \|z_4\|_{S_{k_4}}. 
\]

Case 3. $k_4=k_{min}$. Here we use \eqref{uvL2} to estimate
\[
\begin{split}
|I| & \les 2^{\frac{(n-1)k_4 - k_{max}}2} 2^{\frac{(n-1)k_{med} - k_{max}}2} 2^{k_2+k_3} \|z_1\|_{S_{k_1} + \bar S_{k_1}} 
\|z_2\|_{S_{k_2}} \|z_3\|_{S_{k_3}} \|z_4\|_{S_{k_4}} \\
& \les 2^{k_{4}-k_{max}} 2^{\frac{n(k_1+k_2+k_3-k_4)}2} \|z_1\|_{S_{k_1} + \bar S_{k_1}}  \|z_2\|_{S_{k_2}} \|z_3\|_{S_{k_3}} \|z_4\|_{S_{k_4}} \\
\end{split}
\]

Case 4. There is one remaining case which was not covered in Case 3, that is  $k_1=k_{min}$ and $k_1 \geq k_4-10$. But this is similar to Case 3 above. 

\end{proof}

\begin{proof}[Proof of \eqref{biSka}] Due to the symmetry with respect to $k_2,k_3$ we assume that $k_2 \leq k_3$.
We denote by $\bar k_{min}=\min(k_1,k_2,k_3)$ and $\bar k_{max}=\max(k_1,k_2,k_3)$.
Also we let $\bar k_{med} \in \{k_1,k_2,k_3\}$ be such that $\bar k_{min} \leq \bar k_{med} \leq \bar k_{max}$. 
We notice that in order to have non-zero outputs in the expression $P_{k_4} (z_1 \nabla z_2 \cdot \nabla z_3)$, we should have $k_4 \leq \bar k_{max} + 10$. 
We organize the argument mainly base on the relation between $k_4$ and $\bar k_{min}$ and 
$\bar k_{max}$. 

Case 1. $k_4 \leq \bar k_{min}+10$; in this case we also have $|k_{max}-\bar k_{max}| \leq 100$. We use the Bernstein inequality, the bilinear $L^2_{t,x}$ estimate \eqref{uvL2} for the medium and high frequency product, and $L^\infty_t L^2_x$ for the remaining high frequency, to obtain the estimate
\[
\begin{split}
\| P_{k_4} (z_1 \nabla z_2 \cdot \nabla z_3) \|_{L^2_{t,x}}  
& \les 2^{\frac{nk_4}2} \| z_1 \nabla z_2 \cdot \nabla z_3 \|_{L^2_{t} L^1_x} \\
& \les 2^{\frac{nk_4}2} 2^{k_2+k_3} 2^{\frac{(n-1)k_{med}-k_{max}}2} \|z_1\|_{S_{k_1}+ \bar S_{k_1}} \|z_2\|_{S_{k_2}} \|z_3\|_{S_{k_3}} \\
& \les 2^{(n-1)(k_4-k_{max})} 2^{k_4} 2^{\frac{n(k_1+k_2+k_3-k_4)}2} \|_{S_{k_1}+ \bar S_{k_1}} \|z_2\|_{S_{k_2}} \|z_3\|_{S_{k_3}}. 
\end{split}
\]

Case 2. $k_4 \geq \bar k_{max}-10$; here we also have $|k_4-\bar k_{max}| \leq 10$.  Using the Bernstein inequality for the minimum frequency term (which is placed in $L^\infty_{t,x}$) and the bilinear $L^2_{t,x}$ estimate  \eqref{uvL2} for the medium and high frequency product, allows us to estimate as follows
\[
\begin{split}
\| z_1 \nabla z_2 \cdot \nabla z_3 \|_{L^2_{t,x}}  & \les 2^{\frac{n \bar k_{min}}2} 2^{\bar k_{med}+ \bar k_{max}} 2^{\frac{(n-1)\bar k_{med}- \bar k_{max}}2} \|z_1\|_{S_{k_1}} \|z_2\|_{S_{k_2}} \|z_3\|_{S_{k_3}} \\
& \les 2^{\frac{\bar k_{med}-\bar k_{max}}2} 2^{k_4} 2^{\frac{n(k_1+k_2+k_3-k_4)}2} \|_{S_{k_1}} \|z_2\|_{S_{k_2}} \|z_3\|_{S_{k_3}}.
\end{split}
\]

Case 3. $\bar k_{min} + 10 \leq k_4 \leq \bar k_{max} -10$. We further split this into two subcases.

Case 3a. $k_{1} \geq \bar k_{max}-10$; in this case we have that $k_2 \leq k_4 +100$. We estimate in a similar fashion to Case 2 above
\[
\begin{split}
\| z_1 \nabla z_2 \cdot \nabla z_3 \|_{L^2_{t,x}}  & \les \| z_1\|_{L^\infty_{t,x}} \| \nabla z_2 \nabla z_3\|_{L^2_{t,x}} \\
& \les 2^{\frac{nk_1}2} \|z_1\|_{S_{k_1}} 2^{k_2} 2^{k_3} 2^{\frac{(n-1)k_{2}-k_{3}}2} \|z_1\|_{S_{k_1}+ \bar S_{k_1}} \|z_2\|_{S_{k_2}} \|z_3\|_{S_{k_3}} \\
& \les 2^{\frac{k_2-k_3}2} 2^{k_4} 2^{\frac{n(k_1+k_2+k_3-k_4)}2} \|z_1 \|_{S_{k_1} + \bar S_{k_1}} \|z_2\|_{S_{k_2}} \|z_3\|_{S_{k_3}} \\
& \les 2^{\frac{k_4-k_{max}}2} 2^{k_4} 2^{\frac{n(k_1+k_2+k_3-k_4)}2} \|z_1\|_{S_{k_1} + \bar S_{k_1}} \|z_2\|_{S_{k_2}} \|z_3\|_{S_{k_3}}. 
\end{split}
\]

Case 3b. $k_{1} < \bar k_{max}-11$; here we also have $|k_2-k_3| \leq 10$ and $k_1 = \bar k_{min} \leq k_4 -10$. As a consequence we have
\[
P_{k_2} (z_1 \nabla z_2 \cdot \nabla z_3)= P_{k_2} (z_1 \tilde P_{k_2} (\nabla z_2 \cdot \nabla z_3)).
\]
By simply estimating $\|z_1\|_{L^\infty_{t,x}} \les 2^{\frac{nk_1}2} \|z_1\|_{S_{k_1}+ \bar S_{k_1}}$, this term is then estimated as in Case 2 of the proof of \eqref{biSk}.

This completes the proof of the claim in \eqref{biSka}. 

\end{proof}

\subsection{The algebra property} We end this section by establishing the algebra property for $S^\frac{n}2$ and $S^\frac{n}2 + \bar S^\frac{n}2$.

\begin{lemma} \label{Salg}
$S^\frac{n}2$ is an algebra. Also $S^\frac{n}2 + \bar S^\frac{n}2$ is an algebra.
\end{lemma}

In order to proceed with the proof of this result, we need some technical ingredients about commutator estimates. 
Following \cite{tao}, we adopt the following notation: given a 
kernel $K \in L^1((\R^2)^M)$, we define the multilinear operator
\[
L_O(f_1,..,f_m)(x,t) = \int K(y_1,..,y_M) f_1(x-y_1,t) \cdot ... \cdot f_1(x-y_M,t) dy_1 ... dy_M.  
\]
We use of the following result, see  \cite[Lemma 1]{tao}:
\begin{lemma} \label{LO}
Let $X_1,..,X_M, X$ be translation-invariant Banach spaces such that we have the product estimate
\begin{equation} 
\|f_1 \cdot ... \cdot f_M \|_{X} \leq A \|f_1\|_{X_1} \cdot ... \cdot \| f_M \|_{X_M},
\end{equation}
for all scalar-valued $f_i \in X_i, i \in \{1,..,M\}$ and for some constant $C$. Then we have  
\begin{equation} 
\|L_O(f_1, ..., f_M) \|_{X} \leq C(m,M) A \|f_1\|_{X_1} \cdot ... \cdot \| f_M \|_{X_M},
\end{equation}
for all $f_i \in X_i, i \in \{1,..,M\}$ that are scalar-valued, $m$-dimensional vectors or $m \times m$ matrices. 
\end{lemma}

We also need the following Leibnitz rule for Fourier projectors, see \cite[Lemma 2]{tao}:
\begin{lemma} \label{commute}
If $P'_k$ is any Fourier multiplier with smooth symbol and whose support lies in the region $|\xi| \approx 2^k$, 
then the following commutator identity holds true:
\[
P_k'(fg)= f P_k' g + L_O(\nabla f, 2^{-k} g). 
\]
\end{lemma}

Now we are ready to proceed with the proof of the main result in this section. 

\begin{proof}[Proof of Lemma \ref{Salg}] We first provide a proof of the fact that $S^\frac{n}2$ is an algebra.

Assume that $u,v$ are localized at frequency $2^{k_1}$, respectively $2^{k_2}$ with $k_1 \leq k_2$; it suffices to normalize $k_2=0$ so as to keep the numerology simpler. 

We begin by investigating the low-high to high scenario, that is
 $P_k(uv)$ for $|k-k_2| = |k| \leq 100$. From the bilinear estimate \eqref{uvL2} we have the estimate:
\[
\| u v \|_{L^2} \les 2^{\frac{(n-1)k_1}2} \| u \|_{S_{k_1}} \| v \|_{S_{0}}.
\]
As a consequence we obtain
\[
\| Q_{\preceq k_1} ( u v) \|_{X^{0,\frac12,1}} 
\les 2^{\frac{k_1}2} \| Q_{\prec k_1} (u v) \|_{L^2} 
\les  2^{\frac{nk_1}2} \| u \|_{S_{k_1}} \| v \|_{S_{0}}. 
\]
Since $X^{0,\frac12,1} \subset S_k$ we obtain
\[
\| Q_{\preceq k_1+k_2} ( u v) \|_{S_k} \les  2^{\frac{nk_1}2} \| u \|_{S_{k_1}} \| v \|_{S_{0}}. 
\]
The argument for $Q_{\succ k_1}(uv)$ is the lengthier one. We note that $Q_{\succ k_1} ( Q_{\leq k_1} u \cdot Q_{\leq k_1}  v)=0$. Thus in estimating 
$Q_{\succ k_1} (uv)$ we can assume that one of the factors has modulation $\geq 2^{k_1}$. 

\bf The estimate in $V^2_\Delta$. \rm Assume that $u$ has the high modulation structure. Here we write $u = e^{it\Delta} \tilde u, v = e^{it\Delta} \tilde v$ where $\tilde u, \tilde v \in V^2$. As a consequence we have
\[
\begin{split}
\mathcal{F}_x( e^{-it\Delta} (u v)) (t,\zeta) & = \int e^{it \zeta^2} \hat u(t,\xi) \hat v(t,\zeta-\xi) d\xi 
= \int e^{it (\zeta^2-\xi^2-(\zeta-\xi)^2)} \hat{\tilde u}(t,\xi) \hat{\tilde v}(t,\zeta-\xi) d\xi \\
& = \int e^{it (\zeta^2-\xi^2-(\zeta-\xi)^2)} m(\xi,\zeta-\xi) \hat{\tilde u}(t,\xi) \hat{\tilde v}(t,\zeta-\xi) d\xi,
\end{split}
\]
where $m$ is inserted so as to reflect the support conditions $|\xi| \approx 2^{k_1}, |\zeta|, |\zeta-\xi| \approx 1$. 
At this point we can invoke Theorem 5.6 in \cite{CaHe} to conclude the following
\begin{equation}
\| e^{-it\Delta} (uv)  \|_{V^2} \les \left (\sup_{t \in \R} \| e^{it (\zeta^2-\xi^2-\eta^2)} m(\xi,\eta) \|_{L^\infty_\xi L^2_\eta+L^\infty_\eta L^2_\xi} \right) \|\tilde u\|_{V^2} \| \tilde v \|_{V^2} \les 2^{\frac{nk_1}2} \| u\|_{V^2_\Delta} \|  v \|_{V^2_\Delta}. 
\end{equation}

For the rest of the argument we consider two separate cases. Recall that we seek to estimate $Q_{\geq m} (u v)$ where $m \succ k_1$. Given the constraint $m \succ k_1$ is is easy to see that either $u$ or $v$ has to have modulation $\geq 2^{m-10}$; we abuse notation and use the same $\geq m$ threshold.

\bf The term $Q_{\geq m} (u Q_{\geq m} v)$. \rm The refined mass estimate is straightforward:
\[
\| Q_{\geq m} (u Q_{\geq m} v) \|_{E_0^{\geq m}} \les \|u\|_{L^\infty_{t,x}} \| Q_{\geq m} v \|_{E_0^{\geq m}} \les 
2^{\frac{nk}2} \|u\|_{L^\infty_t L^2_x} \| Q_{\geq m} v \|_{E_0^{\geq m}}.
\]
For the local smoothing one has to be careful since multiplication by $u$ may alter the Fourier support of $P_{0,\e} v$. Here we invoke the commutator estimate in Lemma \ref{commute} to obtain
\[
P_{0,\e} (u v) = u  P_{0,\e} v + L_{O} (\nabla u, v). 
\]
Just as above, we obtain
\[
\| u  P_{0,\e} v \|_{E_{0,\e}^{\geq m}} \les 2^{\frac{nk}2} \|u\|_{L^\infty_t L^2_x} \| Q_{\geq m} v \|_{E_{0,\e}^{\geq m}}.
\]
For the second term $L_{O} (\nabla u, v)$ we derive a stronger structure which will imply the desired bound.  We begin with the estimate
\[
\|\nabla u \cdot v \|_{L^2_{t,x}} \les  2^{k} \| u_0\|_{L^\infty_{t,x}} \| v_0 \|_{L^2_{t,x}} \les  2^{\frac{nk}2} \|u\|_{L^\infty_t L^2_x} 2^\frac{k}2 2^{\frac{k-m}2} \| v\|_{V^2_\Delta} \les 2^{\frac{nk}2} \|u\|_{L^\infty_t L^2_x}  \| v\|_{V^2_\Delta},
\]
where in the last line we used the fact that $v$ is localized at modulation $2^m \ges 2^k$ and that $2^k \les 1$. Our spaces are translation invariant, thus on behalf of Lemma \ref{LO} we can extend this result to the whole second term above
\[
\| L_{O} (\nabla u, v) \|_{L^2_{t,x}} \les  2^{\frac{nk}2} \|u\|_{L^\infty_t L^2_x}  \| v\|_{V^2_\Delta}. 
\]
From this, using Bernstein in the direction of $\e$, we can derive
\[
\|P_{0,\e} L_{O} (\nabla u, v) \|_{E_{0,\e}^{\geq 0}} \les  2^{\frac{nk}2} \|u\|_{L^\infty_t L^2_x}  \| v\|_{V^2_\Delta}, 
\]
which implies the bound for all $m$ with $k \leq m \leq 0$. This concludes all bounds for the term $Q_{\geq m} (u Q_{\geq m} v)$.

\bf The term $Q_{\geq m} (Q_{\geq m} u \cdot v)$. \rm Since we have already dealt with the high modulation on $v$, it suffices to provide an argument for the term $Q_{\geq m} (Q_{\geq m}  u  Q_{\leq m} v)$. From the structure of $E_k^{\geq m}$ we have the estimate
\[
\sum_{I \in \mathcal{I}_{-m}}  \| Q_{\geq m} u \|_{L^\infty_t L^2_x(I)}^2 \les \|u\|_{E_{k}^{\geq m}}^2.
\]
Since, just as above, we rely only on $L^\infty_t L^2_x$ bounds for $u$, it suffices to provide the desired estimate within a single interval $I \in \mathcal{I}_{-m}$; the $l^2_I$ summability provided above suffices for the rest of the argument. 

On the other hand, on behalf of \eqref{V2X121}, on an interval $I \in \mathcal{I}_{-m}$, $Q_{\leq m} v$ has the better structure in $U^2_\Delta$, that is $\| \tilde \chi_I Q_{\leq m} v\|_{X^{0,\frac12,1}} \les \| Q_{\leq m} v\|_{L^\infty_t L^2_x} \les \| Q_{\leq m} v\|_{V^2_\Delta}$ and as a consequence
\begin{equation} \label{aux102}
\sup_{I \in \mathcal{I}_{-m}} \|  \tilde \chi_I Q_{\geq m} v \|^2_{E_k^{\geq m}} \les \| Q_{\leq m} v\|_{V^2_\Delta}.
\end{equation}
This inequality should not be misread in the sense that we claim a high modulation structure for a low modulation component. 
There is no summability with respect to the time intervals, only a $\sup_{I \in \mathcal{I}_{-m}}$. What we gain is a summability with respect to $q \in \mathcal{C}_{-m}$ which is a property that free waves have, hence functions in $X^{0,\frac12,1}$ do too. 

From this point on we proceed as above, just that we work on a single interval to obtain
\[
 \|  \chi_I Q_{\geq m} u \cdot Q_{\leq m} v \|^2_{X_k^{\geq m}} \les 2^{\frac{nk}2}\|  \chi_I Q_{\geq m} u \|_{L^\infty_t L^2_x} \cdot \| v\|_{V^2_\Delta}.
\]
Summing this in an $l^2_I$ fashion and using \eqref{aux2} gives the estimate
\[
\|  Q_{\geq m} u \cdot Q_{\leq m} v \|^2_{X_k^{\geq m}} \les 2^{\frac{nk}2}\|  Q_{\geq m} u \|_{E_k^{\geq m}} \cdot \| v\|_{V^2_\Delta}.
\]
Thus so far we provided the estimate in the low-high to high regime that is
\begin{equation} \label{lhha}
\| P_k(u v) \|_{S_k} \les 2^{\frac{nk_1}2} \| u \|_{S_{k_1}} \| v \|_{S_{k_2}},
\end{equation}
where  $u,v$ are localized at frequency $2^{k_1}$, respectively $2^{k_2}$ with $k_1 \leq k_2$ and $|k-k_2| \leq 100$. 

Next we sketch the high-high to low scenario that is
\begin{equation} \label{hhla}
\| P_k(u v) \|_{S_k} \les 2^{\frac{nk_1}2} \| u \|_{S_{k_1}} \| v \|_{S_{0}},
\end{equation}
where  $u,v$ are localized at frequency $2^{k_1}$, respectively $1$ with $|k_1| \leq 10$ and $k \leq -100$. 
The argument here is easier than the one above. 

Without restricting the generality of the argument we assume $k_1=0$. From the bilinear estimate in $L^2_{t,x}$ we obtain
\[
\|Q_{\leq 100} P_k (u v)\|_{S_k} \les \| Q_{\leq 100} P_k (u v)\|_{X^{0,\frac12,1}} \les \| u v\|_{L^2_{t,x}} \les \|u\|_{X_0} \|v\|_{X_0}.
\]
For any $\e \in \S^{n-1}$ we use Bernstein to obtain 
\[
\| Q_{\geq 0} P_k (uv) \|_{E_{k,\e}^{\geq 0}} \les \| P_k (uv) \|_{L^2_{t,x}} \les \|u\|_{X_0} \|v\|_{X_0}.
\]
Thus all that is left is the estimate in $E_{k}^{\geq 2k}$. Here we note that for the interaction $Q_{\geq 100} P_k (u v)$ to be nonzero, either $u$ or $v$ needs to be localized at high modulation $\geq 0$; assume that term is $u$, that is we estimate $Q_{\geq 100} P_k ( Q_{\geq 0} u v)$. In that case we simply use
\[
\begin{split}
& \| Q_{\geq 100} P_k ( Q_{\geq 0} u \cdot v) \|_{E_{k}^{\geq 2k}} \les \| Q_{\geq 100} P_k ( Q_{\geq 0} u \cdot v) \|_{E_{0}^{\geq 0}} \\
\les & \| Q_{\geq 0} u \cdot  v \|_{E_{0}^{\geq 0}} \les \| Q_{\geq 0} u \|_{E_{0}^{\geq 0}} \|v\|_{L^\infty_{t,x}} \les \|u\|_{S_0} \|v\|_{S_0}. 
\end{split}
\]
This finishes the argument for \eqref{hhla}. Using \eqref{lhha} and \eqref{hhla} a standard argument gives that $S^{\frac{n}2}$ is an algebra. 

For the fact that $S^\frac{n}2 + \bar S^\frac{n}2$ is an algebra we note that an inspection of the argument for \eqref{lhha} shows that it works if the low frequency is replaced by $\bar u$, and this implies the following
\begin{equation} \label{lhha2}
\| P_k(u v) \|_{S_k} \les 2^{\frac{nk_1}2} \| u \|_{\bar S_{k_1}} \| v \|_{S_{k_2}},
\end{equation}
where  $u,v$ are localized at frequency $2^{k_1}$, respectively $2^{k_2}$ with $k_1 \leq k_2$ and $|k-k_2| \leq 100$. On the other hand, in the same localization regime, by direct conjugation we obtain
\begin{equation} \label{lhha3}
\| P_k(u v) \|_{\bar S_k} \les 2^{\frac{nk_1}2} \| u \|_{S_{k_1}} \| v \|_{\bar S_{k_2}}, \quad 
\| P_k(u v) \|_{\bar S_k} \les 2^{\frac{nk_1}2} \| u \|_{\bar S_{k_1}} \| v \|_{\bar S_{k_2}}. 
\end{equation}
Things are simpler for the high-high to low regime, where an inspection of the argument for \eqref{hhla} shows we can replicate to establish
\begin{equation} \label{hhla2}
\| P_k(u v) \|_{S_k} \les 2^{\frac{nk_1}2} \| u \|_{S_{k_1}} \| v \|_{\bar S_{k_2}}, \quad
\| P_k(u v) \|_{S_k} \les 2^{\frac{nk_1}2} \| u \|_{\bar S_{k_1}} \| v \|_{\bar S_{k_2}},
\end{equation}
where  $u,v$ are localized at frequency $2^{k_1}$, respectively $2^{k_2}$ with $|k_1 - k_2| \leq 10$ and $k \leq k_2-100$. 

These estimates suffice to conclude that $S^\frac{n}2 + \bar S^\frac{n}2$ is an algebra. 

\end{proof}

\section{Appendix}

In this section we briefly explain how to obtain the result in Lemma \ref{bilR}, stating that for any $p > \frac{n+3}{n+1}$ and any cube $Q \subset \R^{n+1}$ of size $2^{-2k} R$ the following holds true:
\begin{equation} \label{bilR2}
\left( \sum_{q \in \mathcal{C}_{-2k}: q \subset Q } \| e^{it\Delta} f \cdot  e^{it\Delta} g \|^p_{L^2(q)} \right)^\frac1p \leq C(p)  2^{\frac{(n-1)k}2}  \|  f\|_{L^2} \|  g \|_{L^2},
\end{equation}
where $f$ is assumed to be localized at frequency $2^k \les 1$ and $g$ is assumed to be localized at frequency $\approx 1$. The standard estimate that appears in the literature  replaces $L^2_{t,x}(q)$ with $L^p_{t,x}(q)$, which can be morally achieved using H\"older and Bernstein type estimates. The exact implementation does not work exactly this way and instead we unravel the actual proofs to obtain \eqref{bilR2}. The important observation is that regardless on how one runs the induction argument, the centerpiece is an improved bilinear $L^2_{t,x}$ estimate for some well-chosen components. This estimate is then interpolated with some trivial estimates to obtained the desired 
$L^p_{t,x}$; it is in this step that we modify slightly the argument so as to obtain the version in \eqref{bilR2}.

The $L^p_{t,x}$ version of \eqref{bilR2} was originally established in \cite{Tao-BP} in the full range (up to the end-point). Here we follow the approach in \cite{Be-bil} and \cite{Ca}, which in turn is based on wave-tables introduced by Tao in \cite{Tao-BW} in the context of the wave equation. This has the advantage of providing the estimate without an $\epsilon$ loss, thus avoiding the need of an $\epsilon$-removal type argument which seems to be rather delicate in the present context. Otherwise, we do not bring anything new in the argument for the estimate \eqref{bilR2} since, for all practical reasons, the estimate can be derived from the existing ones in \cite{Be-bil} and \cite{Ca} with minimal adjustments. This is why we provide a sketch of the argument highlighting the only adjustment that we need, see Step 3 below; everything else is a verbatim replica  of the arguments in \cite{Be-bil} or \cite{Ca}.

The argument uses an induction on scale, that is for any $R \geq 2^{-2k}$ and any  $p > \frac{n+3}{n+1}$  we let $A_p(R)$ be the best constant in the estimate
\[
\left( \sum_{q \in \mathcal{C}_{-2k}: q \subset Q } \| e^{it\Delta} f \cdot  e^{it\Delta} g \|^p_{L^2(q)} \right)^\frac1p \leq A_p(R) 2^{\frac{(n-1)k}2} \|  f\|_{L^2_x} \|  g \|_{L^2_x},
\]
for any cube $Q$ of radius $2^{-2k}R$. To keep the argument compact, we ignore the margin concept which is a technicality needed due to the use of wave packets which alter a bit the Fourier support of the underlying wave at each induction step. In the absence of the margin requirement $A_p(R)$ is an increasing function; when the margin is taken into account one uses instead a slight modification of $A_p(R)$, but the argument is essentially the same.

\bf Step 1: Base case. \rm The case $R \les 1$ is the base case, the $l^p$ summation is trivial and it suffices to focus on a single cube $q$. We simply invoke the bilinear $L^2_{t,x}$ for free waves to obtain $A_p(R) \les 2^{\frac{(n-1)k}2}$. 

\bf Step 2: Inductive argument - wave table preparation. \rm Let $Q \subset \R^{n+1}$ be a cube of size $R$. Given $j \in \N$ we split $Q$ into $2^{(n+1)j}$ cubes of size $2^{-j} R$ and denote this family by $\calQ_j(Q)$; thus we have $Q=\cup_{B \in \calQ_j(Q)} B$.  
If $j \in \N$ and $0 \leq c \ll 1$ we define the $(c,j)$ interior $I^{c,j}(Q)$ of $Q$ by
\begin{equation} \label{icj}
I^{c,j}(Q) := \bigcup_{B \in \calQ_j(Q)} (1-c) B.
\end{equation}
Here by $(1-c)B$ we mean the cube obtained by rescaling $B$ by a factor $1-c$ from its center.  Given $j \in \N$ we define a table $\Phi$ on $Q$ to be a vector $\Phi=(\Phi^{B})_{B \in \calQ_j(Q)}$, where each $\Phi^B$ is a free wave and define its mass by
\[
M(\Phi) = \sum_{B \in \calQ_j(Q)} M(\Phi^{B}); 
\]
for a free wave $M(e^{it \Delta} f)= \|f\|_{L^2_x}^2= \|e^{it \Delta} f\|_{L^\infty_t L^2_x}^2$.

Proposition 5.1 in \cite{Be-bil} or Theorem 5.1 in \cite{Ca} (the latter needs to be combined with Theorem 9.3 to extract the improved $L^2_{t,x}$ bounds below) provide the following result. Given a cube $Q$ of size $2^{-2k} R$, there exists a decomposition of $\phi=e^{it\Delta} f, \psi=e^{it\Delta} g$ as follows
\[
\phi=\sum_{B \in \calQ_{C_0}(Q) } \Phi^{B}, \quad \psi =\sum_{B \in \calQ_{C_0}(Q) } \Psi^{B}
\]
with the following properties:

i) For every $B \in \calQ_{C_0}(Q) $, $\phi^B, \psi^B$ are free waves and
\[
M(\Phi) \leq (1+Cc) M(\phi), \quad M(\Psi) \leq (1+Cc) M(\psi).
\]

ii) For every $B',B'' \in \calQ_{C_0}(Q), B' \neq B''$ the following holds true
\begin{equation} \label{phiBpsi}
\| \Phi^{B'} \psi \|_{L^2_{t,x} ((1-c) B'')} \les c^{-C} (2^{-2k} R)^{-\frac{n-1}4} M(\phi)^\frac12 M(\psi)^\frac12, 
\end{equation}
and 
\[
\| \Phi^{B'} \Psi^{B''} \|_{L^2_{t,x} ((1-c) B')} \les c^{-C} (2^{-2k} R)^{-\frac{n-1}4} M(\phi)^\frac12 M(\psi)^\frac12. 
\]
The constants above are as follows. $C_0$ is a constant needed to ensure that the induction hypothesis at scale $\frac{R}2$ applies to $B$'s; $C$ is a universal constant that builds up from the use of $\les$ in various places; $c$ is a parameter to be thought as $R^{-}$ (a small negative power of $R$) that will ensure that the accumulation of terms $(1+Cc)$ at each inductive step is controlled. 

\bf Step 3: Inductive argument - the off diagonal interactions. \rm We remark this is the only part where we need to slightly modify the arguments in \cite{Be-bil} and \cite{Ca}. We look into the term $\Phi^{B'} \psi$ and provide the estimate in our structure. We re-write \eqref{phiBpsi} as follows
\[
\left( \sum_{q \in \mathcal{C}_{-2k}: q \subset (1-c) B''} \| \Phi^{B'} \psi \|^2_{L^2_{t,x} (q)} \right)^\frac12 \les c^{-C} (2^{-2k} R)^{-\frac{n-1}4} M(\phi)^\frac12 M(\psi)^\frac12.
\]
For any fixed time interval $I$ of length $2^{-2k}$, we also have the following
\[
 \sum_{q \in \mathcal{C}_{-2k}: q \cap (I \times \R^n) \neq \emptyset} \| \Phi^{B'} \psi \|_{L^2_{t,x} (q)}  \les 2^{\frac{(n-1)k}2} M(\phi)^\frac12 M(\psi)^\frac12;
\]
This is a consequence of the fact that, for fixed $I$ of size $2^{-2k}$, the bilinear $L^2_{t,x}$ is essentially localized in cubes $q$ due to the fact that the speed of propagation for both waves is $\les 1$. Since there are $\approx R$ time intervals needed to cover $Q$,  from the above estimate we obtain 
\[
 \sum_{q \in \mathcal{C}_{-2k}: q \subset Q} \| \Phi^{B'} \psi \|_{L^2_{t,x} (q)}  \les 2^{\frac{(n-1)k}2} R M(\phi)^\frac12 M(\psi)^\frac12.
\]
Interpolating between the $l^2_q$ and $l^1_q$ estimates above leads to
\[
\left( \sum_{q \in \mathcal{C}_{-2k}: q  \subset (1-c) B''} \| \Phi^{B'} \psi \|^p_{L^2_{t,x} (q)} \right)^\frac1p  \les 2^{\frac{(n-1)k}2} R^{\frac{n+3}2((\frac1p-\frac{n+1}{n+3}))} M(\phi)^\frac12 M(\psi)^\frac12,
\]
which we recall that it holds true for any $B'' \neq B'$. In a similar manner we obtain
\[
\left( \sum_{q \in \mathcal{C}_{-2k}: q  \subset (1-c) B'} \| \Phi^{B'} \Psi^{B''} \|^p_{L^2_{t,x} (q)} \right)^\frac1p  \les 2^{\frac{(n-1)k}2} R^{\frac{n+3}2((\frac1p-\frac{n+1}{n+3}))} M(\phi)^\frac12 M(\psi)^\frac12,
\]
for any $B'' \neq B'$.

\bf Step 4: Inductive argument - conclusion. \rm The rest of the inductive argument is identical to the arguments in \cite{Be-bil} and \cite{Ca}. Using the estimates from Step 3 and the induction hypothesis for the diagonal interactions $\Phi^{B} \Psi^{B} $, we estimate as follows
\[
\begin{split}
& \left( \sum_{q \in \mathcal{C}_{-2k}: q \cap I^{c,C_0}(Q) \neq \emptyset} \| \phi \cdot \psi \|^p_{L^2_{t,x}(q)} \right)^\frac1p 
 \leq   \sum_{B \in \calQ_{C_0}(Q)}  A_p(\frac{R}2) M(\Phi^{B})^\frac12 M( \Psi^{B} )^\frac12 +  K \\
 \leq & A_p(\frac{R}2) \left( \sum_{B \in \calQ_{C_0}(Q)} M(\Phi^{B}) \right)^\frac12 
\left( \sum_{B \in \calQ_{C_0}(Q)} M(\Psi^{B}) \right)^\frac12 +  K \\
 \leq & A_p(\frac{R}2) M(\Phi)^\frac12 M(\Psi)^\frac12 
+ C c^{-C} R^{\frac{n+3}2(\frac1p-\frac{n+1}{n+3})}  M(\phi)^\frac12 M(\psi)^\frac12 \\
 \leq & \left( (1+cC)  A_p(\frac{R}2) + C c^{-C} R^{\frac{n+3}2(\frac1p-\frac{n+1}{n+3})} \right) M(\phi)^\frac12 M(\psi)^\frac12, 
\end{split}
\] 
where we used the notation $K=C c^{-C} R^{\frac{n+3}2(\frac1p-\frac{n+1}{n+3})}  M(\phi)^\frac12 M(\psi)^\frac12$ for the contributions estimated at Step 3. 

Using a version of Lemma 2.3 in \cite{Be-bil} adapted to our setup, we have that, given a cube $\bar Q$ of size $R$ there exists a cube $Q$ of size $2R$ contained in $4\bar Q$ such that
\[
\left( \sum_{q \in \mathcal{C}_{-2k}: q \cap \bar Q \neq \emptyset} \| \phi \cdot \psi \|^p_{L^2_{t,x}(q)} \right)^\frac1p \leq (1+Cc) \left( \sum_{q \in \mathcal{C}_{-2k}: q \cap I^{c,C_0}(Q) \neq \emptyset} \| \phi \cdot \psi \|^p_{L^2_{t,x}(q)} \right)^\frac1p,
\]
which, when combined with the above estimate, leads to 
\begin{equation} \label{ApRi}
A_p(R) \leq (1+cC) A_p(\frac{R}2) + C c^{-C} R^{\frac{n+3}2(\frac1p-\frac{n+1}{n+3})},
\end{equation}
for as long as $C_0 \geq 3$ (this is needed so that $2^{-C_0} 4R \leq \frac{R}2$, which ensures that we can apply the induction hypothesis at the level of $B$'s with $A_p(\frac{R}2)$). 

Then the argument is closed just as in the proof of Theorem 1 in \cite{Be-bil}, see the proof of (2.1) there which is provided right after Proposition 2.2 which essentially states \eqref{ApRi}. $C$ is a universal constant, and one choses $c$ such that $c^{-C}=R^{-\frac{n+3}4(\frac1p-\frac{n+1}{n+3})}$. This choice guarantee that, as long as $p > \frac{n+3}{n+1}$, we obtain $A_p(R) \les A_p(1)$ uniformly in $R \ges 1$. 

Finally to obtain \eqref{LbilfwU2}, which is the upgrade of \eqref{Lbilfw} from free waves to functions in $U^2_\Delta$,
there is a nice and compact argument formalized by Candy in \cite[Theorem 5]{Ca-note}. The basic idea is to randomize the estimate for free waves and use the Khintchine’s inequality to obtain the claim for function in $U^2_\Delta$. The details can be found in the proof of \cite[Theorem 5]{Ca-note}. 

We note that prior to the above argument, strategies for upgrading statements from free waves to functions in $U^2_\Delta$ were developed in \cite{CaHe-t} and \cite{Ca}.

\bibliographystyle{amsplain} \bibliography{RS-refs}

\end{document}